\documentclass[10pt,a4paper,twoside]{article}

\usepackage{a4}
\usepackage{exscale}
\usepackage[]{amsmath,amsthm}
\usepackage{amssymb}
\usepackage{amsfonts}
\usepackage{eucal}
\usepackage{bbm}
\usepackage[colorlinks,citecolor=blue]{hyperref}

\newtheorem{theorem}{Theorem}[section]
\newtheorem{lemma}[theorem]{Lemma}
\newtheorem{definition}[theorem]{Definition}
\newtheorem{proposition}[theorem]{Proposition}
\newtheorem{corollary}[theorem]{Corollary}
\newtheorem{remark}[theorem]{Remark}

\newtheorem{remarks}[theorem]{Remarks}

\newcommand{\R}{\ensuremath{\mathbb{R}}}
\newcommand{\N}{\ensuremath{\mathbb{N}}}

\newcommand{\ccdot}{\,\cdot\,}

\renewcommand{\epsilon}{\varepsilon}

\usepackage[active]{srcltx}
\usepackage[hang,small,bf]{caption}
\usepackage{fancyhdr}

\usepackage[top=3.05cm,bottom=2.6cm,left=2.5cm,right=2.5cm]{geometry}
\newcommand{\comment}[1]{}
\newcommand{\longversion}[1]{}

\newcommand{\bcomment}[1]{}

\DeclareMathOperator{\diverg}{div}
\DeclareMathOperator{\dist}{dist}

\DeclareMathOperator{\supp}{spt}
\DeclareMathOperator*{\esssup}{ess\,sup}

\newcommand{\E}{E}

\newcommand{\Lm}{\mathcal{L}}
\newcommand{\Ln}{\mathcal{L}^n}

\newcommand{\dx}{\, dx}

\newcommand{\dt}{\, dt}
\newcommand{\ds}{\, ds}
\newcommand{\dr}{\, dr}
\newcommand{\dq}{\triangle_{k,h}}
\newcommand{\kabs}[1]{\ensuremath{\vert#1\vert}}
\newcommand{\babs}[1]{\ensuremath{\big\vert#1\big\vert}}
\newcommand{\Babs}[1]{\ensuremath{\Big\vert#1\Big\vert}}
\newcommand{\Bnorm}[1]{\ensuremath{\Big\Vert#1\Big\Vert}}
\newcommand{\bnorm}[1]{\ensuremath{\big\Vert#1\big\Vert}}
\newcommand{\knorm}[1]{\ensuremath{\Vert#1\Vert}}
\renewcommand{\sp}[2]{\ensuremath{\langle \, #1 , #2 \, \rangle}}

\usepackage{titlesec}  
\titleformat{\section}{\bf\Large}{\thesection}{1em}{}
\titleformat{\subsection}{\bf\large}{\thesubsection}{1em}{}

\numberwithin{equation}{section}
\allowdisplaybreaks[3]

\title{\vspace{-0.5cm}
{\bf Random perturbations of nonlinear parabolic systems}
}
\date{\today}
\author{Lisa Beck\footnote{Hausdorff Center for Mathematics, Universit\"at Bonn,
  Germany. Email address: {\tt lisa.beck@hcm.uni-bonn.de}.} 
  \and Franco Flandoli\footnote{Dipartimento di Matematica ``U. Dini'', Universit\`a 
  di Pisa, Italy. Email address: {\tt flandoli@dma.unipi.it}.} }

\begin{document}

%\setpagewiselinenumbers
%\modulolinenumbers[1]
%\linenumbers

\maketitle

\begin{abstract}
Several aspects of regularity theory for parabolic systems are
investigated under the effect of random perturbations. The deterministic
theory, when strict parabolicity is assumed, presents both classes of
systems where all weak solutions are in fact more regular, and examples of
systems with weak solutions which develop singularities in
finite time. Our main result is the extension of a regularity result due 
to Kalita to the stochastic case. Concerning the examples with singular solutions 
(outside the setting of Kalita's regularity result), we do not know whether 
stochastic noise may prevent the emergence of singularities, as it happens for 
easier PDEs. We can only prove that, for a linear stochastic parabolic system with coefficients outside the previous regularity theory, the expected value of
the solution is not singular. \\[1ex]
\end{abstract}

% 60H15 Stochastic partial differential equations
% 35R60 Partial differential equations with randomness, stochastic partial differential equations 
% 60H30 Applications of stochastic analysis (to PDE, etc.) 
% 35B65 Smoothness and regularity of solutions 

\section{Introduction}

Nonlinear parabolic systems of the form
\begin{equation}
\label{system-intro}
\partial_{t}u=\operatorname{div}A( x,t,u,Du)  ,\qquad
u|_{t=0}=u_{0}
\end{equation}
on a cylindrical domain $D \times (0,T)$, with $D \subset \R^n$ a bounded, regular domain, $u \colon D \times [0,T]  \rightarrow \mathbb{R}^{N}$, $A \colon D \times [0,T] \times \mathbb{R}^{N} \times \mathbb{R}^{nN} \rightarrow\mathbb{R}^{nN}$, have been investigated
by many authors, see for instance~\cite{KOSHELEV93,KOSHELEV95,KALITA94} and
references therein. A key feature in the vectorial case $N>1$ is that, 
under the strict parabolicity assumption
\[
\sum_{i,j=1}^n \sum_{\alpha,\beta = 1}^N \frac{\partial A^{\alpha}_{i}}{\partial z^{\beta
}_{j}}( x,t,u,z)  \, \xi^{\alpha}_{i}\xi^{\beta}_{j} \geq \lambda_0 \kabs{\xi}^{2} \qquad \text{ for all } \xi \in \R^{nN}
\]
and some differentiability assumptions with respect to the $(x,u)$-variable, 
there are classes of vector fields $A(x,t,u,z)$ such that all weak
solutions to~\eqref{system-intro} are in fact more regular, and examples of
systems such that there exist weak solutions with singularities;
this dichotomy does not happen for single equations, the case $N=1$, where
regularity of weak solutions is always true, due to the (elliptic and
parabolic) works based on the fundamental results of De Giorgi, Nash and Moser~\cite{DEGIORGI57,NASH58,MOSER60}. 

However, for second-order, parabolic systems under suitable additional assumptions on growth and regularity of the vector field $A(x,t,u,z)$, there are partial regularity results available, yielding H\"older regularity of the solution $u$ (or of its spatial gradient $Du$) outside of a negligible set, the singular set of $u$ (or of $Du$). Hence, for general systems which are nonlinear in the gradient variable, the best regularity to hope for is partial regularity of~$Du$, with an estimate for the Hausdorff dimension of the singular set strictly below the dimension of $\R^n \times [0,T]$, see~\cite{DUMIN05}. For better estimates on the Hausdorff dimension one needs to assume stronger assumptions (or also some a~priori information on the regularity of the solution). Regularity of $u$ on a larger set can be obtained, for instance, in low dimensions or with special structure assumptions (such as vector fields linear in the gradient variable), see e.\,g.~\cite{GIAGIUSTI73,CAMPANATO84,NECSVE91}. Full regularity of $u$ instead is only possible if even more restrictive structural assumptions are imposed. The easiest (and very classical one) of such examples are linear parabolic systems with constant  coefficients. In the nonlinear case, in the famous case of the $p$-Laplacian system it is also possible to prove full regularity of $Du$, see~\cite{DIBENE93}. Furthermore, if the system is still sufficiently close to the Laplacian system, then we still get full regularity of $u$, see~\cite{KOSHELEV93,KALITA94}. This regularity result will be of great importance for our paper.

Let us now go into some details, point out some of the structural prerequisites of the positive (full) regularity theory and confront it with the existing examples of systems admitting a singular weak solutions. For simplicity we focus here in the introduction on the case of quasilinear problems with a vector field of the
form $A(x,t) z$, i.\,e. to weak solutions of
\begin{equation}
\label{system-intro-linear}
\partial_{t}u=\operatorname{div} \big( A (x,t) \, Du \big) \,, \qquad
u|_{t=0}=u_{0} \,.
\end{equation}

As mentioned above, without additional structural conditions on the
coefficients full regularity of the solutions can no longer be expected in the
vectorial case. It was observed by Koshelev and Kalita~\cite{KOSHELEV93,KALITA94} that if the coupling of the single equations is sufficiently weak, then discontinuities of the weak solution can globally be excluded:

\begin{theorem}[\cite{KALITA94}]
\label{thm_Kalita}
Let $u_0 \in W^{1,q}(D,\R^N)$ for some $q>n$ and consider coefficients $A(x,t)$ which are of class $C^1$ in $x$, measurable in $t$ and which satisfy 
\begin{equation}
\label{ass-kalita-intro}
\lambda_{0} \kabs{ \xi }^{2} \leq \sp{ A( x,t)
\, \xi}{\xi} \,, \quad \kabs{A(x,t) \, \xi} \leq \lambda_{1} \kabs{ \xi } \,, \quad
\text{and} \quad \kabs{D_x A(x,t)} \leq L
\end{equation}
for all $\xi \in \R^{nN}$, $(x,t,z) \in D \times [0,T] \times \R^{nN}$ and some positive constants $\lambda_0, \lambda_1, L$. If in addition $\frac{\lambda_0}{\lambda_1} > 1 -
\frac{2}{n}$ holds, then every weak solution $u \colon D \times [0,T] \to \R^N$ to the initial boundary value problem~\eqref{system-intro-linear} is
of class $C^{0,\alpha}_{\rm{loc}}(D \times [0,T],\R^N)$ for some $\alpha > 0$.
\end{theorem}

It is important to mention that the original results apply to more general systems, possibly nonlinear in the gradient variable, provided that the vector field $A(x,t,u,z)$ is sufficiently close to a quasilinear situation with a small dispersion ratio. First, Koshelev proved the existence of a regular solution (which in the situation above is the unique one) by studying an approximation of the system such that its solutions are regular and converge in a suitable norm to a solution of the original system. Later Kalita achieved the regularity result for all solutions with a direct argument (and not as a consequence of a suitable approximating sequence), which essentially relies on Moser's iterative method~\cite{MOSER60}.

Under weaker assumptions than in the previous theorem, such a global regularity
result can no longer be expected. In fact, in a very similar setting the
following example of a system was proposed by Stara and John~\cite{STAJOH95} (actually, the example was constructed on the full space and the solution can be traced back in time $t \to -\infty$), which admits a solutions that starts from a regular -- in particular H\"older continuous -- initial data and develops a
singularity in finite time in the interior of the parabolic cylinder.

\begin{theorem}[\cite{STAJOH95}]
Let $n=N \geq 3$. There exist initial data $u_0 \in W^{1,2n}(B_1(0),\R^n)$ 
and measureable, symmetric coefficients 
$A \in L^{\infty}(B_1(0) \times [0,1),R^{n^2 \times n^2})$,
which are elliptic and bounded in the sense of~\eqref{ass-kalita-intro}$_{1,2}$ for
all $(x,t) \in B_1(0) \times [0,1)$, such that at least one of the solutions to the
initial problem~\eqref{system-intro-linear} develops a discontinuity in the
origin $x=0$ as $t \nearrow 1$.
\end{theorem}

The coefficients constructed in~\cite{STAJOH95} have a dispersion ratio
$\frac{\lambda_0}{\lambda_1} < 1 - \frac{2}{n}$ below the critical one
investigated in~\cite{KALITA94} as well as a lower regularity with respect to
the $x$-variable. For this reason we cannot exclude that the solution develops
the discontinuity due to an interaction with the irregular coefficients.
However, a comparison with the positive regularity results in the elliptic
theory (the stationary case) suggests that the dispersion ratio
$\frac{\lambda_0}{\lambda_1}$ plays an important role. Indeed, in the elliptic
case no regularity in the $x$-variable is required, and a modification of De
Giorgi's counterexample to full regularity shows sharpness of the (elliptic)
condition on the dispersion ratio (see~\cite[Section~2.5]{KOSHELEV95}).
Unfortunately, we didn't find further counterexamples in the literature which
could give a similar complete picture in the case of parabolic systems.

The aim of this paper is to investigate parts of this theory under the effect
of random perturbations. The final aim of our research project, in analogy with
recent results proved for other equations, is to show that the regularity theory of
parabolic systems is, under random perturbations, in some sense (of course only up to a certain degree of regularity) not worse than the deterministic one, and possibly better. 
As in the deterministic case there is more than one approach to the analysis of 
these problems, so we restrict the attention here only to a few directions. 
More precisely, we want to show two results.

First, we study systems with It\^o noise of the form 
\begin{equation}
\label{system-intro-linear-stoch}
du  = \diverg \big( A (x,t) \, Du \big) \dt + H(Du) \, dB_t  \,, \qquad
u|_{t=0}=u_{0}
\end{equation}
(for $H$ Lipschitz) where $(B_{t}) _{t\geq0}$ is a Brownian motion of suitable dimension, and we generalize to the stochastic case one of the regularity results of the deterministic theory, the work of Kalita~\cite{KALITA94} (which was displayed
also above for the special case of a quasilinear system). The passage from deterministic to stochastic of Kalita's approach contains at least one non trivial detail which is rather new in the stochastic setting: the weak solution $u$ we start with is not, a priori, the limit of a sequence of smooth solutions of approximating equations (for instance, due to the nonlinearity, classical mollifiers are difficult to implement; in another direction, in some cases solutions exist as limits of Galerkin or other types of approximations, but we here start with a weak solution which a priori has not been constructed in that way). Therefore, it is not clear how to perform
differential calculus on $u$; and Kalita's approach is heuristically based on
an equation satisfied by second space derivatives of $u$. Therefore one has
to use finite difference quotients in place of derivatives, a classical method in the
deterministic setting, but not common in the stochastic case. This leads to a
number of technical novelties. At the end, we reach a full extension of
Kalita result to a quite general stochastic case, which includes in particular 
perturbations in form of additive or of multiplicative noise. In the quasi-linear 
model case we obtain -- as a particular case of Theorem~\ref{main-result} -- the following result (the precise definition of weak solution is given in Definition~\ref{def_weak_solution_ito} below).

\begin{theorem}
\label{thm-regularity-model}
Let $u_0 \in W^{1,q}(D,\R^N)$ for some $q>n$. Consider coefficients $A(x,t)$ which are of class $C^1$ in $x$, measurable in $t$ and which satisfy~\eqref{ass-kalita-intro} with 
$\frac{\lambda_0}{\lambda_1} > 1 - \frac{2}{n}$, and assume that $H$ is Lipschitz 
continuous with Lipschitz constant $L_H < L_H^*$ for some sufficiently small $L_H^*$ depending only on $n, \lambda_0$ and $\lambda_1$. 
Then there exists $\alpha > 0$ depending only on $n,\lambda_0, \lambda_1$ and $q$ 
such that every weak solution $u \colon D \times [0,T] \times \Omega \to \R^N$ to the initial boundary value problem~\eqref{system-intro-linear-stoch} is of class $C^{0,\alpha}_{\rm{loc}}(D \times [0,T],\R^N)$ with probability $1$.
\end{theorem}

Second, we investigate for these systems the problem recently considered for
other classes of PDEs, see~\cite{FLALNM}, namely the possibility that it is
precisely the noise which prevents the emergence of singularities. The aim of this
project, that we have reached only partially until now, would be to prove
that, under assumptions on the vector field $A(x,t,u,z)$ such that
there exist weak solutions with singularities in the deterministic case, there
are no more singularities if we add a suitable noise. We can only prove an
intermediate but promising result. We consider linear stochastic systems with
Stratonovich bilinear multiplicative noise of the form

\begin{equation}
\label{system-intro-Strat}
du=\operatorname{div} \big( A( x,t)  Du \big) \dt + \sigma Du\circ dB_t
,\qquad  u|_{t=0}=u_{0}
\end{equation}
with regular, bounded and elliptic measurable coefficients $A$. In this
situation we obtain regularity for the mean value.

\begin{proposition}
Given coefficients $A(x,t)$ which are of class $C^1$ in $x$, measurable in $t$, and which satisfy~\eqref{ass-kalita-intro}, there exists $\sigma
_{0}\geq0$ depending only on $\lambda_0, \lambda_1$ such that for all 
$\sigma>\sigma_{0}$, all initial conditions $u_0 \in W^{1,q}(D,\R^N)$ with $q>n$, 
and all weak solutions $u$ of equation~\eqref{system-intro-Strat}
we have that the function $(x,t) \longmapsto E[u(x,t)]$ is locally 
H\"{o}lder continuous on $D \times [0,T]$.
\end{proposition}

This result is proved by applying the deterministic results. 
The key observation in the proof is that the average solves an equation with a better (that is greater) dispersion ratio. So far this class does not cover the counter-examples in~\cite{STAJOH95} mentioned above (because here we need more regularity of $A$
than in the counter-example), but however the theory presented here might be of 
its own interest. We cannot take $\sigma_{0}=0$ but we suspect that this is the 
critical value (namely that for all noise intensities the regularization takes place). 
The fact that this result holds independently of the initial condition $u_{0}$ -- though sufficiently regular -- and of the specific form of $A(x,t)$ in the functional class we consider (in particular the fact that no symmetry is embodied in our assumptions which
could justify compensations due to the expected value), leads us to think that
in fact $u(x,t)$ itself is H\"{o}lder continuous, but we do not
have a proof of this conjecture.

Several problems remain open: 
\begin{enumerate}
  \item[(i)] whether counter-examples exist also in the stochastic case under some conditions on $A$; 
  \item[(ii)] the regularity of $u(x,t)$ itself in the regularization-by-noise above and other related issues, such as the value of $\sigma_{0}$ and a generalization to the nonlinear case;
  \item[(iii)] the generalization of other deterministic approaches to regularity.
\end{enumerate}

Concerning the existence of weak solutions, we could give a quite general
result, but since it is related to the generalization of the approach of
\cite{KOSHELEV93,KOSHELEV95} to regularity, we postpone it to a future work.

\section{Setting and assumptions}

Consider $n,n' \in \N$ with $n \geq 2$, $T>0$, and $D \subset \R^n$ a (regular) bounded
domain. Let $(\Omega,\mathcal{F},P)$ be a complete probability space with a filtration 
$(\mathcal{F}_t)_{t \ge 0}$, and let $(B_t)_{t \ge 0}$ be a standard $n'$-dimensional 
Brownian motion. Let further $A \colon D \times [0,T] \times \R^N \times \R^{nN}
\times \Omega \to \R^{nN}$ be a vector field satisfying the following properties:
\begin{itemize}
\item $A$ is progressively measurable, i.\,e. for every $t \in [0,T]$ the
restriction of $A$ to $D \times [0,t] \times \R^N \times \R^{nN} \times \Omega
\to \R^{nN}$ is $\mathcal{B}(D) \times \mathcal{B}([0,t]) \times
\mathcal{B}(\R^N) \times \mathcal{B}(\R^{nN}) \times \mathcal{F}_t$ measurable;
\item $A(x,t,u,z,\omega)$ (usually abbreviated by $A(x,t,u,z)$) is
differentiable in $x,u$ and $z$ (with $\mathcal{F}_t$-adapted derivatives), and
it satisfies for $P$-almost all $\omega \in \Omega$ the following assumptions
concerning growth and ellipticity:
\begin{equation}
\label{GV-A}
\left\{ \quad \begin{array}{l}
\kabs{A(x,t,u,z)} \leq L \,  \big( \kabs{z} + \kabs{u}^{\frac{n+2}{n}}
+ f^{\frac{a}{2}}(x,t) \big) \\[0.1cm]
\kabs{\xi - \kappa \, D_z A(x,t,u,z) \, \xi}^2 \leq (1 - \nu^2) \, 
\kabs{\xi}^2 \\[0.1cm]
\kabs{D_u  A(x,t,u,z)}  \leq  L \, \big( \kabs{z}^{\frac{2}{n+2}} +
\kabs{u}^{\frac{2}{n}} + f(x,t) \big) \\[0.1cm]
\kabs{D_x  A(x,t,u,z)}  \leq  L \, \big( \kabs{z} + \kabs{u}^{\frac{n+2}{n}}
+ f^2(x,t) \big) \\
\end{array} \right.
\end{equation}
for all $(x,t) \in D  \times [0,T]$, $u \in \R^N$ and $z,\xi \in \R^{nN}$, some constants
$\kappa, \nu, L > 0$, and an $\mathcal{F}_t$-adapted process $f$ which with probability one belongs to $L^a(D \times [0,T])$ for a fixed number $a > n+2$.
\end{itemize}

Moreover, let $H:D \times [0,T] \times\mathbb{R}^{nN} \times \Omega
\rightarrow \mathbb{R}^{n'N}$ be progressively measurable, 
of class $C^1$ in $x$, Lipschitz with respect to the gradient variable of at
most linear growth, 
uniformly in $(x,t)$, i.e.
\begin{equation}
\label{GV-H}
\left\{ \quad \begin{array}{l}
\kabs{ H(x,t,z,\omega )  -H(x,t,\widetilde{z},\omega) } \leq 
L_H \left\vert z-\widetilde
{z}\right\vert \,, \\[0.1cm]
\kabs{ H(  x,t,z,\omega) } \leq L \big(
f_H(x,t,\omega) +\kabs{ z} \big) \,, \\
\kabs{ D_{x} H(  x,t,z,\omega)  } \leq L \big(
f_H^{\frac{a}{a-2}}(x,t,\omega) +\kabs{ z} \big)
\end{array} \right.
\end{equation}
for a constant $L_H$, all $(x,t) \in D  \times [0,T]$, $z,\widetilde{z} \in \mathbb{R}^{nN}$, and almost every $\omega \in \Omega$. Here, $f_H$ denotes a function in $L^a(D \times (0,T) \times \Omega)$.

Under these assumptions we consider a stochastic partial differential equation
with noise of the form
\begin{equation}
\label{equation}
du  = \diverg A (x,t,u,Du) \dt + H(x,t,Du) \, dB_t \qquad \text{ in
} D_T := D \times (0,T) \,,
\end{equation}
where $u \colon D_T \times \Omega \to \R^N$ is a random function. The stochastic 
integral is here understood in the It\^o sense. According to
the growth condition on the vector field $A$, we note that for $P$-almost every
$\omega \in \Omega$ and all $t \in [0,T]$ we have $\diverg A (x,t,v,Dv) \in
W^{-1,2}(D,\R^N)$ -- the dual space to $W^{1,2}_0(D,\R^N)$ --, provided that $v \in
W^{1,2}(D,\R^N)$.

\begin{remark}
We have chosen this level of generality of the noise for two reasons: to keep
a simple PDE structure instead of an abstract operator formulation, and to
cover two interesting examples: additive noise (with $H(x,z)$
independent of $z$) and bilinear multiplicative noise with first derivatives
of $u$ (with $H(x,z)$ linear in $z$). A~priori there is no
conceptual obstacle to consider $H$ depending also on $u$ and also 
to generalize to the case of a Brownian motion $B$ in a
Hilbert space $U$, with suitable assumptions on $H$, but for simplicity we
restrict ourselves to the previous case.
\end{remark}

The function spaces that will be needed in the sequel are the Banach spaces
\begin{align*}
V^{m,p}(D_T,\R^N) & :=  L^{\infty}\big(0,T; L^m(D,\R^N)\big) \cap
L^p\big(0,T;W^{1,p}(D,\R^N)\big) \\
V^{m,p}_0(D_T,\R^N) & := L^{\infty}\big(0,T; L^m(D,\R^N)\big) \cap
L^p\big(0,T;W^{1,p}_0(D,\R^N)\big) \,, 
\end{align*}
with $m,p \geq 1$, and they are equipped with the norm 
\begin{equation*}
\knorm{u}_{V^{m,p}(D_T,\R^N)} := \esssup_{t \in (0,T)}
\knorm{u(t)}_{L^m(D,\R^N)}
	+ \knorm{Du}_{L^p(D_T,\R^N)} \,.
\end{equation*}
When $m=p$ we shall use the abbreviations $V^{p}_{(0)}(D_T,\R^N) =
V^{p,p}_{(0)}(D_T,\R^N)$. We remind that the spaces $V^{m,p}(D_T,\R^N)$ are
embedding in the Lebesgue space $L^{q}(D_T,\R^N)$ with $q = p \frac{n+m}{n} > p$
(see~\cite[Propositions~I.3.1,~I.3.2]{DIBENE93}). We will need only the result
concerning the cases $m=2$, $p\geq2$ or $m=p\geq 2$. In the latter case, the embedding reads as follows (see \cite[Propositions~I.3.3,~I.3.4]{DIBENE93}): let $v \in V^{p}_0(D_T,\R^N)$, $p<n$. Then there exists a constant $c$ depending only on $n$ and $p$ such that 
\begin{equation}
\label{embedding}
\knorm{v}_{L^q(D_T,\R^N)} \leq c \, \knorm{v}_{V^{p}(D_T,\R^N)} 
\end{equation}
(and an analogous result holds without any restriction on the boundary values of
$v$ on $\partial D \times (0,T)$ if $\partial D$ is assumed to be sufficiently
regular).

We are now going to study the properties of weak (or variational) solutions to the
system~\eqref{equation}, which is to be understood in the following sense.

\begin{definition}
\label{def_weak_solution_ito}
An $\mathcal{F}_t$-progressively measurable process $u$ on $[0,T] \times \Omega$
is called a weak solution to the system~\eqref{equation} with initial values
$u_0 \in L^2(D,\R^N)$ if $P$-a.\,e. path satisfies $u(\cdot,\omega) \in
V^{2}(D_T,\R^N)$ and if for all $t \in [0,T]$, we have $P$-a.\,s. the identity
\begin{equation*}
\sp{u(t) - u_0}{\varphi}_{L^2(D)}
	 =  \int_0^t \sp{\diverg
A(\cdot,s,u,Du)}{\varphi}_{W^{-1,2}(D);W^{1,2}_0(D)} \ds
	+ \int_0^t \sp{\varphi}{H(\cdot,s,Du) \, dB_s}_{L^2(D)} 
\end{equation*}
for all $\varphi \in W^{1,2}_0(D,\R^N)$. 
\end{definition}

When a solution is progressively measurable with respect to the (completed)
filtration associated to the Brownian motion, it is usually called a
``strong'' solution in the probabilistic
sense, see~\cite[Section~IX.1]{REVYOR94}. We do not require this condition, so
our result will also apply to the so called ``weak'' solutions in the 
probabilistic sense (those for which there is a filtration 
$(\mathcal{F}_{t})_{t\geq0}$ such that $u$ is $\mathcal{F}_{t}$-progressively 
measurable and $B$ is an $\mathcal{F}_{t}$-Brownian motion). We further note that according to the definition above, a solution is defined as an equivalence class in the sense of versions (a process $Y$ is a version or modification of a process $X$ if for each time $t$ we have $P$-a.\,s. $X_t=Y_t$). Hence, regularity of a weak solution is always to be understood as finding a regular representative in the corresponding equivalence class.

Moreover, we comment on the way in which the initial values are attained. Under mild assumptions on the growth of $A$ and $H$ with respect to the gradient variable one actually deduces from the equation itself that $u$ belongs to $C^0(0,T;L^2(D',\R^N))$ $P$-a.\,s. for every $D' \Subset D$, compare formula~\eqref{apriori1} and the beginning of Step~3 on p.~\pageref{continuous_in_t}. Under further assumptions on the trace of $u$ on $\partial D \times [0,T]$ this extends to continuity of the full $L^2$-norm, with $D' = D$. In this sense the term ``initial value'' in the definition of a weak solution as a function in the space $V^2$ is justified.

%*****************************************************************************
%		Preliminaries
%*****************************************************************************

\section{Preliminaries}

In this section we recall some well-known facts and provide some technical tools. For convenience of the reader we state two suitable versions of It\^o's formula. Furthermore, in analogy with the deterministic theory, we discuss a sufficient condition for the ``existence of weak derivatives with probability one'', and we further give a criterion which guarantees pathwise H\"older continuity.  

\subsection{It\^o formula}

We first recall two versions of It\^o's formula, the first one the standard version
for $N$-dimensional processes and the second one for processes with values in
Hilbert spaces. Consider $(\Omega,F,P)$ a complete probability space and let  
\begin{equation}
\label{Ito_process}
dX(t) = a(t) \dt + b(t) \, dB_t
\end{equation}
be an $N$-dimensional It\^o process which satisfies: $a,b$ are
$\mathcal{F}_t$-adapted (i.\,e., the maps $\omega \mapsto a(t,\omega),
b(t,\omega)$ are $\mathcal{F}_t$ measurable), $(t,\omega) \mapsto b(t,\omega)$
is $\mathcal{B}([0,T]) \times \mathcal{F}$-measurable and 
\begin{equation*}
P \Big( \int_0^T \big[ \kabs{a(s,\omega)} + \kabs{b(s,\omega)}^2 \big] \ds < \infty
\Big) = 1.
\end{equation*}
Then the following general It\^o formula holds (see e.\,g.~\cite[Theorem~4.2.1]{OKSENDAL}).

\begin{theorem}[It\^o's formula I]
\label{Ito-I}
Let $g(t,z)=(g_1(t,z),\ldots,g_p(t,z))$ be a map from $[0,T] \times \R^N$ to
$\R^p$ of class $C^1$ in $t$ and of class $C^2$ in $z$. Then the process
$Y(t,\omega) := g(t,X(t))$ with $X(t)$ defined in~\eqref{Ito_process} is again
an It\^o process whose components are given by
\begin{equation*}
dY_k(t) = \frac{\partial g_k}{\partial t}(t,X) \dt + \sum_{i=1}^N
\frac{\partial g_k}{\partial y_i}(t,X) \, dX_i 
	+ \frac12 \sum_{i,j=1}^N \frac{\partial^2 g_k}{\partial y_i y_j}(t,X)
\, dX_i \, dX_j
\end{equation*}
for all $k \in \{1,\ldots,p\}$, and with $dB_i \, dB_j = \delta_{ij} \dt$ and $dB_i
\dt = 0 = dt \, dB_i$ for all $i,j \in \{1,\ldots,N\}$.
\end{theorem}

In the sequel, we will also employ the following version of the It\^o formula in
Hilbert spaces that can be found in~\cite[Theorem~3.1]{KRYROZ79} or~\cite[Chapter~4.2, Theorem~2]{ROZ90}.

\begin{theorem}[It\^o's formula II]
\label{Ito-II}
Let $V \subset H \subset V'$ be a Gelfand triple, with $H$ a separable Hilbert
space. Assume that we have for $\Lm^1 \times P$ almost all $(t,\omega)$
\begin{equation}
\label{Ito_assumption}
\sp{x(t)}{\varphi}_{H} = \sp{x(0)}{\varphi}_{H} + \int_0^t
\sp{y(s)}{\varphi}_{V',V} \ds + \sp{M_t}{\varphi}_{H}
\end{equation}
for every $\varphi \in V$ where $x(t,\omega), y(t,\omega)$ are taking values in
$V$ and $V'$, respectively, and are progressively measurable with 
\begin{equation*}
P \Big( \int_0^T \big[ \knorm{x(s,\omega)}^2_V + \knorm{y(s,\omega)}^2_{V'} 
\big] \ds   <  \infty \Big)  =  1,
\end{equation*}
and where $M_t$ is a continuous local martingale with values in $H$. Then there exists a set $\widetilde{\Omega} \subset \Omega$ with $P(\widetilde{\Omega})=1$ and a map $\widetilde{x}(t,\omega)$ with values in $H$
such that:
\begin{itemize}
\item[(i)] $\widetilde{x}(t)$ is $\mathcal{F}_t$-adapted, continuous in $t \in
[0,T]$ for every $\omega \in \widetilde{\Omega}$, and $x(t) = \widetilde{x}(t)$
$P$-almost surely;
\item[(ii)] for every $\omega \in \widetilde{\Omega}$ and every $\varphi \in V$
there holds
\begin{equation*}
\sp{\widetilde{x}(t)}{\varphi}_{H}  =  \sp{x(0)}{\varphi}_{H} + \int_0^t
\sp{y(s)}{\varphi}_{V',V} \ds + \sp{M_t}{\varphi}_{H} \,;
\end{equation*}
\item[(iii)]  for every $\omega \in \widetilde{\Omega}$ there holds the
inequality
\begin{equation*}
\knorm{\tilde{x}(t)}_{H}^2  =  \knorm{x(0)}_{H}^2 + 2 \int_0^t
\sp{y(s)}{x(s)}_{V',V} \ds 
	+ 2 \int_0^t \sp{dM_s}{\tilde{x}(s)}_H + [M]_t \,.
\end{equation*}
\end{itemize}
\end{theorem}

\subsection{Weak derivatives}
\label{sect-weak-der}

For a vector-valued function $f \colon \R^n \supset D \to \R^{N}$, $k \in
\{1,\ldots,n\}$ and a real number $h \in \R$ we denote by  $\dq f(x) := h^{-1}
(f(x + he_k) - f(x))$ the finite different quotient in direction $e_k$ and
stepsize $h$ (this makes sense as long as $x$, $x+he_k \in D$). Let $p>1$, $f\in
L^{p}\left(  D\right)$, $k\in\left\{  1,...,n\right\}$ and let
$D_k f$ be the derivative of $f$ in the direction
$k$ in the sense of distributions. Just for comparison, let us recall the
following lemma (not used below).

\begin{lemma}
If there is $h_{n}\rightarrow0$ and $g_{k}\in L^{p}(D)  $ such
that
\[
\lim_{n\rightarrow\infty}\int_{D}\left(  \triangle_{k,h_n}f(x) -g_{k}\left( 
x\right)  \right)  \varphi(x)  \dx=0
\]
for every $\varphi\in C_{0}^{\infty}(D)  $, then $D_k f$ is in $L^{p}(D)$ and is equal to $g_{k}$.
\end{lemma}

As an immediate consequence of this lemma and of the compactness of the $L^p$-spaces with $p > 1$ with respect to weak (or weak-$*$) convergence, we obtain a simple criterion for the existence of the weak derivative $D_k f$ in $L^p$, namely it is sufficient that $\knorm{\dq f}_{L^p(D')}$ is bounded by some constant, for all $h$ and every $D' \Subset D$ such that $\kabs{h} < \dist(D',\partial D)$.

Now this well-known principle shall be carried over to a probabilistic setting. Let $(\Omega,F,P)$ be a complete probability space and consider a
function $f$ in the Banach space $L^{p}(D\times\Omega)$. A function $g_{k} \in L^{p}(D\times\Omega)$ is said to be the weak derivative of $f$ in the $k$-direction if
\[
P \Big(  \int_{D}f \, D_k \varphi \dx=-\int_{D}
g_{k} \varphi \dx \Big)  =1
\]
for every $\varphi\in C_{0}^{\infty}(D)$ (taking a countable
sequence and using a density argument, the property ``for every $\varphi\in
C_{0}^{\infty}(D)$'' can be written inside the probability). We then write $D_k f = g_k $.
The previous lemma has a generalization to functions in $L^{p}(
D\times\Omega)$.

\begin{lemma}
If there is $h_{n}\rightarrow0$ and $g_{k}\in L^{p}(D\times
\Omega)$ such that
\[
\lim_{n\rightarrow\infty}\int\int_{D\times\Omega}\big( 
\triangle_{k,h_n}f(x,\omega) -g_{k}(
x,\omega)  \big)  \varphi(x)  X(\omega)
\dx \, dP(\omega)  =0
\]
for every $\varphi\in C_{0}^{\infty}(D)$ and every bounded
measurable $X:\Omega\rightarrow\mathbb{R}$, then $D_k f$ is in $L^{p}(  D\times\Omega)$ and is equal to $g_{k}$.
\end{lemma}

\begin{proof}
Since $X$ and $\varphi$ are bounded, we may apply (first Fubini and then)
Lebesgue's dominated convergence theorem, and we get
\begin{align*}
- E \Big[  X\int_{D} \big(  f \, D_k \varphi +g_{k}\varphi \big)  \dx \Big]
& = - \int \int_{D\times\Omega} X \big(  f \, D_k \varphi
+g_{k} \varphi \big)  \dx \, dP\\
& =\lim_{n\rightarrow\infty}\int\int_{D\times\Omega}X \big( - f(
x) \, \triangle_{k,-h_n}\varphi(x) - g_{k}(x) \, \varphi(x) \big)  \dx \, dP.
\end{align*}
When $h_{n}< \dist(  \supp \varphi,\partial D)  $, this is equal to
(we apply Fubini twice and a change of variables)
\[
\lim_{n\rightarrow\infty}\int\int_{D\times\Omega}X \big( 
\triangle_{k,h_n}f(x,\omega) \varphi(
x)  -g_{k}(x,\omega) \, \varphi(x) \big)  \dx \, dP.
\]
This limit is zero by assumption, hence%
\[
E \Big[  X\int_{D} \big(  f \, D_k \varphi +g_{k} 
\varphi \big)  \dx \Big]  =0.
\]
The arbitrariness of $X$ implies $\int_{D} \big(  f D_k\varphi +g_{k}\varphi \big)  \dx=0$, as a random variable on $\Omega
$. The proof is complete. 
\end{proof}

\begin{corollary}
If there is a constant $C>0$ such that
\[
E \Big[  \int_{D'}\left|  \triangle_{k,h}f(x) \right|^{p} \dx\Big]  \leq C
\]
for all $h$ and all $D' \Subset D$ such that $\kabs{h} < \dist(D',\partial D)$, then $D_k f$ is in $L^{p}(D \times \Omega)$. \comment{ and there are sequences $h_{n}\rightarrow0$ such that
\[
\lim_{n\rightarrow\infty}\int\int_{D\times\Omega} \Big( 
\triangle_{k,h_n}f(x,\omega) - \frac{\partial
f}{\partial x_{k}} \Big)  \varphi(  x)  X(  \omega)
\dx \, dP(\omega)  =0
\]
for every $\varphi\in C_{0}^{\infty}(  D)  $ and every bounded
measurable $X:\Omega\rightarrow\mathbb{R}$.}
\end{corollary}

\begin{proof}
The family $g_{k,h}(x,\omega)  :=\triangle_{k,h}f(x,\omega) $ is
equibounded in $L^{p}(D' \times \Omega)$, hence there is a sequence
$h_{n}\rightarrow0$ such that $g_{k,h_{n}}$ converges weakly in $L^{p}\left(
D\times\Omega\right)  $ to some function $g_{k}\in L^{p}\left(  D\times
\Omega\right)  $. The product $\varphi(  x)  X(
\omega)  $ is in $L^{p^{\prime}}(D\times\Omega)$
($p^{\prime}$ conjugate to $p$) for every $\varphi\in C_{0}^{\infty}(
D)$ and every bounded measurable $X:\Omega\rightarrow\mathbb{R}$.
Hence, we may apply the lemma and obtain the assertion. 
\end{proof}

First, for our later application, we replace $D$ by $D \times [
0,T]$ and we allow different integrability exponents with
respect to the variables in $[0,T]$ and $D$, respectively. 
Let $f:\Omega\rightarrow L^q(0,T; L^{p}(D))$ be a
measurable function with $p \in (1,\infty)$ and $q > 1$. We say that a function
$g_{k}:\Omega\rightarrow  L^q(0,T; L^{p}(D))$ is weak derivative of $f$ in the $k$-direction with probability one if for a.\,e. $(t,\omega)  \in [0,T] \times \Omega$ we have
\[
\int_{D} f \, D_k \varphi \dx=-\int_{D}g_{k} \varphi \dx
\]
for every $\varphi\in C_{0}^{\infty}(D)$, and we then write $D_k f = g_k$. Furthermore, let us generalize to a
scheme where we relax the integrability in $\Omega$.

\begin{theorem}
\label{thm-derivatives}
Let $Y \colon [0,T] \times \Omega \to (0,1]$ be a positive random variable, with $P( \inf_{t \in [0,T]} Y > 0 )  =1$. If
there is a constant $C>0$ such that
\[
E\Big[  \bnorm{Y(t) \, \triangle_{k,h_n}f(x,t) }_{L^q(0,T; L^{p}(
D'))}^{p} \Big]  \leq C
\]
for all $h$ and $D' \Subset D$ satisfying $\kabs{h} < \dist(D',\partial D)$, 
then $D_k f \in L^q(0,T; L^{p}(D))$ with 
probability one and there hold
\begin{align*}
& Y \dq f \to Y D_kf \quad \text{weakly in } L^{p}(\Omega; L^q(0,T; L^{p}(D))),\\
& E \big[ \bnorm{Y D_k f}_{L^q(0,T; L^{p}(
D))}^{p} \big] \leq C
\end{align*}
with the same constant $C$. \comment{Furthermore, there are sequences $h_{n}%
\rightarrow0$ such that
\[
\lim_{n\rightarrow\infty}\int\int_{D\times\Omega} \Big( 
\triangle_{k,h_n}f(x,\omega) -\frac{\partial
f}{\partial x_{k}} \Big)  \varphi(x)  X(\omega)
\dx \, dP(\omega)  =0
\]
for every $\varphi\in C_{0}^{\infty}(D)  $ and every bounded
measurable $X:\Omega\rightarrow\mathbb{R}$.}

\begin{proof}
The family $Z_{k,h}(x,t,\omega):=Y(t) \triangle_{k,h_n}f(x,t,\omega)$ is equibounded in
$L^{p}(\Omega; L^q(0,T; L^{p}(D')))  $, hence there
is a sequence $h_{n}\rightarrow0$ such that $Z_{k,h_{n}}$ converges weakly in
$L^{p}(\Omega; L^q(0,T; L^{p}(D)))$ (or weakly-$\ast$
if $q = \infty$) to some
function $Z_{k}\in L^{p}(\Omega; L^q(0,T; L^{p}(D)))$. 
This implies (again with $\psi$, $X$ bounded, measurable and $\varphi$ smooth, compactly supported)
\[
\lim_{n\rightarrow\infty}\int\int\int_{
D \times [0,T] \times \Omega} \big(  Y(t)
\triangle_{k,h_n}f(x,t)-Z_{k}(x,t) \big) \, \varphi(x) \, \psi(
t) \, X \dx \dt \, dP=0 \, .
\]
Hence, by Fubini and change of variables as above, we find
\[
\lim_{n\rightarrow\infty}\int\int\int_{
D \times [0,T]  \times \Omega} \big(  Y(t) f( x,t) 
\triangle_{k,-h_n} \varphi(x) +Z_{k}(x,t) \, \varphi(x) \big)
\psi(t)  X \dx \dt \, dP=0 \,,
\]
which in turn implies by Lebesgue's theorem
\[
\int\int\int_{D \times [0,T]  \times  \Omega} \big(  Y(t) f(
x,t) \, D_k \varphi(x) +Z_{k}(x,t) \, \varphi(x) \big)
\psi(t)  X \dx \dt \,dP=0 \,.
\]
Arbitrariness of $X$ and $\psi$ thus yields
\[
\int_{D} \big(  Y(t) f(x,t) \, D_k \varphi(x) +Z_{k}(x,t) \, \varphi(x) \big)  \dx=0
\]
for a.\,e. $(t,\omega) \in [0,T] \times \Omega$. Therefore, we have
\begin{equation}
\int_{D} \big(  f(x,t) \, D_k \varphi(x) +g_{k}(x,t) \, \varphi(x) \big)  \dx=0 \label{distr}
\end{equation}
for a.\,e. $(t,\omega)  \in [0,T] \times \Omega$,
where $g_{k}=Y^{-1}Z_{k}$. Since $Z_{k}$ belongs to $L^{p}( \Omega; L^q(0,T; L^{p}( D)))$, 
it is $L^q(0,T; L^{p}( D))$ for $P$-a.\,e. $\omega\in\Omega$. Hence, by assumption on $Y$,  we also have $g_{k}\in L^q(0,T; L^{p}(D))$ for $P$-a.\,e. $\omega\in\Omega$. The only difference with the definition of $g_k$ being the ``weak derivative of $f$ in the $k$-direction with probability one'' is that the negligible set of $(t,\omega)  \in [0,T] \times \Omega$ where~\eqref{distr} may fail depends on $\varphi\in C_{0}^{\infty}(D)$, until now. But $W^{1,p^{\prime}}(D)$ ($p^{\prime}$ is conjugate to $p$) is separable and $C_{0}^{\infty}(D)$ is dense in it. Hence, there is a countable family $\{  \varphi_{n}\}
\subset C_{0}^{\infty}(D)$ which is dense in $W^{1,p^{\prime}}(D)$. If we call $N$ the countable union of all negligible sets of $(t,\omega)  \in [0,T] \times \Omega$ where~\eqref{distr} may fail for $\{ \varphi_{n} \}$, $N$ is negligible,
and on the complementary we have~\eqref{distr} for every $\varphi_{n}$, hence
by density for all $\varphi\in W^{1,p^{\prime}}(D)$ and then
for all $\varphi\in C_{0}^{\infty}(D)$. Having identified $g_k$
as the weak derivative of $f$ in the $k$-direction we take advantage of the lower 
semi-continuity of the norm with respect to weak (or weak-$\ast$) convergence
and thus find 
\[ E \big[ \knorm{Y D_k f}_{L^q(0,T; L^{p}(
D))}^{p} \big] = E \big[ \knorm{Z_k}_{L^q(0,T; L^{p}( D))} \big] \leq C.
\]
The proof is complete. 
\end{proof}
\end{theorem}

\begin{remark}
\label{rem-diff-quot}
This result will be applied later in the cases $p = q$ where the assumption then
reads as
\[
E \Big[  \int_{0}^{T}\int_{D} \kabs{ Y(t) \dq f_1(x,t) }^{p} \dx \dt \Big] 
\leq C
\]
and where we have $L^{p}(\Omega; L^q(0,T; L^{p}(D))) =
L^{p}(D \times [0,T] \times \Omega)$, or in the case $q = \infty$
where we then require
\[
E \Big[  \sup_{t \in (0,T)}  \int_{D}  \kabs{ Y(t) \dq f_2(x,t) }^{p} \dx
\Big]  \leq C \,.
\]
From the theorem we then conclude that $D_k f_1 \in L^p(D \times [0,T])$
and $D_k f_2 \in L^{\infty}(0,T; L^{p}(D))$ with probability one, respectively. In particular, if we take a function $f \in W^{1,p}(D)$ and if the previous assumptions are satisfied for $f_1 = Df$ and $f_2 = f$, then the conclusions are equivalent to $D_k f \in V^p(D_T)$.
\end{remark}

\subsection{A criterion for pathwise H\"older continuity}

We next discuss a criterion which guarantees H\"older continuity of (a suitable representative of) a given functions $u \colon D \times [0,T] \rightarrow\mathbb{R}^N$. For example, Sobolev's embedding theorem provides a criterion easy to apply if $u$ is in a suitable Sobolev space $W^{1,q}( D \times [0,T],\R^N)$ -- but which in general is not satisfied for the solutions considered in our paper since derivatives in time need not exist. Instead, we now prove that it is sufficient that only the spatial derivatives belong to a suitable Lebesgue space, provided that a weak form of continuity in time (i.\,e. of the $L^2(D)$-norm) is available.

\begin{lemma}
\label{lem_det_Hoelder}
If a function $u \colon D \times [0,T] \rightarrow\mathbb{R}^N$ has the
properties
\[
Du\in L^{\infty}(0,T;L^{n+\alpha}(D,\R^{nN})),\qquad u\in C^{\beta}(0,T;L^{2}(D,\R^N))
\]
for some $\alpha,\beta>0$, $D\subset\mathbb{R}^{n}$ a bounded, regular domain,
then
\[
u\in C^{\gamma}(D \times [0,T],\R^N)
\]
for some $\gamma>0$, depending only on $\alpha$, $\beta$ and $n$. 

\begin{proof}
First, we deduce spatial H\"older continuity for every time slice. 
From the assumption $Du\in L^{\infty}(0,T;L^{n+\alpha}(D))$ we deduce $u\in
L^{\infty}(0,T;C^{\delta}(D))$ for some $\delta>0$, depending only on $\alpha$
and $n$, by Sobolev's embedding theorem. Namely, there exists $C_{1}>0$ such
that
\begin{equation}
\left\vert u(x,t)  -u(y,t)  \right\vert \leq
C_{1}\left\vert x-y\right\vert ^{\delta}\label{Holder1}%
\end{equation}
for all $t \in [0,T]  $, $x,y\in D$. 

Our next aim is H\"older continuity in time, at a fixed point. From the inequality
\begin{equation*}
\knorm{ u(\cdot,t)  -u(\cdot,s) }_{L^{2}(D)}\leq C_{2} \kabs{ t-s }^{\beta}
\end{equation*}
for $s,t \in [0,T]$, we infer for every set $B\subset D$
\[
\inf_{x\in B} \kabs{u(x,t) - u(x,s)}
\leq\frac{1}{\kabs{B}} \int_{B} \kabs{ u(x,t) -u(x,s)} \dx \leq\frac{1}{\kabs{B}^{1/2}} \knorm{ u(\cdot,t) - u(\cdot,s) }_{L^{2}(D)} \leq \frac{C_{2} \, \kabs{t-s}^{\beta}}{\kabs{B}^{1/2}}.
\]
Let $x_{0} \in D$ be given. In order to prove H\"older continuity in time at $x_0$, 
we estimate
\begin{align*}
\left\vert u(  x_{0},t)  -u(  x_{0},s)  \right\vert  &
\leq\left\vert u(  x_{0},t)  -u(  x,t)  \right\vert
+\left\vert u(  x,t)  -u(  x,s)  \right\vert +\left\vert
u(  x,s)  -u(x_{0},s)  \right\vert \\
& \leq2C_{1}\left\vert x-x_{0}\right\vert ^{\delta}+\left\vert u(
x,t)  -u(  x,s)  \right\vert
\end{align*}
for every $x\in D$. Hence, if we take $x$ in a ball $B(  x_{0},\rho)$, we have
\begin{align*}
\left\vert u(  x_{0},t)  -u( x_{0},s)  \right\vert  &
\leq2C_{1}\rho^{\delta}+\inf_{x\in B(  x_{0},\rho)  }\left\vert
u(  x,t)  -u(  x,s)  \right\vert \\
& \leq2C_{1}\rho^{\delta}+C_{3}\frac{C_{2}\left\vert t-s\right\vert ^{\beta}%
}{\rho^{n/2}}
\end{align*}
where $C_{3}$ is such that $\left\vert B(  x_{0},\rho)  \right\vert
=\rho^{n}/C_{3}^{2}$. Let us now choose $\rho=\left\vert t-s\right\vert
^{\varepsilon}$ for some $\varepsilon > 0$:
\[
\left\vert u(  x_{0},t)  -u(x_{0},s)  \right\vert
\leq2C_{1}\left\vert t-s\right\vert ^{\varepsilon\delta}+C_{3}C_{2}\left\vert
t-s\right\vert ^{\beta-\varepsilon n/2}.
\]
If we choose for instance $\varepsilon=\beta/n$, we get
\begin{equation}
\left\vert u(  x_{0},t)  -u( x_{0},s)  \right\vert \leq
C_{4}\left\vert t-s\right\vert ^{\eta}\label{Holder2}
\end{equation}
for some $\eta,C_{4}>0$, independently of $x_{0}\in D$, $t,s\in\left[
0,T\right]  $. The constant $\eta$ depends only on $\beta$, $\delta$ and $n$.

From~\eqref{Holder1} and~\eqref{Holder2} it is now
straightforward to deduce the claim of the lemma.
\end{proof}
\end{lemma}

With the previous lemma at hand, we now give a criterion in the probabilistic setting, with $(\Omega,F,P)$ a complete probability space, which is adapted to weak solutions.

\begin{proposition}
\label{prop_Hoelder}
Let $u: D \times [0,T] \times \Omega \rightarrow \mathbb{R}^{nN}$ have
the properties
\begin{equation}
\label{int_assumption_Hoelder}
P\big(  Du\in L^{\infty}(0,T;L^{n+\varepsilon}(D,\mathbb{R}^{nN})) \big)  =1
\end{equation}
\[
u(  x,t)  =u_{0}(x)  +\int_{0}^{t}a(x,s) \ds +\int
_{0}^{t}b(x,s)\,dB_{s}%
\]
for some $\varepsilon > 0$, $u_{0}\in L^{2}(D)$, and with progressively measurable fields $a$, $b$ such that
\[
P \Big(  \int_{0}^{T}\int_{D} \kabs{ a(x,s) }^{2} \dx \ds+ 
\int_0^T \Big( \int_{D} \kabs{ b(x,s) }^{2} \dx \Big)^{\frac{2 + \varepsilon}{2}} \ds < \infty \Big)  =1.
\]
Then
\[
P\big(  u\in C^{\gamma}(D \times [0,T]) \big) =1
\]
for some $\gamma>0$ depending only on $\epsilon$.

\begin{proof}
\textbf{Step~1}. If we prove that, for some $\beta>0$,
\[
P\big(  u\in C^{\beta}(0,T;L^{2}(D)) \big)  =1 \,,
\]
then we get the claim of the proposition after the pathwise application of the previous Lemma~\ref{lem_det_Hoelder} (using in particular the stated independence of the H\"older exponent). To this end we observe that the function $u$ is the sum of two terms:
\[
u_{1}(  x,t)  =u_{0}(  x)  +\int_{0}^{t}a(x,s) \ds,\qquad
u_{2}(  x,t)  =\int_{0}^{t}b(x,s)\,dB_{s}%
\]
The term $u_{1}$ is, with probability one, of class $W^{1,2}(
0,T;L^{2}(D))$, hence it is of class $C^{1/2}(0,T;L^{2}(D))$:
\begin{equation*}
\knorm{ u_{1}(  t)  -u_{1}(  s) }
_{L^{2}(D)}  = \Bnorm{ \int_{s}^{t}a(\cdot,r) \dr }_{L^{2}(D)}
\leq\left\vert t-s\right\vert ^{1/2} \Big(  \int_{0}^{T}\int_{D} 
 \kabs{a(x,r)}^{2} \dx \dr \Big)^{1/2}.
\end{equation*}
So it only remains to prove that, for some $\beta>0$,
\[
P\left(  u_{2}\in C^{\beta}(0,T;L^{2}(D))\right)  =1.
\]

\textbf{Step~2}. For $R>0$, let
\[
\tau_{R}=\inf\big\{  t\in(0,T] \colon 
\int_0^t \knorm{b(\cdot,s)}_{L^{2}(D)}^{2+\varepsilon} \ds >R \big\}
\]
if the set is non empty, otherwise $\tau_{R}=T$. Let $\Omega_{R}\subset\Omega$
be the set where $\tau_{R}=T$. The family $\left\{  \Omega_{R}\right\}
_{R>0}$ is increasing, with
\[
P\big(\bigcup_{R>0}
\Omega_{R}\big) = 1
\]
because by assumption we have $P( \int_0^T \knorm{b(\cdot,s)}_{L^{2}(D)}^{2+\varepsilon} \ds < \infty)=1$. We now set
\[
b_{R}(x,s)=b(x,s)1_{s\leq\tau_{R}} \qquad \text{and} \qquad
u_{2,R}(  t)  =\int_{0}^{t}b_{R}(x,s)\,dB_{s}=\int_{0}^{t \wedge
\tau_{R}}b(x,s)\, dB_{s} \,.
\]
We then have
\[
\int_0^T \knorm{ b_{R}(\cdot,s,\omega) }_{L^{2}(D)}^{2+\varepsilon} \ds \leq R
\]
uniformly in $\omega$. Hence, for every $p\geq1$, we find
\begin{align*}
E\Big[  \knorm{ u_{2,R}(t)  -u_{2,R}(s)} _{L^{2}(D)}^{p} \Big]  
  & =E \Big[  \Bnorm{ \int_{s}^{t} b_{R}(\cdot,r)\,dB_{r} }_{L^{2}(D)}^{p} \Big] \\
  & \leq C_{p} E \Big[  \Big(  \int_{s}^{t} \knorm{ b_{R}(\cdot,r)}_{L^{2}(D)}^{2} 
    \dr \Big)^{\frac{p}{2}} \Big] \\
  & \leq C_p \kabs{t-s}^{\frac{p \varepsilon}{2(2+\varepsilon)}} E \Big[  \Big(  \int_{s}^{t} \knorm{ b_{R}(\cdot,r)}_{L^{2}(D)}^{2 + \varepsilon}  \dr \Big)^{\frac{p}{2+\varepsilon}} \Big] 
  \leq C_{p} \, R^{\frac{p}{2 + \varepsilon}} \kabs{t-s}^{\frac{p \varepsilon}{2(2+\varepsilon)}} .
\end{align*}
This implies, for $p = p(\varepsilon)$ sufficiently large, by Kolmogorov's regularity theorem for processes taking values in $L^{2}(D)$ (see~\cite[Theorem~3.3]{DAPZAB92} for a version in Banach spaces), that $u_{2,R}$ has a H\"{o}lder continuous version in $L^{2}(D)$
\[
\left\Vert u_{2,R}(\cdot,t,\omega)  -u_{2,R}(
\cdot,s,\omega)  \right\Vert _{L^{2}(D)}\leq C_{\beta,R}(
\omega)  \kabs{ t-s }^{\beta}
\]
with $\beta$ any H\"{o}lder exponent with $\beta<\frac{\varepsilon}{2 (2 + \varepsilon)}$. For $\omega\in\Omega_{R}$ we thus have (recalling the definition of $u_{2,R}$)
\[
\left\Vert u_{2}(\cdot,t,\omega)  -u_{2}(\cdot,s
,\omega)  \right\Vert _{L^{2}(D)}\leq C_{\beta,R}(\omega)
\kabs{ t-s }^{\beta}.
\]
Since $\bigcup_{R>0} \Omega_{R}$ is of full $P$-measure, we obtain $u_{2}\in C^{\beta}(0,T;L^{2}(D))$ for $P$-a.\,e. $\omega\in\Omega$. Now the previous Lemma~\ref{lem_det_Hoelder} can be applied, and the proof is complete.
\end{proof}
\end{proposition}

\subsection{A technical lemma}

In Kalita's paper a crucial point is to show higher regularity (such as higher integrability and differentiability) not only for the solution, but also for powers of the solution (resp. its gradient). For this purpose the following technical lemma was essential.

\begin{lemma}[\cite{KALITA94}]
\label{tech-kalita}
Let $u \colon \R^n \to \R^N$ be a function which is a.\,e. differentiable. Set $v = u \, \kabs{u}^{s}$ with $s \in (-1,\infty)$. Then, for
$\mu(s) := 1- (\frac{s}{2+s})^2$, we have a.\,e.
\begin{equation*}
Du \cdot Dv \geq \mu^{\frac{1}{2}}(s) \, \kabs{Du} \, \kabs{Dv} \,.
\end{equation*}
\end{lemma}

We need the following modification of this result, which on the one hand 
allows to test the system with powers (truncated for large values)
and which on the other hand satisfies an estimate corresponding to the one from
Lemma~\ref{tech-kalita}.

\begin{lemma}
\label{tech-kalita-2}
For every $K > 0$ and every $q \geq 1$ there exists a $C^2$-function $T_{q,K}
\colon \R^+ \to \R^+$ such that
\begin{enumerate}
 \item[(i)] $T_{q,K}$ is strictly increasing and convex on $\R^+$, and it
satisfies $T_{q,K}(t) = t^{2q}$ for all $t \leq K$;
 \item[(ii)] for all $t \in \R^+$ and a constant $c(q)$ the
growth with respect to $t$ is estimated by 
\begin{equation*}
T_{q,K}(t) + T'_{q,K}(t) \, t + T''_{q,K}(t) \, t^2 \leq  c(q) \, \min
\big\{ K^{2q-2} t^{2}, t^{2q} \big\} ;
\end{equation*} 
moreover, the inequalities $T_{q,K}''(t) t - T_{q,K}'(t)\leq 2(q-1) T_{q,K}'(t)$
as well as $T''_{q,K}(t) t^2 \leq c(q) T'_{q,K}(t) t \leq c(q) T_{q,K}(t)$ hold
true on $\R^+$;
 \item[(iii)] If $u \colon \R^n \to \R^N$ is a function which is a.\,e. differentiable and $\mu(q):= 1- (\frac{q-1}{q})^2$, then for the function $v = T_{q,K}'(\kabs{u})
\kabs{u}^{-1} u$ the following inequality is satisfied a.\,e.:
\begin{equation*}
Du \cdot Dv \geq \sqrt{\mu(q)} \, \kabs{Du} \, \kabs{Dv} \geq
\sqrt{\mu(q)} \, T_{q,K}'(\kabs{u}) \, \kabs{u}^{-1} \, \kabs{Du}^2 \, .
\end{equation*}
\end{enumerate}

\begin{proof}
We first assume $K=1$. We set 
\begin{equation*}
T_{q,1} (t) = \left\{ \quad \begin{array}{ll} t^{2q} \qquad & \text{if } t
\leq 1 \\
a t^2 + bt + c & \text{if } t > 1
\end{array} \right.
\end{equation*}
for some coefficients $a,b,c \in \R$ to be determined as follows. The
$C^2$-regularity condition implies that the following linear system has to be
satisfied:
\begin{equation*}
\begin{pmatrix}
1 & 1 & 1 \\
2 & 1 & 0 \\
2 & 0 & 0 \\
\end{pmatrix}
\begin{pmatrix}
a \\ b \\ c
\end{pmatrix}
=
\begin{pmatrix}
1 \\ 2q \\ 2q (2q -1)
\end{pmatrix}
\Rightarrow
\begin{pmatrix}
a \\ b \\ c
\end{pmatrix}
=
\begin{pmatrix}
q (2q - 1) \\ -4q (q-1) \\ 1 -3q + 2q^2 
\end{pmatrix}.
\end{equation*}
We now calculate some crucial quantities. We first observe that
\begin{equation*}
T_{q,1}'(t)  =   \left\{ \quad \begin{array}{ll} 
				2q t^{2q-1} \qquad & \text{if } t \leq 1 \\
				2q  (2q - 1) t - 4q (q-1) & \text{if } t > 1
                              \end{array} \right.
\end{equation*}
is strictly increasing and positive on $\R^+$. Thus, we immediately obtain assertion (i) of the lemma. Furthermore, we have 
\begin{equation*}
T_{q,1}''(t) t - T_{q,1}'(t)  =   \left\{ \quad \begin{array}{ll} 
				4q (q-1) t^{2q-1} \qquad & \text{if } t \leq 1
\\
				4q (q-1) & \text{if } t > 1,
                              \end{array} \right.
\end{equation*}
which is again positive on $\R^+$. Moreover, for all $t \in \R^+$ we obtain
\begin{equation}
\label{T-estimate-quantitative}
T_{q,1}''(t) t - T_{q,1}'(t)  \leq  2(q-1) T_{q,1}'(t) \,,
\end{equation}
which in particular yields the inequality $T_{q,1}''(t) t^2 \leq c(q) T_{q,1}'(t)
t$ of assertion (ii). The last inequality $T'_{q,1}(t) t \leq c(q) T_{q,1}(t)$ is also checked easily. For the function $v$ given in (iii) we next compute
\begin{align*}
D_{i} v^{\alpha} & = T_{q,1}'(\kabs{u})
\frac{D_{i} u^{\alpha}}{\kabs{u}} + \big( T_{q,1}''(\kabs{u}) \kabs{u} - T_{q,1}'(\kabs{u})
\big)  \frac{D_{i} u \cdot u \, u^{\alpha}}{\kabs{u}^3} \,,\\
Du \cdot Dv & =  T_{q,1}'(\kabs{u})
\frac{\kabs{Du}^2}{\kabs{u}}  + \big( T_{q,1}''(\kabs{u}) \kabs{u} - T_{q,1}'(\kabs{u}) \big)
 \frac{\kabs{Du \cdot u}^2}{\kabs{u}^3}\,.
\end{align*}
In particular, this shows $Du \cdot Dv \geq 0$, using again the positivity of $T_{q,1}''(t) t - T_{q,1}'(t)$ and of $T_{q,1}'(t)$ on $\R^+$. Furthermore, we obtain
\begin{multline*}
\quad \kabs{Dv}^2 =  T_{q,1}'(\kabs{u})^2 \, \frac{\kabs{Du}^2}{\kabs{u}^2} + \big( T_{q,1}''(\kabs{u}) \kabs{u} - T_{q,1}'(\kabs{u})
\big)^2 \, \frac{\kabs{Du \cdot u}^2}{\kabs{u}^4} \\
	 {}+  2 \, T_{q,1}'(\kabs{u}) \, \big( T_{q,1}''(\kabs{u}) \kabs{u} -
T_{q,1}'(\kabs{u}) \big) \, \frac{\kabs{Du \cdot u}^2}{\kabs{u}^4}
	 \geq  T_{q,1}'(\kabs{u})^2 \, \frac{\kabs{Du}^2}{\kabs{u}^2}\,, \quad
\end{multline*}
which yields the second inequality in (iii). Now, keeping in mind the definition of $\mu(\cdot)$, we find via the previous estimate~\eqref{T-estimate-quantitative}
\begin{align*}
\kabs{Du \cdot Dv}^2 - \mu(q) \, \kabs{Du}^2 \kabs{Dv}^2 & = 
\Big(\frac{q-1}{q}\Big)^2 \, T_{q,1}'(\kabs{u})^2  \,
\frac{\kabs{Du}^4}{\kabs{u}^2} +  \big( T_{q,1}''(\kabs{u}) \kabs{u} - T_{q,1}'(\kabs{u})
\big)^2 \, \frac{\kabs{Du \cdot u}^4}{\kabs{u}^6} \\
	& \quad {}- \big( T_{q,1}''(\kabs{u}) \kabs{u} - T_{q,1}'(\kabs{u})
\big)^2 \, \frac{2q-1}{q^2}
	\frac{\kabs{Du \cdot u}^2 \kabs{Du}^2}{\kabs{u}^4}\\
	& \quad {} + 2 \, \Big(\frac{q-1}{q}\Big)^2 \, T_{q,1}'(\kabs{u}) \, 
	\big( T_{q,1}''(\kabs{u}) \kabs{u} - T_{q,1}'(\kabs{u}) \big) 
	\,  \frac{\kabs{Du \cdot u}^2 \kabs{Du}^2}{\kabs{u}^4} \\
	& \geq  \Big(\frac{q-1}{q}\Big)^2 \, T_{q,1}'(\kabs{u})^2  \,
\frac{\kabs{Du}^4}{\kabs{u}^2} + \big( T_{q,1}''(\kabs{u}) \kabs{u} - T_{q,1}'(\kabs{u})
\big)^2 \, \frac{\kabs{Du \cdot u}^4}{\kabs{u}^6} \\
	& \quad {}- 2 \, \frac{q-1}{q} \, T_{q,1}'(\kabs{u}) \, \big(
T_{q,1}''(\kabs{u}) 
	\kabs{u} - T_{q,1}'(\kabs{u}) \big) 
	\, \frac{\kabs{Du \cdot u}^2 \kabs{Du}^2}{\kabs{u}^4} \,,
\end{align*}
which is non-negative by Young's inequality. This finishes the proof of (iii) for the case $K=1$. To complete the proof of the lemma it is sufficient to observe that for general $K > 0$ the coefficients $a,b,c$ have to be replaced by $a K^{2q-2}, b K^{2q-1}, c K^{2q}$, and the conclusion then follows exactly as above.
\end{proof}
\end{lemma}

%******************************************************************************
%		Higher differentiability of weak solutions
%******************************************************************************

\section{Higher differentiability of weak solutions}

In this section we start working on the solution $u$ of the parabolic
system~\eqref{equation}. First, we prove an upper bound for the average of \emph{weighted} norms of $Du$. This will be done in Section~\ref{sec_apriori} and serves also to explain the general strategy to obtain such estimates. We will then extract higher regularity properties of the solution, still following the ideas given in Section~\ref{sec_apriori}. More precisely, as final result of this section, we are interested in pathwise higher integrability of the gradient $Du$, which will be the core of the proof of the regularity result given in Theorem~\ref{main-result}. 

\subsection{An a~priori estimate}
\label{sec_apriori}

From Definition~\ref{def_weak_solution_ito} of a weak solution $u$ to~\eqref{equation}, no a~priori information is available on the expected value of the solution. In particular, we only know that $u(\omega)$ belongs to the space $V^2(D_T,\R^N)$ for $P$-almost every $\omega \in \Omega$, but it is still possible that the average $E[\knorm{u}_{V^2(D_T,\R^N)}]$ is infinite. Even if this cannot be excluded, we can win an a~priori information on the average of weighted norms of $u$.

The strategy for a~priori estimates for deterministic elliptic of parabolic systems is simply to ``test'' the equation with the solution (or some modification of it), and an estimate then follows by employing the regularity and growth properties of the system. For stochastic systems, testing with an the appropriate version is replaced by the application of an It\^o's formula for Banach spaces. Then, a first pathwise estimate follows (Step~1). Since we are interested in averages, the first estimate is rewritten (Step~2) by introducing weights depending on the solution itself. With these weights we can finally take the expectation (Step~3) and end up with the desired estimate, which we now state in its precise form.

\begin{lemma}
\label{lemma_apriori}
Let $u \in V^2(D_T,\R^N)$ be a weak solution to the initial boundary value problem
to~\eqref{equation} under the assumptions~\eqref{GV-A}$_{1,2}$,~\eqref{GV-H}$_{1,2}$ and with $u(\ccdot,0) = u_0(\cdot) \in L^2(D,\R^N)$. Suppose further that the
smallness condition $L_H^2 < 2 \kappa^{-1} ( 1 - (1-\nu^2)^{1/2})$ is satisfied,
and let $D_0 \subset D$ with $d_0 := \dist(D_0,\partial D) > 0$. Then there holds
\begin{equation*}
E \, \Big[ \int_0^T e^{-\int_0^t c_0 G_0(u,f)
\ds } \knorm{Du(t)}_{L^2(D_0)}^2 \dt \Big]
	\leq c_0 \big( \knorm{u_0}_{L^2(D)}^2 + 1 
	+ \E \big[ \knorm{f_H}_{L^2(D_T)}^2 \big] \big) 
\end{equation*}
for a constant $c_0$ depending only on $D,L,L_H,d_0,\kappa$ and $\nu$, and a function $G_0(u,f)$ given by~\eqref{aprioriG0}.

\begin{proof}
\textbf{Step~1.} A preliminary pathwise estimate. 
We start by multiplying the equation~\eqref{equation} with a standard cut-off function $\eta \in C^{\infty}(D,[0,1])$ which satisfies  $\eta \equiv 1$ on $D_0$ and $\kabs{D\eta}
\leq c(d_0)$. Obviously, the map $\eta \dq u$ has the same properties concerning
integrability and measurability as $u$, and the It\^o formula from Theorem
\ref{Ito-II} in Banach spaces may be applied with the Gelfand triple
$W_0^{1,2}(D,\R^N) \subset L^2(D,\R^N) \subset W^{-1,2}(D,\R^N)$. This yields
the existence of a subset $\Omega' \subset \Omega$ of full measure
$P(\Omega')=1$  and a function $u' \colon [0,T] \to W^{1,2}(D,\R^N)$ which satisfies: $u'$ is $\mathcal{F}_t$-adapted on $[0,T] \times
\Omega'$, continuous in~$t$ for every $\omega \in \Omega'$, and $u' = u \eta$ holds
for $P \times \Lm^1$-almost all $(t,\omega) \in [0,T] \times \Omega$. Moreover, using the integration by parts formula, we have for every $\omega \in \Omega'$
\begin{multline}
\label{apriori1}
\knorm{u'(t)}^2_{L^2(D)} + 2 \int_0^t \sp{ D( u(s)
\, \eta^2)}{ A(\ccdot,s,u(s),Du(s))}_{L^2(D)} \ds \\
	=  \knorm{u_0 \, \eta}^2_{L^2(D)} 
	+ 2 \int_0^t \sp{ u'(s) \, \eta}{ H(\cdot,s,Du(s)) \,
dB_s}_{L^2(D)} + \int_0^t \knorm{H(\cdot,s,Du(s)) \, \eta}_{L^2(D)}^2
\ds
\end{multline}
(with the convention $\kabs{M}^2 = \sum_{i,j=1}^n M_{ij}^2$ for every $n \times
n$ matrix). Next we need to estimate the second integral on the left-hand side of the previous identity, employing the assumptions~\eqref{GV-A}. For this purpose, we first observe with~\eqref{GV-A}$_{1,2}$ and Young's inequality that
\begin{align*}
\lefteqn{ \hspace{-0.5cm} \sp{Du(s) \, \eta^2}{A(\ccdot,s,u(s),Du(s))}_{L^2(D)} }\\
  & = \int_0^1 \sp{Du(s) \, \eta^2}{D_z A(\ccdot,s,u(s),rDu(s)) \, Du(s)}_{L^2(D)} \dr + \sp{Du(s) \, \eta^2}{A(\ccdot,s,u(s),0)}_{L^2(D)} \\
  & \geq \frac{1}{\kappa} \big(1 - (1-\nu^2)^{\frac{1}{2}} - \epsilon \big) \, \knorm{Du(s) \, \eta}_{L^2(D)}^2 - c(\epsilon^{-1}, L) \, \big(  \knorm{u(s)}_{L^{\frac{2 (n+2)}{n}}(D)}^{\frac{2(n+2)}{n}} + \knorm{f(s)}_{L^a(D)}^a \big)
\intertext{for all $s \in (0,T)$ and every $\epsilon > 0$. Moreover, we find}
\lefteqn{ \hspace{-0.5cm} \kabs{ \sp{u(s) \otimes D\eta \, \eta}{A(\ccdot,s,u(s),Du(s))}_{L^2(D)} } }\\ 
  & \leq \frac{\epsilon}{\kappa} \knorm{Du(s) \, \eta}_{L^2(D)}^2 + c(L,\epsilon^{-1},\kappa,d_0) \, \big( \knorm{u(s)}_{L^2(D)}^2 + \knorm{u(s)}_{L^{\frac{2 (n+2)}{n}}(D)}^{\frac{2(n+2)}{n}} + \knorm{f(s)}_{L^a(D)}^a \big) \,.
\end{align*}
Next, the integrand of the last term on the right-hand side of~\eqref{apriori1} is bounded via~\eqref{GV-H}$_{1,2}$ by
\begin{equation*}
\knorm{H(\cdot,s,Du(s)) \, \eta}_{L^2(D)}^2 \leq \big(L_H^2 + \frac{\epsilon}{\kappa}\big) \, \knorm{Du(s) \, \eta}_{L^2(D)}^2 + c(\epsilon^{-1},\kappa) \, \knorm{f_H(s)}_{L^2(D)}^2 \,.
\end{equation*}
Combining the last three inequalities (here enters the smallness assumption on $L_H$) with~\eqref{apriori1}, choosing $\epsilon$ sufficiently small (in dependency of $L_H$ and $\nu$) and using H\"older's inequality, we thus end up with the announced pathwise estimate
\begin{align}
\label{apriori2}
\lefteqn{ \hspace{-0.5cm} \knorm{u'(t)}^2_{L^2(D)} + c^{-1}(L_H,\kappa,\nu) \int_0^t \knorm{Du(s) \, \eta}_{L^2(D)}^2 \ds } \nonumber \\
  & \leq \knorm{u_0 \, \eta}^2_{L^2(D)}  + 2 \int_0^t \sp{ u'(s) \, \eta}{ H(\cdot,s,Du(s)) \, dB_s}_{L^2(D)} \nonumber \\
  & \quad {} + c_0(D,L,L_H,d_0,\kappa,\nu) \int_0^t \big( 1 + \knorm{u(s)}_{L^{\frac{2 (n+2)}{n}}(D)}^{\frac{2(n+2)}{n}} \!\! + \knorm{f(s)}_{L^a(D)}^a +  \knorm{f_H(s)}_{L^2(D)}^2 \big) \ds \,.
\end{align}

\textbf{Step~2.} An improved pathwise estimate. The next step consists in getting a pathwise estimate where the bound on the right-hand side contains a deterministic part almost \emph{independent} of the weak solution and the function $f$, and a stochastic part which might still depend on the solution. We start by defining
\begin{equation}
\label{aprioriG0}
G_0(u,f)(s) = 1 + \knorm{u(s)}_{L^{\frac{2 (n+2)}{n}}(D)}^{\frac{2(n+2)}{n}} + \knorm{f(s)}_{L^a(D)}^a
\end{equation}
for $s \in (0,T)$. Obviously, $G_0$ belongs to $L^1(0,T)$ with probability one. Then we use a Gronwall-type argument, by applying the one-dimensional It\^o-formula to ${\rm exp} (-\int_0^t c_0 G_0(u,f)(\tilde{s}) \, d\tilde{s}) \, (1+\knorm{u'(t)}^2_{L^2(D)})$, as e.\,g. in~\cite[Proof of Theorem~5.1]{Schmalfuss97}. Thus, we get
\begin{align}
\label{apriori3}
\lefteqn{ \hspace{-0.5cm} e^{-\int_0^t c_0 G_0(u,f)(\tilde{s}) \, d\tilde{s}} \,
\knorm{u'(t)}^2_{L^2(D)}  + c^{-1}(L_H,\kappa,\nu) \int_0^t
	e^{-\int_0^s c_0 G_0(u,f)(\tilde{s}) \, d\tilde{s} } \knorm{Du(s) \,
\eta}_{L^2(D)}^2 \ds} \nonumber \\
	& \leq  \knorm{u (0)\, \eta}^2_{L^2(D)} + 1 + 
	2 \int_0^t e^{-\int_0^s c_0 G_0(u,f)(\tilde{s}) \, d\tilde{s} } \, 
	\sp{ u'(s) \, \eta}{ H(\cdot,s,Du(s)) \, dB_s}_{L^2(D)} \nonumber
\\
	& \quad {}+ c_0 \int_0^t e^{-\int_0^s c_0 G_0(u,f)(\tilde{s}) \, d\tilde{s}
} \, \knorm{f_H(s)}_{L^2(D)}^2 \ds \,.
\end{align}
Note that we here have omitted a negative term which appeared on the right-hand side and the positive term containing the $1$ on the left-hand side. This is the desired improved pathwise estimate. We note that $u$ and $f$ still appear in the function $G_0$ in the deterministic integral on the right-hand side, but in a way that for greater values of $u$ or $f$ the integral gets smaller. At the same time obviously also the exponential factor on the left-hand side will get smaller, but this allows us now to pass to Step~3.  

\label{continuous_in_t}
\textbf{Step~3.} An estimate for the expected value with weights. Uniform estimates for the average of the weak solution (e.\,g. for expressions of the form $E[\knorm{u}]$ for some norm of $u$ or $Du$) of course cannot be expected under such weak assumptions as we have supposed in the lemma. But the previous inequality~\eqref{apriori3} now allows us to get a weighted inequality, with no stochastic terms on the right-hand side. Since the expectation of the stochastic integral is not a~priori known to vanish, we now apply a stopping time argument. 

From identity~\eqref{apriori1} it follows that the process $\knorm{ u'(t)}_{L^{2}(D)}^{2}$ has a continuous version in $t$, used in the following argument. For every $R>0$ we introduce the random time
\[
\tau_{R}:=\inf \Big\{  t \in [0,T]  \colon \int_{0}^{t} \knorm{u'(s)}_{L^{2}(D)}^{2} \, \knorm{H(s,Du(s)) \eta }_{L^{2}(D)  }^{2} \ds > R \Big\}
\]
with $\tau_{R}=T$ when the set is empty. We note that $\knorm{u'(s)}_{L^{2}(D)}^{2}  \knorm{H(s,Du(s)) \eta }_{L^{2}(D)}^{2}$ is in $L^1(0,T)$ with probability one, because of the property $u\in V_{0}^{2}(D_{T}%
,\mathbb{R}^{N})$ and the assumption~\eqref{GV-H}$_2$ on $H$. 
Hence, we have in particular $P(\lim_{R\rightarrow\infty}\tau_{R}=T)=1$ and
\[
P \big(  \lim_{R\rightarrow\infty} \knorm{ u'(t\wedge\tau_{R}) }_{L^{2}(D)}^{2}= \knorm{ u'(t) }_{L^{2}(D)}^{2} \big) = 1
\]
for every $t \in [0,T]$. Now we take inequality~\eqref{apriori3} at time $t\wedge\tau_{R}$ and get
\begin{align*}
\lefteqn{ \hspace{-0.5cm} e^{-\int_0^{t\wedge\tau_{R}} c_0 G_0(u,f)(\tilde{s}) \, d\tilde{s}} \,
\knorm{u'({t\wedge\tau_{R}})}^2_{L^2(D)}  + c^{-1} \int_0^{t\wedge\tau_{R}}
	e^{-\int_0^s c_0 G_0(u,f)(\tilde{s}) \, d\tilde{s} } \knorm{Du(s) \,
\eta}_{L^2(D)}^2 \ds}  \\
	& \leq  \knorm{u (0)\, \eta}^2_{L^2(D)} + 1 + 
	2 \int_0^{t\wedge\tau_{R}} e^{-\int_0^s c_0 G_0(u,f)(\tilde{s}) \, d\tilde{s} } \, 
	\sp{ u'(s) \, \eta}{ H(\cdot,s,Du(s)) \, dB_s}_{L^2(D)} 
\\
	& \quad {}+ c_0 \int_0^{t\wedge\tau_{R}} e^{-\int_0^s c_0 G_0(u,f)(\tilde{s}) \, d\tilde{s}
} \, \knorm{f_H(s)}_{L^2(D)}^2 \ds \,.
\end{align*}
Now we have
\begin{multline*}
\quad \int_{0}^{t\wedge\tau_{R}}e^{-\int_{0}^{s}c_{0}G_0(u,f)\,(  \tilde
{s}) \,d\tilde{s}} \, \sp{u'(s)\eta}{H(s,Du(s)) \, dB_{s}}_{L^{2}(D)} \\
=\int_{0}^{t}e^{-\int_{0}^{s} c_{0} G_0(u,f)\,(\tilde{s})
\,d\tilde{s}} \, 1_{s\leq\tau_{R}} \sp{u'(s)\eta}
{H(s,Du(s)) \, dB_{s}}_{L^{2}(D)} \quad
\end{multline*}
and
\begin{multline*}
\quad \int_{0}^{t} e^{-2 \int_{0}^{s}c_{0}G_0(u,f)\,(\tilde{s})
\,d\tilde{s}} 1_{s\leq\tau_{R}} \knorm{ u'(s)}_{L^{2}(D)}^{2}
\knorm{H(s,Du(s)) \eta }_{L^{2}(D)}^{2} \ds\\
 \leq \int_{0}^{t\wedge\tau_{R}} \knorm{ u'(s) }_{L^{2}(D)}^{2}
 \knorm{ H(s,Du(s)) \eta }_{L^{2}(D)}^{2} \ds \leq R \quad
\end{multline*}
by definition of $\tau_{R}$. Thus the stopped stochastic integral above is a
martingale, hence with expected value zero. This implies
\begin{multline*}
\quad E \Big[  e^{-\int_{0}^{t\wedge\tau_{R}}c_{0}G_0(u,f)\,(
\tilde{s})  d\tilde{s}}\, \knorm{u'(t\wedge\tau
_{R}) }_{L^{2}(D)}^{2}  \Big]  +E\Big[ c^{-1} \int_{0}^{t\wedge
\tau_{R}}e^{-\int_{0}^{s}c_{0}G_0(u,f)\,(\tilde{s})  d\tilde{s}
} \knorm{Du(s) \,\eta }_{L^{2}(D)}^{2} \ds \Big]
\\
\leq \Vert\triangle_{k,h}u_0\,\eta\Vert_{L^{2}(D)}^{2}+1+ c_0 \, E\Big[
\,\int_{0}^{T}\, \knorm{ f_{H}(s) }_{L^{2}(D)}^{2} \ds \, \Big]  \,. \quad
\end{multline*}
On the left-hand-side we now apply Fatou's lemma to the first term and the monotone convergence theorem to the second one, and we get
\begin{multline*}
\quad E\Big[  e^{-\int_{0}^{t}c_{0}G_0(u,f)\,(\tilde{s})  d\tilde{s}}\, \knorm{u'(t) }_{L^{2}(D)}^{2} \Big]  +E \Big[ c^{-1} \int_{0}^{t}e^{-\int_{0}^{s}c_{0}
G_0(u,f)\,(\tilde{s})  d\tilde{s}} \knorm{Du(s) \,\eta }_{L^{2}(D)}^{2} \ds \Big]  \\
\leq \knorm{ u_0\,\eta }_{L^{2}(D)}^{2}+1+ c_0 \, E \Big[
\int_{0}^{T} \knorm{ f_{H}(s) }_{L^{2}(D)}^{2} \ds \, \Big] \quad
\end{multline*}
for every $t\in [0,T]$. This proves the bound claimed by the lemma.
\end{proof}
\end{lemma}

\subsection{Existence of second space derivatives}
\label{section_second_derivatives}

We next study the existence of second order space derivatives. For
deterministic elliptic and parabolic partial differential equations it is a
standard procedure to establish the existence of higher order derivatives by
finite difference quotients methods. The basic idea in the deterministic case is
the following. Once the norm of finite difference quotients of $Du$ 
are kept under control \emph{independently} of its step size, i.\,e. $\knorm{\dq
Du}_{L^p} \leq C$ with $C$ independent of $h$ and with $p \in (1,\infty)$, then
the weak derivative $D_k Du$ exists and has finite norm in $L^p$ (and as long as
one is away from the boundary also the reverse it true). So uniformly bounded
difference quotients of $Du$ can heuristically be considered as second
derivatives $D_k Du$. This uniform bound in turn is usually achieved by
``testing the system'' with appropriate modifications of the solutions (formally
one might think of $\triangle_{k,-h} \dq u$) and relies on the one hand on the
ellipticity of the vector field $A$ and on the other hand on its regularity with
respect to the $x$ and $u$ variables (we note that it seems mandatory to have at
least Lipschitz-regularity in order to expect the existence of full second space
derivatives).

For the stochastic perturbed system~\eqref{equation} the approach for proving the existence of higher order derivatives is still very similar, but we need some
modifications due to the stochastic terms. The above strategy (with testing replaced by the the use of the It\^o formula in Banach spaces) applied to our stochastic system gives --~after some standard, though very technical computations~-- a preliminary pathwise estimate for finite difference quotients of $u$ and $Du$ (this corresponds in some sense to Step~1 in the proof of the previous Lemma~\ref{lemma_apriori}). But since this estimate still involves a stochastic integral, it is not yet possible to gain immediately any information on second order derivatives. In a second step, this pathwise estimate is rewritten (here again some Gronwall-type inequality is needed), which allows in the third step to take the expectation of a weighted version of $\knorm{\dq Du}_{L^2}$ and to bound it independently of the stepsize $h$. This is still sufficient to deduce the existence of $D_k Du$ with probability one (see Section~\ref{sect-weak-der}). 

Given a deterministic initial condition $u_0$ (sufficiently regular) we now give
the precise statement on the boundedness of the expectation of finite difference
quotients of $Du$.

\begin{lemma}
\label{Diff-quotients}
Let $u \in V^2(D_T,\R^N)$ be a weak solution to the initial boundary value problem
to~\eqref{equation} under the assumptions~\eqref{GV-A},~\eqref{GV-H} and with
$u(\ccdot,0) = u_0(\cdot) \in W^{1,2}(D,\R^N)$. Suppose further that the
smallness condition $L_H^2 < 2 \kappa^{-1} ( 1 - (1-\nu^2)^{1/2})$ is satisfied,
and let $D' \subset D$ with $d' := \dist(D',\partial D) > 0$. Then there holds
\begin{multline*}
\quad \sup_{\kabs{h} < d'} \, \E \, \Big[ \, \sup_{t \in (0,T)} e^{-\int_0^t c'
G'(u,f) \ds } \knorm{\dq u(t)}^2_{L^2(D')} + \int_0^T e^{-\int_0^t c' G'(u,f)
\ds } \knorm{D \dq u(t)}_{L^2(D')}^2 \dt \Big] \\
	 \leq c' \, \big(  \knorm{D_k u_0}_{L^2(D)}^2 + 1 
	+ \E \big[ \knorm{f_H^{\frac{a}{a-2}}}_{L^2(D_T)}^2 \big] \big) \quad
\end{multline*}
for every $k \in \{1,\ldots,n\}$, a constant $c'$ depending only on
$n,D,T,L,L_H,d',\kappa$, and $\nu$, and a function $G'(u,f)$ given by~\eqref{higher2}
further below.

\begin{proof}
We here proceed similarly to the proof of Lemma~\ref{lemma_apriori}, with the main difference that instead of $u$ we now need to estimate the difference quotients $\dq u$.

\textbf{Step~1.} We first observe that if $u$ is a solution of~\eqref{equation}, then for all $t \in [0,T]$ by definition also the following identity holds true for $P$-a.\,s.:
\begin{align*}
\sp{\eta \, \dq u(t) - \eta \, \dq u_0}{\varphi}_{L^2(D)} 
	& =   \int_0^t \sp{\eta \, \diverg 
	\dq A(\cdot,s,u(s),Du(s))}{\varphi}_{W^{-1,2}(D);W^{1,2}_0(D)} \ds \\
	& \quad {} + \int_0^t \sp{\varphi}{\eta \, \dq H(\cdot,s,Du(s)) \big)
\, dB_s}_{L^2(D)} 
\end{align*}
for all $\varphi \in W^{1,2}_0(D,\R^N)$. Here $k \in \{1,\ldots,n\}$ is
arbitrary, $h \in \R$ with $\kabs{h} < d'$, and for sets $D' \subset D_0 \subset D$ we denote by $\eta \in C^{\infty}(D_0,[0,1])$ a standard cut-off function satisfying $\eta \equiv 1$ on $D'$ and $\kabs{D\eta} \leq c(d')$. Therefore, the map $\eta \dq u$ has the same properties concerning integrability and measurability as $u$, and the It\^o formula from Theorem~\ref{Ito-II} in Banach spaces is again applied with the Gelfand triple
$W_0^{1,2}(D,\R^N) \subset L^2(D,\R^N) \subset W^{-1,2}(D,\R^N)$. We hence get
a subset $\Omega' \subset \Omega$ of full measure $P(\Omega')=1$ and a function $u_k' \colon [0,T] \to W^{1,2}(D,\R^N)$ with the following properties: $u_k'$ is $\mathcal{F}_t$-adapted on $[0,T] \times \Omega'$, continuous in $t$ for every $\omega \in \Omega'$, and satisfies $u_k' =  \dq u \eta$ for $P \times \Lm^1$-almost all $(t,\omega) \in [0,T] \times \Omega$. Moreover, for every $\omega \in \Omega'$ we have
\begin{align}
\label{higher1}
\lefteqn{\hspace{-0.5cm} \knorm{u_k'(t)}^2_{L^2(D)} + 2 \int_0^t \sp{ D(\dq u(s)
\, \eta^2)}{ \dq A(\ccdot,s,u(s),Du(s))}_{L^2(D)} \ds } \nonumber \\
	& =  \knorm{\dq u_0 \, \eta}^2_{L^2(D)} 
	+ 2 \int_0^t \sp{ u_k'(s) \, \eta}{ \dq H(\cdot,s,Du(s)) \,
dB_s}_{L^2(D)} \nonumber \\
	& \quad {}+ \int_0^t \knorm{\dq H(\cdot,s,Du(s)) \, \eta}_{L^2(D)}^2
\ds \,. \quad
\end{align}
Our first aim is to deduce a pathwise estimate for finite
differences of $u$ and $Du$, respectively. To this end we first study in detail
the second term on the left-hand side. For almost every $t \in [0,T]$ we
decompose the finite difference quotient applied on $A(x,t,u,Du)$ as follows
\begin{align}
\label{def-ABC}
\lefteqn{\hspace{-0.75cm} \dq A(x,t,u(x),Du(x))} \nonumber  \\
	& =  h \, \big[ A(x+h e_k,t, u(x+h e_k),Du(x+h e_k)) -  A(x + h e_k,t,
u(x+h e_k),Du(x)) \big] 
	\nonumber \\
	& \quad {} + h \, \big[ A(x + h e_k,t, u(x+h e_k),Du(x)) -  A(x +
e_k,t, u(x),Du(x)) \big] \nonumber \\
	& \quad {} + h \, \big[ A(x + h e_k,t, u(x),Du(x)) -  A(x,t,
u(x),Du(x)) \big] \nonumber \\
	& =  \int_0^1 D_z A(x + h e_k,t, u(x+h e_k),Du(x) + r \, h \, \dq
Du(x)) \dr \, 
	\dq Du(x) \nonumber  \\
	& \quad {} + \int_0^1 D_u A(x + h e_k,t, u(x) + r \, h \, \dq
u(x),Du(x)) \dr \, \dq u(x) 
	\nonumber \\ 
	& \quad {} + \int_0^1 D_x A(x + r \, h e_k,t, u(x),Du(x)) \dr
\nonumber \\
	& =:  {\cal A} (h) + {\cal B} (h) + {\cal C} (h)
\end{align}
with the obvious abbreviations. Using the assumptions~\eqref{GV-A},
H\"older's and Young's inequality, we now estimate the different terms arising
from this decomposition in equation~\eqref{higher1} on time slices $t \in (0,T)$
(on such slices we omit the notion of $t$). We first find for almost every $t
\in (0,T)$ and every $\epsilon > 0$
\begin{align*}
\lefteqn{\hspace{-0.75cm} \sp{D(\dq u \, \eta^2)}{{\cal A} (h) }_{L^2(D)} } \\
	& =  \kappa^{-1} \sp{D(\dq u \, \eta^2)}{D \dq u }_{L^2(D)} 
	- \kappa^{-1} \sp{D(\dq u \, \eta^2)}{D \dq u - \kappa \, {\cal A} (h)
}_{L^2(D)} \\
	& \geq  \kappa^{-1} \knorm{D \dq u \, \eta}_{L^2(D)}^2 
	- 2 \epsilon \, \knorm{D \dq u \, \eta}_{L^2(D)}^2  - c(\kappa,\epsilon)
\, \knorm{\dq u \, D\eta}_{L^2(D)}^2 \\
	& \quad {} - \kappa^{-1} \, (1-\nu^2)^{\frac{1}{2}} \, \knorm{D \dq u \,
\eta}_{L^2(D)}^2 \\
	& \geq  \big( \kappa^{-1} (1 - (1-\nu^2)^{\frac{1}{2}}) - 2 \epsilon
\big) \, \knorm{D \dq u \, \eta}_{L^2(D)}^2 
	 - c(\kappa,\epsilon,\knorm{D\eta}_{L^{\infty}(D)}) \, \knorm{D_k
u}_{L^2(D)}^2  
\end{align*}
where in the last line we have used the fact that the norm of the finite difference
quotient of a compactly supported function is always bounded by the norm of the
partial derivative (provided that the stepsize is sufficiently small). This
lower bound will be crucial (and can be understood as some ellipticity of the
vector field $A$ up to lower order terms). We next observe with the
Sobolev-Poincar\'{e} embedding (applied on every time-slice)
\begin{align*}
\lefteqn{\hspace{-0.75cm} \babs{ \sp{D (\dq u \, \eta^2)}{{\cal B} (h)
}_{L^2(D)} }}\\
	& \leq  L \, \big(\knorm{D \dq u \, \eta}_{L^2(D)} + 2 \knorm{\dq u \,
D\eta}_{L^2(D)} \big)
	\, \knorm{\dq u \, \eta}_{L^{\frac{2n}{n-2}}(D)}^{\theta} \\
	& \qquad {} \times \knorm{\dq u \,	
	\eta}_{L^2(D)}^{1-\theta} \, \big(
\knorm{Du}_{L^{\frac{2n}{(n+2)\theta}}(D)}^{\frac{2}{n+2}} 
	+ \knorm{u}_{L^{\frac{2}{\theta}}(D)}^{\frac{2}{n}} +
\knorm{f}_{L^{\frac{n}{\theta}}(D)} \big) \\
	& \leq  \big( \knorm{D \dq u \, \eta}_{L^2(D)}^{1+\theta} 
	+ \knorm{\dq u \, D\eta}_{L^2(D)}^{1+\theta} \big) 
	\\
	& \qquad {} \times c(n,D,L) \, \knorm{\dq u \, \eta}_{L^2(D)}^{1-\theta}
\, 
	\big( \knorm{Du}_{L^{\frac{2n}{(n+2)\theta}}(D)}^{\frac{2}{n+2}} 
	+ \knorm{u}_{L^{\frac{2}{\theta}}(D)}^{\frac{2}{n}} +
\knorm{f}_{L^{\frac{n}{\theta}}(D)} \big)
\intertext{for every $\theta \in (0,1)$; in the two-dimensional case $n=2$,
$\frac{2n}{n-2}$ shall be interpreted as any arbitrary number greater than $2$.
We note that we have omitted the step of passing from the shifted to the
original domain in the first inequality and we have applied $c d^{\theta} +
c^{\theta} d \leq c^{1+\theta} + d^{1+\theta} $ for all $c,d \geq 0$ to get from
the first to the second inequality. To estimate further we choose $\theta=
\frac{n}{n+2}$, according to the integrability assumptions of $u$ (using the
embedding given in~\eqref{embedding}), $Du$ and $f$. Thus, Young's inequality
gives}
\lefteqn{\hspace{-0.75cm} \babs{ \sp{D (\dq u \, \eta^2)}{{\cal B} (h)
}_{L^2(D)} }}\\
	& \leq  \epsilon \, \knorm{D \dq u \, \eta}_{L^2(D)}^2  
	+ c(n,D,L,\knorm{D\eta}_{L^{\infty}(D)},\epsilon) \, \big(\knorm{\dq u
\, \eta}_{L^2(D)}^{2} 
	+ 1 \big) \\
	& \qquad {} \times \big( \knorm{Du}_{L^2(D)}^{2} +
	\knorm{u}_{L^{\frac{2(n+2)}{n}}(D)}^{\frac{2(n+2)}{n}} +
\knorm{f}_{L^{n+2}(D)}^{n+2} \big) \,.
\end{align*}
Finally, using Young's inequality and standard properties of finite difference
quotients, we estimate the last term involving ${\cal C}(h)$ by
\begin{align*}
\lefteqn{\hspace{-0.5cm} \babs{ \sp{D (\dq u \, \eta^2)}{{\cal C} (h) }_{L^2(D)}
}} \\
	& \leq  L \, \big(\knorm{D \dq u \, \eta}_{L^2(D)} + 2 \knorm{\dq u \,
D\eta}_{L^2(D)} \big) 
	\big( \knorm{Du}_{L^2(D)} 
	+ \knorm{u}_{L^{\frac{2(n+2)}{n}}(D)}^{\frac{n+2}{n}} +
\knorm{f}_{L^{4}(D)}^2 \big) \\
	& \leq  \epsilon \, \knorm{D \dq u \, \eta}_{L^2(D)}^2  + 
	c(D,L,\knorm{D\eta}_{L^{\infty}(D)},\epsilon) \, \big( 1+
\knorm{Du}_{L^2(D)}^2 
	+ \knorm{u}_{L^{\frac{2(n+2)}{n}}(D)}^{\frac{2(n+2)}{n}} +
\knorm{f}_{L^{n+2}(D)}^{n+2} \big) \,.
\end{align*}
Now we have estimated all terms coming from the integral involving $\dq
A(x,t,u,Du)$. Next we study the last integral in equation~\eqref{higher1}. Employing the
properties~\eqref{GV-H}, we find
\begin{equation*}
\knorm{\dq H(s,Du(s)) \, \eta}_{L^2(D)} 
	\leq  \knorm{f_H^{\frac{a}{a-2}}(s) +  \kabs{Du(s)} \, \eta}_{L^2(D)} + L_H \knorm{D \dq u(s) \, \eta}_{L^2(D)}\,.
\end{equation*}
Before summarizing the previous estimates for the single terms, we introduce,
for ease of notation, the function
\begin{equation}
\label{higher2}
 G'(u,f)(s) :=  1 + \knorm{Du(s)}_{L^2(D)}^2 
	+ \knorm{u(s)}_{L^{\frac{2(n+2)}{n}}(D)}^{\frac{2(n+2)}{n}} 
	+ \knorm{f(s)}_{L^{a}(D)}^{a}
\end{equation}
with  $s \in (0,T)$, which belongs to $L^1(0,T)$ almost surely. Note that by definition we have $G'(u,f) \geq G_0(u,f)$ with $G_0(u,f)$ denoting the function introduced in Lemma~\ref{lemma_apriori}. Combining the latter estimates with the decomposition given in~\eqref{def-ABC}, using standard properties for finite difference quotients and choosing $\epsilon=\epsilon(\kappa,\nu,L_H)$ sufficiently small, we hence infer from~\eqref{higher1} that for every $\omega \in \Omega'$ there holds 
\begin{align}
\label{higher3}
\lefteqn{ \hspace{-0.5cm} \knorm{u_k'(t)}^2_{L^2(D)} + c^{-1}(L_H,\kappa,\nu)  \int_0^t \knorm{D \dq u(s)
\, \eta}_{L^2(D)}^2 \ds } \nonumber
	\\
	& \leq  \knorm{\dq u (0)\, \eta}^2_{L^2(D)} + c' \int_0^t 
	\big(\knorm{\dq u(s) \, \eta}_{L^2(D)}^{2} 
	+ 1 \big) \, G'(u,f)(s) \ds \nonumber \\
	& \quad {} + 2 \int_0^t \sp{ u_k'(s) \, \eta}{ \dq H(\cdot,s,Du(s)) \,
dB_s}_{L^2(D)} 
	+ 2 \int_0^t \knorm{f_H^{\frac{a}{a-2}}(s)}_{L^2(D)}^2 \ds \,, 
\end{align}
and the constant $c'$ depends only on $n,D,L,L_H,d',\kappa$ and $\nu$. Here we assume $c' \geq c_0$ with $c_0$ denoting the constant given in Lemma~\ref{lemma_apriori}. 
This is the preliminary pathwise estimate on the finite difference quotients (which however involves the stochastic integral) and concludes the Step~1. 

\textbf{Step~2.} Before passing to the expectation value as described in
the beginning we still need the Gronwall-type argument, similarly as in the proof of Lemma~\ref{lemma_apriori}. However, we here observe that the second integral on the right-hand side is in general not known to be finite, but the first factor of the integrand ``almost'' happens to appear in the sum of its left-hand side (in the sense that $u_k'$ differs from $\dq u \, \eta $ only on a negligible set). Hence, to get rid of this possibly uncontrollable term we apply the one-dimensional It\^o-formula (recalling $a > n+2$), and we obtain
\begin{align}
\label{higher4}
\lefteqn{ \hspace{-0.5cm} e^{-\int_0^t c' G'(u,f)(\tilde{s}) \, d\tilde{s}} \,
\knorm{u_k'(t)}^2_{L^2(D)}  + c^{-1} \int_0^t e^{-\int_0^s c' G'(u,f)(\tilde{s}) \, d\tilde{s} } \knorm{D \dq u(s) \,
\eta}_{L^2(D)}^2 \ds} \nonumber \\
	& \leq c' \int_0^t e^{-\int_0^s c' G'(u,f)(\tilde{s}) \, d\tilde{s}} 
	G'(u,f) \, \big( \knorm{\dq u(s) \, \eta}_{L^2(D)}^{2}  - \knorm{u_k'(s)}^2_{L^2(D)}\big) \ds \nonumber \\
	& \quad {} + \knorm{\dq u (0)\, \eta}^2_{L^2(D)} + 1 + 
	2 \int_0^t e^{-\int_0^s c' G'(u,f)(\tilde{s}) \, d\tilde{s} } \, 
	\sp{ u_k'(s) \, \eta}{ \dq H(\cdot,s,Du(s)) \, dB_s}_{L^2(D)} \nonumber
\\
	& \quad {}+ 2 \int_0^t e^{-\int_0^s c' G'(u,f)(\tilde{s}) \, d\tilde{s}
} 
	\, \knorm{f_H^{\frac{a}{a-2}}(s)}_{L^2(D)}^2 \ds \,.
\end{align}
Here, we again note that the average of the first integral on the right-hand side vanishes, due to the fact that $u_k'$ and $\dq u \, \eta $ coincide on $[0,T] \times \Omega$ except for a set of $\Lm^1 \times P$-measure zero. 

\textbf{Step~3.} We now derive the estimate for the average with weights. In contrast to Lemma~\ref{lemma_apriori}, we now derive in a first step an estimate for the (weighted) average of the $L^2(L^2)$-norm of $D \dq u$ (which proceeds in exactly the same way as before). Then, we use this estimate to get also an upper bound for the (weighted) average of the $L^{\infty}(L^2)$-norm of $\dq u$.  

We first note that identity~\eqref{higher1} implies that the process $\knorm{u_{k}'
(t) }_{L^{2}(D)}^{2}$ has a continuous version in $t$. Now, for every $R>0$, we introduce the random time
\[
\tau_{R}:=\inf \Big\{  t\in [0,T] \colon \int_{0}^{t} \knorm{u_{k}'(s) }_{L^{2}(D)}^{2} \, \knorm{ \triangle_{k,h} H(s,Du(s)) \eta }_{L^{2}(D)}^{2} \ds > R \Big\}
\]
with $\tau_{R}=T$ when the set is empty. Notice that
\[
\int_{0}^{T}\left\Vert u_{k}^{\prime}(s)\right\Vert_{L^{2}(D)
}^{2}\left\Vert \triangle_{k,h}H(s,Du(s))
\eta\right\Vert _{L^{2}(D)}^{2} \ds < \infty
\]
with probability one, because of the property $u\in V_{0}^{2}(D_{T}%
,\mathbb{R}^{N})$ and the assumption~\eqref{GV-H}$_2$ on $H$. 
Hence, when $R \rightarrow \infty$, $\tau_{R}$ is eventually equal to $T$, 
with probability one. In particular, $P(\lim_{R\rightarrow\infty}\tau_{R}=T)=1$ and for every $t \in [0,T]$ we have
\[
P \Big(  \lim_{R\rightarrow\infty}\Vert u_{k}^{\prime}(t\wedge\tau_{R}%
)\Vert_{L^{2}(D)}^{2}=\Vert u_{k}^{\prime}(t)\Vert_{L^{2}(D)}^{2} \Big) = 1 \,.
\] 

\textbf{Step~3a.} We compute inequality~\eqref{higher4} at time $t\wedge\tau_{R}$ and get
\begin{align*}
&  \hspace{-0.5cm} e^{-\int_{0}^{t\wedge\tau_{R}}c'G'(u,f)\,(\tilde
{s})  d\tilde{s}} \knorm{ u_{k}'(t\wedge\tau_{R}
)}_{L^{2}(D)}^{2} + c^{-1} \int_{0}^{t\wedge\tau_{R}}e^{-\int_{0}%
^{s}c'G'(u,f)\,(\tilde{s})  d\tilde{s}}\Vert D \triangle
_{k,h} u(s)  \,\eta\Vert_{L^{2}(D)}^{2} \ds\\
&  \leq c' \int_0^{t\wedge\tau_{R}} e^{-\int_0^s c' G'(u,f)(\tilde{s}) \, d\tilde{s}} 
	G'(u,f) \, \big( \knorm{\dq u(s) \, \eta}_{L^2(D)}^{2}  - \knorm{u_k'(s)}^2_{L^2(D)}\big) \ds \\
  & \quad {} + \Vert\triangle_{k,h}u_0\,\eta\Vert_{L^{2}(D)}^{2}+1+2\int
_{0}^{t\wedge\tau_{R}}e^{-\int_{0}^{s}c'G'(u,f)\,(  \tilde{s})
\,d\tilde{s}}\, \sp{ u_{k}^{\prime}(s) \eta}{\triangle_{k,h}H(
s,Du(s)) \, dB_{s} }_{L^{2}(D)}\\
& \quad {}+ 2 \int_{0}^{t\wedge\tau_{R}}e^{-\int_{0}^{s}c'G'(u,f)\,(  \tilde
{s})  \,d\tilde{s}}\,\Vert f_{H}^{\frac{a}{a-2}}(s)
\Vert_{L^{2}(D)}^{2} \ds\,.
\end{align*}
Now we have
\begin{multline*}
\quad \int_{0}^{t\wedge\tau_{R}}e^{-\int_{0}^{s}c'G'(u,f)\,(  \tilde
{s}) \,d\tilde{s}} \, \sp{u_{k}^{\prime}(s)\eta}{\triangle_{k,h} H(s,Du(s)) \, dB_{s}}_{L^{2}(D)} \\
=\int_{0}^{t}e^{-\int_{0}^{s} c' G'(u,f)\,(\tilde{s})
\,d\tilde{s}} \, 1_{s\leq\tau_{R}} \sp{u_{k}^{\prime}(s)\eta}
{\triangle_{k,h} H(s,Du(s)) \, dB_{s}}
_{L^{2}(D)} \quad
\end{multline*}
and
\begin{multline*}
\quad \int_{0}^{t} e^{-2 \int_{0}^{s}c'G'(u,f)\,(\tilde{s})
\,d\tilde{s}} 1_{s\leq\tau_{R}} \knorm{ u_{k}^{\prime}(s)}_{L^{2}(D)}^{2}
\knorm{ \triangle_{k,h} H(s,Du(s)) \eta }_{L^{2}(D)}^{2} \ds\\
 \leq \int_{0}^{t\wedge\tau_{R}} \knorm{ u_{k}^{\prime}(s) }_{L^{2}(D)}^{2}
 \knorm{ \triangle_{k,h} H(s,Du(s)) \eta }_{L^{2}(D)}^{2} \ds \leq R \quad
\end{multline*}
by definition of $\tau_{R}$. Thus, the stopped stochastic integral above is a
martingale, hence with expected value zero. This implies (using that $u_k'$ equals $\dq u \, \eta $ outside a set of $\Lm^1 \times P$-measure zero)
\begin{multline*}
\quad E \Big[  e^{-\int_{0}^{t\wedge\tau_{R}}c'G'(u,f)\,(
\tilde{s})  d\tilde{s}} \knorm{u_{k}'(t\wedge\tau_{R})}_{L^{2}(D)}^{2} \Big]  +c^{-1} E\Big[  \int_{0}^{t \wedge \tau_{R}}e^{-\int_{0}^{s}c'G'(u,f)\,(\tilde{s})  d\tilde{s}
} \knorm{ D \dq u(s)  \,\eta }_{L^{2}(D)}^{2} \ds \Big]
\\
\leq \knorm{ \triangle_{k,h}u_0 \,\eta }_{L^{2}(D)}^{2}+1+2 \, E\Big[
\,\int_{0}^{T} \knorm{ f_{H}^{\frac{a}{a-2}}(s) }_{L^{2}(D)}^{2} \ds\, \Big] \, . \quad
\end{multline*}
We apply Fatou's lemma to the first term and monotone convergence theorem to the
second one on the left-hand-side, and we get
\begin{multline*}
\quad E\Big[  e^{-\int_{0}^{t}c'G'(u,f)\,(\tilde{s}) d\tilde{s}} \knorm{u_{k}'(t) }_{L^{2}(D)}^{2} \Big] + c^{-1} E \Big[  \int_{0}^{t}e^{-\int_{0}^{s}c'
G'(u,f)\,(\tilde{s})  d\tilde{s}} \Vert D \dq u(s)  \,\eta\Vert_{L^{2}(D)}^{2} \ds \Big]  \\
\leq \knorm{ \triangle_{k,h}u_0\,\eta }_{L^{2}(D)}^{2}+1+ 2 \, E \Big[
\int_{0}^{T} \knorm{ f_{H}^{\frac{a}{a-2}}(s) }_{L^{2}(D)}^{2} \ds \, \Big] \quad
\end{multline*}
for every $t \in [0,T]$. This proves one of the two bounds claimed
by the Lemma. 

\textbf{Step~3b.} It is almost our final estimate except that we need the supremum
in time inside the first expected value, and thus we have to repeat the
previous computations by means of martingale inequalities. The previous
estimate (as well as the a~priori estimate from Lemma~\ref{lemma_apriori}) 
will be used in the next one; we found it convenient to proceed in
two steps. From the stopped inequality above we have
\begin{align*}
&  \hspace{-0.5cm} E \Big[  \sup_{t\in [0,T]  }e^{-\int
_{0}^{t\wedge\tau_{R}}c'G'(u,f)\,(\tilde{s})  d\tilde{s}
} \knorm{ u_{k}^{\prime}(t\wedge\tau_{R}) }_{L^{2}(D)}^{2}
\Big]  \\
& \leq \Vert\triangle_{k,h}u_0\,\eta\Vert_{L^{2}(D)}^{2}+1+2 \, E \Big[ \int_{0}
^{T} \knorm{ f_{H}^{\frac{a}{a-2}}(s) }_{L^{2}(D)}
^{2} \ds \Big]  \,.\\
& \quad {} +2 \, E \Big[  \sup_{t\in [0,T] } \Babs{ \int_{0}
^{t}e^{-\int_{0}^{s}c'G'(u,f)\,(\tilde{s})  \,d\tilde{s}
}\,1_{s\leq\tau_{R}}\left\langle u_{k}^{\prime}(s)\eta,\,\triangle
_{k,h}H(s,Du(s))  dB_{s}\right\rangle _{L^{2}
(D)} } \, \Big]  .
\end{align*}
We apply again Fatou's lemma to the expected value on the left-hand-side. The
last term on the right-hand-side is estimated, by means of the
Burkholder-Davis-Gundy inequality, by
\begin{align*}
& \hspace{-0.5cm} C \, E\Big[  \Big(  \int_{0}^{T}e^{-2\int_{0}^{s}
c'G'(u,f)\,(\tilde{s})  \,d\tilde{s}}\,1_{s\leq\tau_{R}
}\left\Vert u_{k}^{\prime}(s)\right\Vert _{L^{2}(D)}^{2}\left\Vert
\triangle_{k,h}H(s,Du(s))  \eta\right\Vert
_{L^{2}(D)}^{2} \ds\Big)^{1/2} \Big]  \\
& =C \, E \Big[  \Big(  \int_{0}^{T}e^{-2\int_{0}
^{s\wedge\tau_{R}}c'G'(u,f)\,(\tilde{s})  \,d\tilde{s}
}\,1_{s\leq\tau_{R}}\left\Vert u_{k}^{\prime}(s\wedge\tau_{R})\right\Vert
_{L^{2}(D)}^{2}\left\Vert \triangle_{k,h}H(s,Du(s)) \eta\right\Vert _{L^{2}(D)}^{2} \ds \Big)  ^{1/2} \Big]  \\
& \leq C \, E \Big[  I_{1}^{1/2}I_{2}^{1/2}\Big]  \leq
\frac{1}{2} E \big[  I_{1}\big]  +\frac{C^{2}}{2} 
E \big[  I_{2} \big]
\end{align*}
where
\begin{align*}
I_{1}  & =\sup_{t \in [0,T]  }e^{-\int_{0}^{t\wedge\tau_{R}}
c'G'(u,f)\,(\tilde{s}) \,d\tilde{s}} \knorm{
u_{k}'(t\wedge\tau_{R}) }_{L^{2}(D)}^{2} \,,  \\
I_{2}  & =\int_{0}^{T}e^{-\int_{0}^{s}c'G'(u,f)\,(\tilde{s})
\,d\tilde{s}}\knorm{ \triangle_{k,h}H(s,Du(s))  \eta }_{L^{2}(D)}^{2} \ds.
\end{align*}
Hence, we have proved
\[
\frac{1}{2}E\big[  I_{1} \big]  \leq \Vert\triangle_{k,h}u_0\,\eta
\Vert_{L^{2}(D)}^{2}+1+2 \, E \Big[  \int_{0}^{T} \knorm{ f_{H}^{\frac{a}{a-2}
}(s) }_{L^{2}(D)}^{2} \ds \Big]  +\frac{C^{2}}{2} 
E\big[  I_{2}\big]  .
\]
From the estimate above for $\knorm{ \triangle_{k,h} H(s,Du(s)) \eta}_{L^{2}(D)}^{2}$ 
and the estimate of Step~3a, we know that $E[I_{2}]$ is bounded from
above via 
\begin{align*}
E[I_{2}] & \leq 4 \, E \Big[ \int_{0}^{T} e^{-\int_{0}^{s}c'G'(u,f)\,(\tilde{s})
\,d\tilde{s}} \knorm{ Du(s) \, \eta}_{L^2(D)}^2 \ds \Big] \\
  & \quad {} + c \, \Big(  \knorm{ \triangle_{k,h}u_0\,\eta }_{L^{2}(D)}^{2}+1+ E \Big[
\int_{0}^{T} \knorm{ f_{H}^{\frac{a}{a-2}}(s) }_{L^{2}(D)}^{2} \ds \, \Big] \Big) \,.
\end{align*}
With a suitable choice of $D_0$ (in dependency of $D$ and $D'$) and keeping in mind $c' G'(u,f) \geq c_0 G_0(u,f)$ by construction, the first average on the right-hand side of the last inequality is bounded due to Lemma~\ref{lemma_apriori}. Hence we get the bound for $E[I_{1}]$ as asserted in the statement of the lemma. Since $h$ was arbitrary and $\eta=1$ on $D^{\prime}$, the proof is complete.
\end{proof}
\end{lemma}

\begin{remark}
For SPDEs having first space derivatives of the solution in the
coefficient of the noise, the most general condition for existence of
solutions in $L^2$, which becomes also a condition for an improvement of 
$W^{k,2}$-regularity, is more precise than just the control on the 
Lipschitz constant of $H$ expressed by the statement of Lemma~\ref{Diff-quotients}; 
see~\cite{PARDOUX75, KRYROZ79}. However, when we go to $W^{k,p}$-regularity with 
$p>2$, the computations are too involved and the algebraic simplicity of the 
condition of~\cite{PARDOUX75, KRYROZ79} seems to be lost. 
For this reason we have simplified the estimate also for $p=2$.
\end{remark}

Applying Theorem~\ref{thm-derivatives} with $(p,q) = (2,2)$ and 
$(p,q) = (2,\infty)$ and summing over $k \in \{1,\ldots,n\}$ we then infer 
from the previous lemma that second order spatial derivatives of $u$ exist
almost surely. We should 
note that this result does not extend up to the boundary of $D$ since the 
constant $c'$ blows up for $\dist(D',D) \searrow 0$, but the result holds on
any fixed subset $D' \Subset D$.

\begin{corollary}
\label{cor-second-der}
Let $u \in V^2(D_T,\R^N)$ be a weak solution under the assumptions of the
Lemma~\ref{Diff-quotients}. Then there holds $Du \in V^2(D'_T,\R^N)$
with probability one, and
\begin{multline*}
\quad \E \, \Big[ \, \sup_{t \in (0,T)} e^{-\int_0^t c' G'(u,f) \ds } \knorm{D
u}^2_{L^2(D')} + \int_0^T e^{-\int_0^t c' G'(u,f) \ds }
\knorm{D^2u}_{L^2(D')}^2 \dt \Big] \\
	 \leq c' \Big( \knorm{D u_0}_{L^2(D)}^2 + 1 
	+ \E \big[ \knorm{f_H^{\frac{a}{a-2}}}_{L^2(D_T)}^2 \big] \Big)\quad
\end{multline*}
for the constant $c'$ from Lemma~\ref{Diff-quotients}. Moreover, we have for all $k \in \{1,\ldots,n\}$
\begin{equation*}
e^{- \frac{1}{2} \int_0^t c' G'(u,f) \ds } \dq Du \, \rightarrow \, e^{- \frac{1}{2}
\int_0^t c' G'(u,f) \ds } D_k Du \quad \text{weakly in } L^2(D_T' \times \Omega ) \,. 
\end{equation*}
\end{corollary}

%****************************************************************************
%		Iteration
%****************************************************************************

\subsection{Iteration}

In the next step we want to iterate the procedure from the previous section, in
a way such that we do not only know the spacial gradient $Du$ to belong to the
space $V^2$ with probability one, but that we get this result also for certain
powers of $Du$. For convenience we introduce the function 
\begin{equation*}
W_q \colon \R^k \to \R \quad \text{defined by} \quad W_q(\xi)  := \kabs{\xi}^q
\end{equation*}
for every $q \geq 0$.
We start by briefly describing the strategy how this regularity improvement
is achieved. First we observe from the results in Section~\ref{section_second_derivatives} that there exists a subset of $\Omega$ of full measure on which $Du$ belongs to $V^2_{\rm{loc}}(D_T)$, hence we can now take advantage of higher integrability properties for $u$ and $Du$. This shall be done with the following (but formal) iteration scheme:
\begin{align*}
W_{q_j}(Du) \in V^2 \quad &\longrightarrow \quad Du \in L^{2q_j \frac{n+2}{n}}
\text{ and }
u \in L^{2q_j (\frac{n+2}{n})^2} \cap L^{\infty}(L^{2q_j \frac{n}{n-2}})\\
	&\longrightarrow  
	\quad D_uA(x,t,u,Du) \in L^{\min\{q_j \frac{(n+2)^2}{n},a\}}  \text{ and
}
	D_x A(x,t,u,Du) \in L^{\min\{2q_j \frac{n+2}{n},\frac{a}{2}\}}
\\[0.15cm]
	&\longrightarrow  
	\quad W_{q_{j+1}}(Du) \in V^2
\end{align*}
for a sequence $\{q_j\}_{j \in \N}$ of numbers $q_j \geq 1$ for all $j \in \N$.
The first implication indeed follows from the Sobolev's embedding for the space
$V^{2,p}$, the second one from the growth conditions on the vector field $A$, the
third one from the iteration (and a convergence result concerning finite
difference quotients). After a finite number of steps we then arrive at a final
(maximal) higher integrability exponent, which essentially reflects how close
the vector field $A$ is to the Laplace system. This should be understood in the
following sense: the closer $\nu$ is to one (note that $\nu = 1$ corresponds to
the case $A(x,t,u,z) = z$ plus potential lower order terms), the more
integrability for $Du$ can be gained in the iteration and the better will be the
final regularity properties of $u$. Finally, we note that in every step of the
iteration we will have to reduce the radius of the parabolic cylinder and we
will also have to restrict ourselves to smaller subsets of $\Omega$.
Nevertheless, the higher integrability results will always be true on sets of
probability one.

We now start with some preliminary remarks and consider again the equation
\eqref{equation}
\begin{equation*}
du \, = \, \diverg A (x,t,u,Du) \dt + H(x,t,Du) \, dB_t 
\end{equation*}
in $D_T$. We observe that $\diverg A(x,t,u,Du)$ is well defined in view
of the regularity assumptions~\eqref{GV-A} and the existence of second
order spatial derivatives, see Corollary~\ref{cor-second-der}. More precisely,
it is easy to check that for every weak solution $u \in V^2_0(D_T,\R^N)$  we have:
$\diverg A(x,t,u,Du) \in L^{2}_{\rm{loc}}(D',\R^N)$ with probability one, and the
equation above holds for $\Ln$-almost every $x \in D'$ for $\Lm^1 \times P
$-almost all $(t,\omega) \in (0,T) \times \Omega $. Hence, we can now work
immediately with this equation without passing to its weak formulation. 

In the next lemma we will provide the main step of the iteration argument:
\begin{lemma}
\label{diff-quotients-iteration}
Let $u \in V^2(D_T,\R^N)$ be a weak solution to the initial boundary value problem
to~\eqref{equation} under the assumptions~\eqref{GV-A},~\eqref{GV-H} and with
initial values $u(\ccdot,0) = u_0(\cdot) \in W^{1,2}(D,\R^N)$, and assume that
\begin{equation*}
E \Big[ \knorm{ Y_p^p \, W_{p }(Du) }_{L^{2  \frac{n+2}{n}}(D'_T)}^{\frac{2}{p}} \Big] \leq C_p < \infty 
\end{equation*}
for some $p \geq 1$, a set $D' \subset D$, and $Y_p \colon [0,T] \times \Omega  \to (0,1]$ given by $Y_p(t,\omega) = {\rm exp}(- \int_0^t G_p(s,\omega) ds)$ for some function $G_p$ which is in $L^1(0,T)$ with probability one. Let
$D'' \subset D'$ with $d := \dist(D'',\partial D') > 0$. Then for every number
$q \geq 1$ satisfying
\begin{equation}
\label{restrict-q}
q  \leq \min\Big\{ p \frac{n+2}{n}, 1+ p \frac{n+2}{n} \frac{a-4}{a}, \frac{a-2}{2}\Big\} \quad \text{and} \quad L_H^2 <
\frac{1}{\kappa (q - \frac{1}{2})} \Big( \Big[ 1 - \Big( \frac{q-1}{q}\Big)^2
\Big]^{\frac{1}{2}} - \Big[ 1 - \nu^2 \Big]^{\frac{1}{2}} \Big) \,,
\end{equation}
all initial values $u_0 \in W^{1,2q}(D,\R^N)$, and every $k \in \{1,\ldots,n\}$ there holds
\begin{multline*}
\quad \sup_{\kabs{h} < d} \E \, \Big[ \Big( \sup_{t \in (0,T)}  \knorm{Y^q_q \, W_q(\dq
u)}^{2}_{L^2(D'')} 
	+ \int_0^T \knorm{Y_q^q \,  DW_q(\dq u)}_{L^2(D'')}^2 \dt \Big)^{\frac{1}{q}} \Big] \\
	\leq  c \, \Big( \knorm{W_q(D_k u_0)}^{2}_{L^{2}(D)} + 1 + 
	\E \big[ \knorm{f_H(s) )}_{L^a(D_T)}^a  \big] \Big)^{\frac{1}{q}} \,, \quad
\end{multline*}
for $Y_q \colon [0,T] \times \Omega \to (0,1]$ given by $Y_q(t,\omega) = {\rm exp}(- \int_0^t G_q(s,\omega) ds)$ for some function $G_q$ which is in $L^1(0,T)$ with probability one, and a constant $c$ depending only on
$n,p,D,T,L,L_H,d,\kappa,\nu$, and $C_p$.
\end{lemma}

\begin{remarks}
In the case of additive noise (with $L_H = 0$) the second condition
\eqref{restrict-q} for the restriction on the integrability exponent $q$ reduces
to the inequality $q < \frac{1}{1-\nu}$. For multiplicative noise instead, the
right-hand side in the second inequality~\eqref{restrict-q} is decreasing in $q$ 
(note that for $q=1$ it just reproduces the condition required in Lemma~\ref{Diff-quotients}) and allows the following interpretation. Obviously, the previous restriction $q < \frac{1}{1-\nu}$ for additive noise remains valid, and in fact 
the more multiplicative noise is considered (in the sense that $L_H$ should not be too small), the smaller will be the maximal integrability exponent which still 
satisfies both inequalities in~\eqref{restrict-q}. For this reason multiplicative 
noise might destroy some regularity in form of integrability of the gradient $Du$.

Moreover, we comment on the scaling of the hypothesis and the assertion with respect to $u$ and $u_0$, respectively, in order to avoid confusion. In view of the definition of $W_p$ it is easy to see that Lemma~\ref{diff-quotients-iteration} is stated in a way such that an weighted average of a quadratic quantity in $Du$ gives an information about the weighted average of a quadratic quantity in $\dq u$. In this sense, the scaling is the natural one.
\end{remarks}

\begin{proof}
We now follow the line of arguments from the proof of Lemma~\ref{Diff-quotients} (and of Lemma~\ref{lemma_apriori}), but this time we will estimate powers of the difference quotients $\dq u$.

\textbf{Step~1.} We consider $k \in \{1,\ldots,n\}$ arbitrary, $h \in \R$ with $\kabs{h} < d$,
and $\eta \in C^{\infty}(D',[0,1])$ a standard cut-off function satisfying $\eta
\equiv 1$ on $D''\Subset D'$ and $\kabs{D\eta} \leq c(d)$. We first observe
that, by the integrability assumption on $Du$ and the integrability assumption on $G_p$ (which implies strict positivity of $\inf_{t \in [0,T]} Y_p$ for $P$-almost every $\omega$), with probability one we have
\begin{equation*}
u \in L^{2p (\frac{n+2}{n})^2}_{\rm{loc}}\big(D'_T,\R^{N}\big) \cap
L^{\infty}_{\rm{loc}}\big(0,T; L^{2p \frac{n}{n-2}}(D',\R^N)\big) \,.
\end{equation*}
Furthermore, due to the restriction on $q$, it is guaranteed that
$W_q(Du)$ belongs to $L^2$ locally on $D'_T$ with probability one. For almost every (fixed) $x \in D$ we first consider finite differences in direction $e_k$
and stepsize $h$ of the differential equation~\eqref{equation}, i.e.
\begin{equation*}
d \, \eta^{\frac{1}{q}} \, \dq u(x,t)  =  \eta^{\frac{1}{q}} \, \diverg  \dq
A (x,t,u,Du) \dt +  \eta^{\frac{1}{q}} \, \dq H(x,t,Du) \, dB_t
\end{equation*}
in $(0,T)$ for $q\geq 1$. We next introduce (because of technical reasons) for
$K>0$ the approximating function $T_{q,K}(\cdot)$ of class $C^2$ according to Lemma
\ref{tech-kalita-2}, and we recall that $T_{q,K}$ satisfies in particular the
polynomial growth conditions $T_{q,K}(t) = t^{2q}$ for all $t \leq K$ and
$T_{q,K}(t) \leq c(q) K^{2q-2} t^{2}$ for all $t \in \R$. Employing the
one-dimensional It\^o formula (note that $\diverg A(x,t,u,Du)$ is as a
composition of $\mathcal{F}_t$-adapted functions again $\mathcal{F}_t$-adapted)
from Theorem~\ref{Ito-I}, applied with $g(t,u(x,t))= \eta^2 \, T_{q,K}
(\kabs{\dq  u(x,t)})$, we obtain the identity
\begin{align*}
\lefteqn{ \hspace{-0.5cm} d \big( \eta^2  T_{q,K} (\kabs{\dq u(x,t)}) \big) } \\
	&  =  \eta^{2} T'_K (\kabs{\dq u(x,t)}) \kabs{\dq u(x,t)}^{-1} 
	\sp{\dq u(x,t)}{\diverg \dq A (x,t,u,Du)}_{\R^N} \dt  \nonumber \\ 
	& \quad {}+ \frac{1}{2}  \eta^{2} \, \big[ T_{q,K}''(\kabs{\dq u(x,t)})
\kabs{\dq u(x,t)} 
	-  T_{q,K}'(\kabs{\dq u(x,t)}) \big] \nonumber \\
	& \qquad {} \times \kabs{\dq u(x,t)}^{-3}  \kabs{\sp{\dq u(x,t)}{\dq
H(x,t,Du)}}^2 
	\dt \nonumber \\
	& \quad {}+ \frac{1}{2} \eta^{2} \, T_{q,K}'(\kabs{\dq u(x,t)}) \,
\kabs{\dq u(x,t)}^{-1} \, 
	\kabs{\dq H(x,t,Du)}^2 \dt \nonumber \\
	& \quad {}+  \eta^{2} \, T'_{q,K} (\kabs{\dq u(x,t)}) \kabs{\dq
u(x,t)}^{-1} \, 
	\sp{\dq u(x,t)}{\dq H(x,t,Du)dB_t}_{\R^N}.
\end{align*}
In order to prove the assertion of the lemma, we start with a simple observation
concerning the terms involving $\dq H(x,t,Du)$. Taking into account the
properties of the function $T_{q,K}$, see Lemma~\ref{tech-kalita-2}, we estimate
\begin{multline*}
\big[ T_{q,K}''(\kabs{\dq u(x,t)}) \kabs{\dq u(x,t)} 
	-  T_{q,K}'(\kabs{\dq u(x,t)}) \big] \, \kabs{\dq u(x,t)}^{-3}  
	\kabs{\sp{\dq u(x,t)}{\dq H(x,t,Du)}}^2 \\
	\leq  2(q-1) \, T_{q,K}'(\kabs{\dq u(x,t)}) \, \kabs{\dq u(x,t)}^{-1}
\, \kabs{\dq H(x,t,Du)}^2 \,.
\end{multline*}
We next introduce the abbreviation 
\begin{equation*}
V(\xi)  :=  T'_{q,K} (\kabs{\xi}) \, \kabs{\xi}^{-1} \xi
\end{equation*}
for all $\xi \in \R^N$, and we note $\kabs{V(\xi)} = T'_{q,K} (\kabs{\xi})$. Now we integrate
over $x \in D$, and then we apply Fubini which due to the truncation procedure
is always allowed, see Lemma~\ref{tech-kalita-2} ii). Applying the integration
by parts formula, we hence obtain
\begin{align}
\label{ineq-it-prelim}
\lefteqn{\hspace{-0.5cm} \bnorm{ (T_{q,K} \kabs{\dq u(t)})^{\frac{1}{2}} \eta
}^2_{L^2(D)} 
	+ \int_0^t \sp{D \big(  V(\dq u(s)) \, \eta^2 \big)}
	{\dq A (\cdot,s,u,Du)}_{L^2(D)} \ds} \nonumber \\
	& \leq  \bnorm{ (T_{q,K} \kabs{\dq u_0})^{\frac{1}{2}} \eta
}^2_{L^2(D)} \nonumber \\
	& \quad {}+ (q - 2^{-1}) \int_0^t 
	\bnorm{T_{q,K}'(\kabs{\dq u(s)})^{\frac{1}{2}} \, 
	\kabs{\dq u(s)}^{-\frac{1}{2}} \, \dq H(\cdot,s,Du) \, \eta}_{L^2(D)}^2 \ds
\nonumber \\
	& \quad {}+ \int_0^t  
	\sp{  V(\dq u(x,s)) \,\eta^2 }{\dq H(\cdot,s,Du) \, dB_s}_{L^2(D)}
\end{align}
Now the second term on the left-hand side of this inequality shall be
estimated. Using the decomposition introduced in~\eqref{def-ABC} and applying
Lemma~\ref{tech-kalita-2}, we first find for every $\epsilon > 0$:
\begin{align*}
\lefteqn{\hspace{-0.25cm} \sp{D\big( V(\dq u(s)) \,\eta^2 \big)}{{\cal A} (h)
}_{L^2(D)} } \\
	& =  \kappa^{-1} \sp{D\big( V(\dq u(s)) \, \eta^2 \big)}{ D \dq u
}_{L^2(D)} 
	- \kappa^{-1} \sp{D\big( V(\dq u(x,s)) \, \eta^2 \big)}
	{D \dq u - \kappa \, {\cal A} (h) }_{L^2(D)} \\
	& \geq  \kappa^{-1} \knorm{D \big(V(\dq u(s))\big) \cdot D \dq u \,
\eta^2}_{L^1(D)} 
	- 2 \epsilon \, \knorm{T_{q,K}'(\kabs{\dq u})^{\frac{1}{2}} \, \kabs{\dq
u}^{-\frac{1}{2}}
	 \, D \dq u \, \eta}_{L^2(D)}^2 \\
	& \quad {}  - c(q,\kappa,\epsilon) \, 
	\knorm{\kabs{\dq u }^{q} \, D\eta}_{L^2(D)}^2 
	- \kappa^{-1} \, (1-\nu^2)^{\frac{1}{2}} \, 
	\knorm{D\big(V(\dq u(s))\big)
	\, \kabs{D \dq u } \, \eta^2}_{L^1(D)}  \\
	& \geq  \big( \kappa^{-1} \, \mu^{\frac{1}{2}}(q) - 
	\kappa^{-1} \, (1-\nu^2)^{\frac{1}{2}} - 2 \epsilon \big) \, 
	\knorm{T_{q,K}'(\kabs{\dq u})^{\frac{1}{2}} \, 
	\kabs{\dq u}^{-\frac{1}{2}} \, D \dq u \, \eta}_{L^2(D)}^2 \\
	& \quad {} - c(q,\kappa,\epsilon,\knorm{D\eta}_{L^{\infty}(D)}) \, 
	\knorm{W_{q}(\dq u)}_{L^2(D')}^2 \,. 
\end{align*}
We observe from the definition of $\mu(q)$ and the second bound in
\eqref{restrict-q} on $q$ that the factor $\mu^{\frac{1}{2}}(q)  -
(1-\nu^2)^{\frac{1}{2}}$ appearing in the previous inequality is always strictly
positive. Now, for the second term in the decomposition~\eqref{def-ABC} we obtain via the inequalities $T''_{q,K}(t) t^2 \leq c(q) T'_{q,K}(t) t \leq c(q) T_{q,K}(t)$ on $\R^+$ and the Sobolev-Poincar\'{e} embedding (applied on every time-slice):
\begin{align*}
\lefteqn{\hspace{-0.25cm} \babs{ \sp{D \big( V(\dq u(s)) \, \eta^2\big)}{{\cal B}
(h) }_{L^2(D)} }}\\
	& \leq  c(L,q)  \, \big( \knorm{T_{q,K}'(\kabs{\dq u})^{\frac{1}{2}}
\, 
	\kabs{\dq u}^{-\frac{1}{2}} D \dq u \, \eta}_{L^2(D)} 
	+ \knorm{T_{q,K}'(\kabs{\dq u})^{\frac{1}{2}} \, 
	\kabs{\dq u}^{\frac{1}{2}} \, D\eta}_{L^2(D)} \big) \\
	& \qquad {} \times \knorm{T_{q,K}'(\kabs{\dq u})^{\frac{1}{2}} 
	\kabs{\dq u}^{-\frac{1}{2}} \, {\cal B} (h)  \, \eta}_{L^2(D)} \\
	& \leq  c(L,q)  \, \big( \knorm{T_{q,K}'(\kabs{\dq u})^{\frac{1}{2}}
\, 
	\kabs{\dq u}^{-\frac{1}{2}} D \dq u \, \eta}_{L^2(D)} 
	+ \knorm{T_{q,K}(\kabs{\dq u})^{\frac{1}{2}} \, D\eta}_{L^2(D)} \big) \\
	& \qquad {} \times \knorm{T_{q,K}'(\kabs{\dq u})^{\frac{1}{2}} \, 
	\kabs{\dq u}^{-\frac{1}{2}} \dq u \,
\eta}_{L^{\frac{2n}{n-2}}(D)}^{\theta} 
	\, \knorm{T_{q,K}'(\kabs{\dq u})^{\frac{1}{2}} \, 
	\kabs{\dq u}^{-\frac{1}{2}} \dq u \, \eta}_{L^2(D)}^{1-\theta} \\
	& \qquad {} \times \big( \knorm{Du}_{L^{\frac{2n}{(n+2)\theta}}(\supp
\eta)}^{\frac{2}{n+2}} 
	+ \knorm{u}_{L^{\frac{2}{\theta}}(\supp \eta)}^{\frac{2}{n}} 
	+ \knorm{f}_{L^{\frac{n}{\theta}}(D)} \big) \\
	& \leq  \big( \knorm{T_{q,K}'(\kabs{\dq u})^{\frac{1}{2}} \, 
	\kabs{\dq u}^{-\frac{1}{2}} D \dq u \, \eta}_{L^2(D)}^{1+\theta} 
	+ \knorm{T_{q,K}(\kabs{\dq u} )^{\frac{1}{2}} \, D\eta}_{L^2(D)}^{1+\theta} \big) 
	\\
	& \qquad {} \times c(n,D,T,L,q) \, \knorm{T_{q,K}(\kabs{\dq
u})^{\frac{1}{2}} 
	\, \eta}_{L^2(D)}^{1-\theta} \, 
	\big( \knorm{Du}_{L^{\frac{2n}{(n+2)\theta}}(\supp
\eta)}^{\frac{2}{n+2}} 
	+ \knorm{u}_{L^{\frac{2}{\theta}}(\supp \eta)}^{\frac{2}{n}} 
	+ \knorm{f}_{L^{\frac{n}{\theta}}(D)} \big)
\end{align*}
for every $\theta \in (0,1)$. We now choose $\theta = \max\{p^{-1}
(\frac{n}{n+2})^2,\frac{n}{a}\}$, for which the last expression in brackets of
the previous inequality is consequently bounded with probability one, according
to the integrability assumptions on $f, Du$ and the consequences on the
integrability of $u$ explained at the beginning of the proof. Young's inequality
then implies
\begin{align*}
\babs{ \sp{D \big( V(\dq u(x,s)) \,  \eta^2 \big)}{{\cal B} (h) }_{L^2(D)} } 
	& \leq  \epsilon \, \knorm{T_{q,K}'(\kabs{\dq u})^{\frac{1}{2}} \, 
	\kabs{\dq u}^{-\frac{1}{2}} D \dq u \, \eta}_{L^2(D)}^2  \\
	& \quad {}+ c(n,D,T,L,q,\knorm{D\eta}_{L^{\infty}(D)},\epsilon) \, 
	\big(\knorm{T_{q,K}(\kabs{\dq u})^{\frac{1}{2}} 
	\, \eta}_{L^2(D)}^{2} + 1 \big) \\
	& \qquad {} \times \big( 1 + \knorm{Du}_{L^{2p \frac{n+2}{n}}(\supp
\eta)}^{2p \frac{n+2}{n}} +
	\knorm{u}_{L^{2p (\frac{n+2}{n})^2}(\supp \eta)}^{2p (\frac{n+2}{n})^2} 
	+ \knorm{f}_{L^{a}(D)}^{a} \big) \,.
\end{align*}
Finally,  via the bounds for $q$ in terms of $n,p,a$ and $\nu$, the last term in the decomposition involving ${\cal C}(h)$ is estimated with Young's inequality and the well-known estimates for finite difference quotients by
\begin{align*}
\lefteqn{\hspace{-0.25cm} \babs{ \sp{D \big( V(\dq u(s)) \, \eta^2 \big)}{{\cal C}
(h) }_{L^2(D)} }} \\
	& \leq  c \, \big(\knorm{T_{q,K}'(\kabs{\dq u})^{\frac{1}{2}}
	\kabs{\dq u}^{-\frac{1}{2}} D \dq u \, \eta}_{L^2(D)} 
	+ \knorm{T_{q,K}(\kabs{\dq u} )^{\frac{1}{2}} \, D\eta}_{L^2(D)} \big) \\
	& \qquad {} \times \knorm{T_{q,K}'(\kabs{\dq u})^{\frac{1}{2}} 
	\kabs{\dq u}^{-\frac{1}{2}} \, {\cal C} (h)  \, \eta}_{L^2(D)} \\
	& \leq  \epsilon \, \knorm{T_{q,K}'(\kabs{\dq u})^{\frac{1}{2}} \, 
	\kabs{\dq u}^{-\frac{1}{2}} D \dq u \, \eta}_{L^2(D)}^2  \\
	& \quad {} + 
	c(D,T,L,q,\knorm{D\eta}_{L^{\infty}(D)},\epsilon) \, 
	\big( 1+ \knorm{Du}_{L^{2p \frac{n+2}{n}}(\supp \eta)}^{2p
\frac{n+2}{n}} +
	\knorm{u}_{L^{2p (\frac{n+2}{n})^2}(\supp \eta)}^{2p (\frac{n+2}{n})^2} 
	+ \knorm{f}_{L^{a}(D)}^{a} \big)
\intertext{provided that $4q \leq a$. For the general case, one again has to argue more subtle, using the Sobolev embedding on time slices as for the term with ${\cal B}(h)$. With the analogous calculations as before this yields}
\lefteqn{\hspace{-0.25cm} \babs{ \sp{D \big( V(\dq u(s)) \, \eta^2 \big)}{{\cal C}
(h) }_{L^2(D)} }} \\
  & \leq \epsilon \, \knorm{T_{q,K}'(\kabs{\dq u})^{\frac{1}{2}} \, 
	\kabs{\dq u}^{-\frac{1}{2}} D \dq u \, \eta}_{L^2(D)}^2  \\
  & \quad {}+ c(n,D,T,L,q,\knorm{D\eta}_{L^{\infty}(D)},\epsilon) \, 
	\big(\knorm{T_{q,K}(\kabs{\dq u})^{\frac{1}{2}} 
	\, \eta}_{L^2(D)}^{2} + 1 \big) \\
	& \qquad {} \times \big( 1 + \knorm{Du}_{L^{2p \frac{n+2}{n}}(\supp
\eta)}^{2p \frac{n+2}{n}} +
	\knorm{u}_{L^{2p (\frac{n+2}{n})^2}(\supp \eta)}^{2p (\frac{n+2}{n})^2} 
	+ \knorm{f}_{L^{a}(D)}^{a} \big) \,.
\end{align*}
It now still remains to handle the second term on the right-hand side of inequality~\eqref{ineq-it-prelim}. With the assumptions~\eqref{GV-H} on $H$ and Young's inequality, we easily find
\begin{align*}
\lefteqn{\hspace{-0.25cm} \bnorm{T_{q,K}'(\kabs{\dq u(s)})^{\frac{1}{2}} \, 
	\kabs{\dq u(s)}^{-\frac{1}{2}} \, \dq H(\cdot,s,Du) \, \eta}_{L^2(D)}^2 }\\
  & \leq (L_H^2 + \epsilon) \, \knorm{T_{q,K}'(\kabs{\dq u})^{\frac{1}{2}} \, 
	\kabs{\dq u}^{-\frac{1}{2}} D \dq u \, \eta}_{L^2(D)}^2  \\
  & \quad {} +  c(L,L_H,\varepsilon) \, \big( 1 + \knorm{Du}_{L^{2p \frac{n+2}{n}}(\supp
\eta)}^{2p \frac{n+2}{n}} + \knorm{f_H}_{L^{a}(D)}^{a} \big)\,.
\end{align*}
For every $s \in (0,T)$ we now define
\begin{equation}
\label{def-G''}
G''(u,f)(s) :=  \frac{2q}{c''} G_p(s) + 1 + \knorm{Du}_{L^{2p \frac{n+2}{n}}(D')}^{2p
\frac{n+2}{n}} +
	\knorm{u}_{L^{2p (\frac{n+2}{n})^2}(D')}^{2p (\frac{n+2}{n})^2} 
	+ \knorm{f}_{L^{a}(D)}^{a} \,,
\end{equation}
which is a $L^{1}(0,T)$ with probability one. Furthermore, we set $G_q := \frac{1}{2q} c'' G''(u,f) \geq G_p$ which immediately gives $Y_q \leq Y_p$. Then, taking into account
the smallness condition~\eqref{restrict-q}, choosing $\epsilon$ sufficiently
small and combining the previous estimates for the various terms arising in~\eqref{ineq-it-prelim}, we find a preliminary (though still $K$-depending) pathwise estimate
\begin{align*}
\lefteqn{ \hspace{-0.5cm} \knorm{ T_{q,K} (\kabs{\dq
u(t)})^{\frac{1}{2}} \eta }^{2}_{L^{2}(D)} +  c^{-1}(L_H,\kappa,\nu)  \int_0^t 
\knorm{T_{q,K}'(\kabs{\dq u})^{\frac{1}{2}} \, \kabs{\dq u}^{-\frac{1}{2}}
	 \, D \dq u \, \eta}_{L^2(D)}^2 \ds } 
	\\
	& \leq  \knorm{ T_{q,K} (\kabs{\dq u_0})^{\frac{1}{2}} \eta
}^{2}_{L^{2}(D)} 
	+ c'' \int_0^t \big(\knorm{ T_{q,K} (\kabs{\dq
u(t)})^{\frac{1}{2}} \eta}_{L^2(D)}^{2} 
	+ 1 \big) \, G''(u,f) \ds  \\
	& \quad {} + c \int_0^t 
	\bnorm{f_H(s)}_{L^a(D)}^a \ds + \int_0^t  
	\sp{  V(\dq u(x,t)) \, \eta^2}{\dq H(\cdot,s,Du)\,
dB_s}_{L^2(D)}  \,.
\end{align*}

\textbf{Step~2.} We may now apply in a first step It\^o's formula in exactly the same way as before in the derivation of estimate~\eqref{higher4}:
\begin{align*}
\lefteqn{  \hspace{-0.5cm}  e^{-\int_0^t c'' \, G''(u,f) \ds } \, 
\bnorm{  T_{q,K} (\kabs{\dq
u(t)})^{\frac{1}{2}} \eta}^{2}_{L^{2}(D)}  } \\
\lefteqn{ \hspace{-0.5cm} {}+ c^{-1} \int_0^t
	e^{-\int_0^s c'' \, G''(u,f) \, d\tilde{s} } \, 
	\bnorm{T_{q,K}'(\kabs{\dq u})^{\frac{1}{2}} \, \kabs{\dq u}^{-\frac{1}{2}}
	 \, D \dq u \, \eta }_{L^2(D)}^2 \ds  }  \\
	& \leq  \bnorm{T_{q,K} (\kabs{\dq u_0})^{\frac{1}{2}} \eta }^{2}_{L^{2}(D)}
	+ 1  + c \int_0^t \bnorm{f_H(s)}_{L^a(D)}^a \ds  \\
	& \quad {} + c \int_0^t e^{-\int_0^s c'' \, G''(u,f) \, d\tilde{s} } \,  
	\sp{V(\dq u(x,t)) \, \eta^2 }{\dq H(\cdot,s,Du) \,
dB_s}_{L^2(D)} \,.
\end{align*}

\textbf{Step~3.} Similarly to the proof of Lemma~\ref{Diff-quotients}, we introduce the random time
\begin{equation*}
\tau _{R}:=\inf \Big\{ t\in [ 0,T] :\int_{0}^{t} \bnorm{
\left\vert \Delta _{k,h}u(s) \right\vert ^{2q-1}\eta
^{2}\left\vert \Delta _{k,h}H\left( \cdot ,s,Du\right) \right\vert
}_{L^{1}(D) }^{2} \ds>R \Big\} 
\end{equation*}%
with $\tau _{R}=T$ when the set is empty. Differently from Lemma~\ref{Diff-quotients}, the
property 
\begin{equation*}
P\Big( \int_{0}^{T} \bnorm{ \left\vert \Delta _{k,h}u(s)
\right\vert ^{2q-1}\eta ^{2}\left\vert \Delta _{k,h}H\left( \cdot
,s,Du\right) \right\vert }_{L^{1}( D) }^{2} \ds<\infty
\Big) =1
\end{equation*}
which is needed to have $P\left( \lim_{R\rightarrow \infty }\tau
_{R}=T\right) =1$ is not clear a priori. We shall prove it a posteriori.

Notice that, by Lemma~\ref{tech-kalita-2}, 
\begin{multline*}
\quad \int_{0}^{t\wedge \tau _{R}}e^{-2\int_{0}^{s}c'' G''(u,f) d\tilde{s}} \Big( \int_{D} \kabs{V(\dq u(x,s))} \, \eta^{2} \left\vert
\Delta _{k,h}H(s,Du) \right\vert dx \Big)^{2} \ds \\
\leq \int_{0}^{t\wedge \tau _{R}} \bnorm{ \left\vert \Delta
_{k,h}u(s) \right\vert^{2q-1}  \eta ^{2} \left\vert \Delta _{k,h}H(
\cdot ,s,Du) \right\vert }_{L^{1}(D)}^{2} \ds \leq R \,. \quad
\end{multline*}

\textbf{Step~3a.} 
The last calculation shows that the stochastic integral from Step~2, stopped at $\tau _{R}$, is a martingale (and thus it has zero expectation). Therefore (as in Lemma~\ref{Diff-quotients})
\begin{align*}
\lefteqn{  \hspace{-0.5cm} E \Big[ e^{-\int_{0}^{t\wedge \tau _{R}}c'' G''(
u,f) ds}  \bnorm{  T_{q,K} (\kabs{\dq
u(t\wedge \tau _{R})})^{\frac{1}{2}} \eta}^{2}_{L^{2}(D)} \Big] } \\
\lefteqn{ \hspace{-0.5cm} {}+ c^{-1} E \Big[ \int_{0}^{t\wedge \tau _{R}}e^{-\int_{0}^{s}c'' G''(u,f) d\tilde{s}}\bnorm{T_{q,K}'(\kabs{\dq u})^{\frac{1}{2}} \, \kabs{\dq u}^{-\frac{1}{2}} \, D \dq u \, \eta }_{L^2(D)}^2  \ds \Big] } \\
& \leq  \bnorm{T_{q,K} (\kabs{\dq u_0})^{\frac{1}{2}} \eta }^{2}_{L^{2}(D)} +1+c \, E \Big[ \int_{0}^{T} \bnorm{
f_{H}(s)} _{L^{a}(D) }^{a} \ds \Big] .
\end{align*}
At this stage we may pass to the limit $K \to \infty$ via Fatou's Lemma on the
left-hand side and monotone convergence on the right-hand side, and we obtain
\begin{align}
\lefteqn{  \hspace{-0.5cm} E \Big[ e^{-\int_{0}^{t\wedge \tau _{R}}c'' G''(
u,f) ds} \bnorm{
W_{q}( \Delta _{k,h}u (t \wedge \tau_{R})) \, \eta}_{L^{2}(D) }^{2} \Big] } \nonumber \\
\lefteqn{ \hspace{-0.5cm} {}+ c^{-1} E \Big[ \int_{0}^{t\wedge \tau _{R}}e^{-\int_{0}^{s}c'' G''(u,f) d\tilde{s}} \bnorm{ \kabs{\dq u(s)}^{q-1} D \dq u(s) \, \eta }_{L^2(D)}^2  \ds \Big] } \nonumber \\
& \leq  \bnorm{W_q (\dq u_0) \, \eta }^{2}_{L^{2}(D)} +1+c \, E \Big[ \int_{0}^{T} \bnorm{
f_{H}(s)} _{L^{a}(D) }^{a} \ds \Big] .
\label{first step}
\end{align}

\textbf{Step~3b.} Next we apply Burkholder-Davis-Gundy inequality to the inequality above
stopped at $\tau _{R}$, raised to the power $\frac{1}{q}$. Taking the limit $K \to \infty$ as in~\eqref{first step}, we get
\begin{align*}
\lefteqn{ \hspace{-0.5cm} E \Big[ \sup_{t \in [0,T] }e^{- \frac{1}{q} \int_{0}^{t\wedge \tau _{R}}c'' G''(u,f) ds} \bnorm{ W_{q}( \Delta _{k,h}u (t \wedge \tau_{R})) \, \eta }^{\frac{2}{q}}_{L^{2}(D)}  \Big] } \\
&\leq \bnorm{W_q (\dq u_0) \, \eta}^{\frac{2}{q}}_{L^{2}(D)} 
+1+c \, E \Big[ \Big( \int_{0}^{T} \bnorm{ f_{H}(s)}_{L^{a}(D) }^{a} \ds \Big)^{\frac{1}{q}} \Big]  \\
& \quad {}+ C \, E \Big[ \Big( \int_{0}^{T\wedge \tau
_{R}}e^{-2\int_{0}^{s}c'' G''( u,f) d\tilde{s}} 
\bnorm{ \kabs{ \dq u(s)}^{2q-1} \eta^{2} \left\vert \Delta _{k,h}H( \cdot ,s,Du) \right\vert
}_{L^{1}(D) }^{2} \ds \Big)^{\frac{1}{2q}} \Big] \,.
\end{align*}
Since due to H\"older's inequality we have
\begin{align*}
\lefteqn{ \hspace{-0.5cm} \bnorm{ \kabs{ \Delta _{k,h} u(s)}^{2q-1}\eta^{2} 
\left\vert \Delta _{k,h}H(\cdot ,s,Du)
\right\vert }_{L^{1}(D)}^{2}}  \\
&\leq \left\Vert \left\vert \Delta _{k,h}u(s) \right\vert ^{q}\eta
\right\Vert _{L^{2}(D) }^{2} \bnorm{ \left\vert \Delta
_{k,h}u (s) \right\vert ^{q-1}\eta \left\vert \Delta
_{k,h}H( \cdot ,s,Du) \right\vert }_{L^{2}(D) }^{2} \,,
\end{align*}
the last term of the previous inequality, similarly to the proof of Lemma~\ref{Diff-quotients}, is bounded by
\begin{equation*}
c \, E\big[ I_{1}^{1/2}I_{2}^{1/2} \big] \leq \frac{1}{2}E[ I_{1}] +
\frac{C^{2}}{2}E[ I_{2}] 
\end{equation*}
where
\begin{align*}
I_{1}& =\sup_{t\in [ 0,T] }e^{- \frac{1}{q} \int_{0}^{t\wedge \tau
_{R}}c'' G''(u,f) \,d\tilde{s}}\, \bnorm{ W_{q}\left( \Delta _{k,h}u( t \wedge \tau _{R})
\right) \eta }_{L^{2}(D) }^{\frac{2}{q}}  \, ,\\
I_{2}& = \Big( \int_{0}^{T \wedge \tau _{R}} e^{-\int_{0}^{s}c'' G''(u,f) \,d\tilde{s}}\, \bnorm{ \left\vert \Delta_{k,h}u(s) \right\vert ^{q-1}\eta \left\vert \Delta _{k,h}H( \cdot ,s,Du) \right\vert }_{L^{2}(D) }^{2} \ds \Big)^{\frac{1}{q}} \, .
\end{align*}
Hence, we have proved that
\begin{equation*}
\frac{1}{2}E[ I_{1}] \leq \bnorm{  W_{q}( \Delta_{k,h}u_0) \eta }_{L^{2}(D)}^{\frac{2}{q}}
+1+c \, E \Big[ \Big( \int_{0}^{T}\bnorm{f_{H}(s)}_{L^{a}(D)}^{a} \ds \Big)^{\frac{1}{q}}\Big] + 
\frac{C^{2}}{2}E[I_{2}] \,.
\end{equation*}
Now, by the assumptions~\eqref{GV-H} on $H$, Young's inequality and the bound on $q$, we have
\begin{align*}
\frac{C^2}{2} E[ I_{2}] & \leq C \, E\Big[ \Big( \int_{0}^{T \wedge \tau
_{R}} e^{-\int_{0}^{s}c'' G''(u,f) d\tilde{s}
}\bnorm{ \kabs{ \Delta _{k,h}u(s) }^{q-1} D \dq u(s) \, \eta }_{L^{2}(D)}^{2} \ds \Big)^{\frac{1}{q}}\Big] \nonumber \\ 
& \quad{}+ C + C \, E
\Big[ \Big( \int_{0}^{T} \bnorm{ f_{H}(s) }_{L^{a}(D)}^{a} \ds \Big)^{\frac{1}{q}} \Big] + \frac{1}{4} E[ I_{1}] \\
 & \quad {}+  C \, E \Big[ \Big( \int_{0}^{T} e^{-\int_{0}^{s}c'' G''(u,f) d\tilde{s}}  \bnorm{ W_{q}(Du(s)) }_{L^{2}(D') }^{2} \ds \Big)^{\frac{1}{q}} \Big] \,.
\label{crucial}
\end{align*}
We observe that the last term remains bounded, due to the assumption of the lemma on the average and the choice of $G''(u,f)$ (which ensures that $c'' G''(u,f) \geq 2 q G_p$). Thus, by inequality~\eqref{first step} proved above, we find
\begin{equation*}
\frac{1}{4}E[I_{1}] \leq c \bnorm{ W_{q}( \Delta
_{k,h}u_0) \, \eta}_{L^{2}(D) }^{\frac{2}{q}}+c+c \, E
\Big[ \int_{0}^{T} \bnorm{ f_{H}(s) }_{L^{a}(D)}^{a} \ds \Big]^{\frac{1}{q}} 
\end{equation*}
with a new constant. 

\textbf{Step~3c.} In Step~3a and Step~3b we almost proved the two bounds claimed by the lemma since the previous inequality along with~\eqref{first step} gives us
\begin{align*}
\lefteqn{ \hspace{-0.5cm} E\Big[ \sup_{t \in [0,T] }e^{- \frac{1}{q} \int_{0}^{t \wedge \tau _{R}}c'' G''(u,f) ds} \bnorm{W_{q}( \Delta _{k,h}u( t \wedge \tau _{R}))
\eta }_{L^{2}(D) }^{\frac{2}{q}} \Big] } \\
\lefteqn{ \hspace{-0.5cm} {}+ E\Big[ \Big( \int_{0}^{T \wedge \tau _{R}}e^{-\int_{0}^{s}c'' G''(u,f) d\tilde{s}} \bnorm{ \left\vert \Delta
_{k,h}u(s)\right\vert^{q-1} D \dq u(s) \, \eta }_{L^{2}(D)}^{2} \ds \Big)^{\frac{1}{q}}\Big] } \\
&\leq c \, \Big( \bnorm{ W_{q}( \Delta _{k,h}u_0) \, \eta }_{L^{2}(D) }^{2}
+c+c \, E \Big[ \int_{0}^{T} \bnorm{f_{H}(s)}_{L^{a}(D) }^{a} \ds \Big] \Big)^{\frac{1}{q}} \,.
\end{align*}
It now remains to justify (as already observed above) the limit $\tau_R \to T$ as $R \to \infty$ with probability one. Indeed, since $R\mapsto \tau _{R}$ is non-decreasing and bounded above by $T$, there exists the a.s. limit 
\begin{equation*}
\tau :=\lim_{R\rightarrow \infty }\tau _{R}
\end{equation*}
and $\tau (\omega) \in [0,T]$. By Fatou's lemma and monotone convergence, 
\begin{multline*}                                                                                
\quad E\Big[ \Big( \sup_{t \in [0,\tau] }e^{- \int_{0}^{t}c'' G''(u,f) ds} \bnorm{ W_{q}( \Delta_{k,h}u(t) ) \, \eta }_{L^{2}(D)}^{2} \\ + 
\int_{0}^{\tau} e^{-\int_{0}^{s}c'' G''(u,f) d\tilde{s}} \bnorm{ \kabs{ \Delta _{k,h}u(s)}^{q-1} D \dq u(s) \, \eta }_{L^{2}(D)}^{2} \ds \Big)^{\frac{1}{q}} \Big] \quad
\end{multline*}
is finite, hence the argument of the expectation is finite with probability
one. Since $\int_{0}^{T}c'' G''(u,f) \ds$ is finite with probability one, we get
\begin{equation*}
\sup_{t \in [0,\tau]} \bnorm{W_{q}(\Delta_{k,h}u(t) ) \, \eta }_{L^{2}(D)}^{2} 
+\int_{0}^{\tau} \bnorm{ \kabs{\Delta_{k,h}u(s) }^{q-1} D \dq u(s) \, \eta }_{L^{2}(D) }^{2} \ds<\infty 
\end{equation*}
with probability one. Thus (with the same inequalities used above)
\begin{multline*}
\quad \int_{0}^{\tau }\bnorm{ \kabs{ \Delta _{k,h}u(s) }^{2q-1} \eta^{2} \left\vert \Delta _{k,h}H( \cdot,s,Du) \right\vert }_{L^{1}(D) }^{2} \ds \\ 
  \leq  C \Big( 1 + \sup_{t \in [0,\tau]} \bnorm{W_{q}(\Delta_{k,h}u(t) ) \, \eta }_{L^{2}(D)}^{2}  + \int_{0}^{\tau } \bnorm{ \kabs{\Delta
_{k,h}u(s) }^{q-1} D \dq u(s) \, \eta }_{L^{2}(D)}^{2} \ds \\
  + \int_{0}^{T} \bnorm{ W_{q}(Du(s)) }_{L^{2}(D') }^{2} \ds  
  + \int_{0}^{T} \bnorm{ f_{H}(s) }_{L^{a}(D)}^{a} \ds  \Big)^2 < \infty \quad
\end{multline*}
with probability one. If $\tau(\omega) <T$, by definition of $
\tau_{R}$ we have
\begin{equation*}
\int_{0}^{\tau }\bnorm{ \kabs{ \Delta _{k,h}u(s) }^{2q-1}\eta^{2} \left\vert \Delta _{k,h}H(\cdot
,s,Du) \right\vert }_{L^{1}(D)}^{2} \ds=\infty 
\end{equation*}
which is false, hence $P( \tau =T ) =1$. Having this basic fact,
the same estimates just proved give us the result of the lemma, by taking into account the inequality $\kabs{DW(\dq u)} \leq q \kabs{\dq u}^{q-1} \kabs{D \dq u}$ and the definition of $G_q$ (and hence of $Y_q$) given after~\eqref{def-G''}.
\end{proof}

\section{Proof of the regularity result}

Having the previous lemma at hand, we may now proceed to our main result.

\begin{theorem}
\label{main-result}
Let $u$ be a weak solution to the initial boundary value problem
to~\eqref{equation} with initial values $u(\ccdot,0) = u_0(\cdot) \in W^{1,a-2}(D,\R^N)$. Assume further the assumptions~\eqref{GV-A} with $\nu >  (n- 2)/n$ such that
\begin{equation*}
L_H^2 < (L_H^*)^2(n) := 
\frac{2}{\kappa \, (n-1)} \Big( \Big[ 1 - \Big( \frac{n-2}{n}\Big)^2
\Big]^{\frac{1}{2}} - \Big[ 1 - \nu^2 \Big]^{\frac{1}{2}} \Big) \,.
\end{equation*}
Then there exists $\alpha > 0$ depending only on $n,\nu$ and $a$ such that for 
every subset $D_c \Subset D$ we have
\begin{equation*}
P \big( \knorm{u}_{C^{0,\alpha}(D_c \times [0,T],\R^N)}  <  \infty \big) 
=  1 \,.
\end{equation*}

\begin{proof}
To prove the result, we want to apply Proposition~\ref{prop_Hoelder}. Therefore, the crucial point is to show higher integrability of $Du$ for ``great'' powers with probability one, in order that hypothesis~\eqref{int_assumption_Hoelder} of the proposition is satisfied. We start by defining a sequence
\begin{align*}
\tilde{q}_0 & := 1 \,, \\
\tilde{q}_{j+1} & := \min\Big\{ q_j \frac{n+2}{n} , 1 + q_j
\frac{n+2}{n} \frac{a-4}{a}, \frac{a-2}{2}, q_j + 1 \Big\}  \text{ for } j \geq 1 \,.
\end{align*}
Before defining a further sequence $(q_{j})$ in order to perform the iteration, we make some observations on $L_H^*(s)$ as a function in $s \in [1,2/(1 - \nu)]$ (we note that $L_H^*(2q)$ already appeared in hypothesis~\eqref{restrict-q} which gave an upper bound for $q$ in the iteration). Clearly, $L_H^*(s)$ is strictly decreasing in $s$, with $L_H^*(2/(1 - \nu)) = 0$. 

We now set $q_{j} = \tilde{q}_{j}$ as long as $\tilde{q}_{j} > \tilde{q}_{j-1}$
and $L_H^*(2\tilde{q}_{j}) > L_H$, and for the first index $j$ which doesn't
satisfy these assumptions any more we set $q_{j} = q^*$ for a number $q^* > n/2$
(which is determined below). In what
follows we shall denote this set of indices by $J \subset N_0$. We first study
some properties of the sequence $\tilde{q}$ and give a definition of the final
member $q^*$ of the sequence $(q_{j})_{j \in J}$: the first and the forth term
in the rewritten formula for $\tilde{q}$ are strictly increasing in $j$ and
diverge for $j \to \infty$, whereas the the monotonicity properties of the
second term depend on both the values of $a$ and the size of $q_j$. More
precisely, if $a \geq 2(n+2)$, then the second term increases with $j$ and
diverges for $j \to \infty$, but for every $a \in (n+2,2(n+2))$ it increases
only up to $q_{max}(a,n) = na / (4(n+2) -2a) > n/2$. Observing $L_H^*(n)> L_H$ by assumption, we thus define
\begin{equation*}
q^*  :=    \text{arbitrary number in } \Big(\frac{n}{2}, \min \Big \{(L^*_H)^{-1}(L_H),\frac{a-2}{2},q_{max} \Big \}\Big) \,.
\end{equation*}
It is easy to calculate that this number $q^*$ is reached after a finite number
of steps (depending only on $n,\nu,a$ and the difference $q_{max}-q^*$ (in the
sense that the number of steps diverges as $q^* \nearrow q_{max}$), hence
$\kabs{J} < \infty$, i.e. $(q_{j})_{j \in J}$ is a finite sequence.

We are now going to establish by induction that for every $j \in J$  we have
\begin{align*}
(i) \quad & \sup_{\kabs{h} < \dist(D_j,\partial D_{j-1})} E \big[
\bnorm{Y_{q_j}^{q_j} \, W_{q_j}(\dq u) }^{2/q_j}_{V^2(D_j \times (0,T))} \big]  \leq  C_j \text{
for all } k \in \{1,\ldots,n\} \,, \\
(ii) \quad & \sup_{\kabs{h} < \dist(D_j,\partial D_{j-1})} E \big[ \bnorm{ Y_{q_j}^{q_j}
W_{q_j}(\dq u) }^{2/q_j}_{L^{2 \frac{n+2}{n}}(D_j \times (0,T))} \big]  \leq  c(n,D_j)
\, C_j \text{ for all } k \in \{1,\ldots,n\} \,, \\
(iii) \quad &  E \big[ \bnorm{ Y_{q_j}^{q_j}
W_{q_j}(Du) }^{2/q_j}_{L^{2 \frac{n+2}{n}}(D_j \times (0,T))} \big]  \leq  
\, \widetilde{C}_j  \, ,\\
(iv) \quad & Du \in L^{\infty}(0,T; L^{2q_j}(D_j,\R^{nN})) \text{ with
probability one.} 
\end{align*}
Here $(Y_{q_j})_{j \in J}$ is a sequence of random variables given by $Y_{q_j}(t,\omega) = {\rm exp}(- \int_0^t G_{q_j}(s,\omega) ds)$ for each $j \in J$, for a sequence of functions $(G_{q_j})_{j \in J}$ which are in $L^1(0,T)$ with probability one and which will be determined later, and $(D_j)_{j \in J}$ is a monotone decreasing sequence of open sets satisfying $D_c \subset D_{j} \subset D_{j-1} \subset \ldots \subset D_0 \subset D_{-1} = D$. 

We start by setting
\begin{equation*}
Y_1  :=  e^{- \frac{1}{2} \int_0^t c' G'(u,f) \ds} \,,
\end{equation*}
where $G'(u,f)$ was defined in~\eqref{higher2}. It is obvious from its definition
that $Y_1 \colon [0,T] \times \Omega \to (0,1]$ satisfies $P( \inf_{t \in [0,T]} Y_1 > 0)  =1$. We then observe from Lemma~\ref{Diff-quotients} that
\begin{multline*}
\quad \sup_{\kabs{h} < d} \, \E \, \Big[ \, \sup_{t \in (0,T)} \knorm{Y_1 \dq
u}^2_{L^2(D_0)} + \int_0^T \knorm{Y_1 D \dq u}_{L^2(D_0)}^2 \dt \Big] \\
	 \leq  c' \Big( \knorm{D_k u_0}_{L^2(D)}^2 + 1 
	+ \E \big[ \knorm{f_H^{\frac{a}{a-2}}}_{L^2(D_T)}^2 \big] \Big) =: C_0 \quad
\end{multline*}
is satisfied for every open set $D_0$ compactly supported in $D$. By definition
of the space $V^2$, this establishes the statement (i)$_0$. Furthermore,
(ii)$_0$ follows immediately from the Sobolev embedding~\eqref{embedding},
applied for $P$-almost every $\omega$ to the functions $Y_1 \, \dq u$, for $k
\in \{1,\ldots,n\}$. To conclude the first step of the iteration it only remains
to justify the statements (iii)$_0$ and (iv)$_0$. To this end we take advantage
of Theorem~\ref{thm-derivatives} twice, in the way as explained in Remark
\ref{rem-diff-quot} (and actually as already performed in Corollary~\ref{cor-second-der}). First we apply it with the choices $p=q=2q_0 \frac{n+2}{n}$
to the inequality from (ii)$_0$ (for all $k \in \{1,\ldots,n\}$), leading to the
existence of $Du$ in the Lebesgue space $L^{2 (n+2)/n}(D_0 \times (0,T),\R^{nN})$ with the required estimate for the average of $Y_1 Du$; secondly, we apply it with the choice
$p=2q_0$ and $q=\infty$ to (i)$_0$ --~more precisely to the first term in the
$V^2$-norm~-- and, keeping in mind the pathwise strict positivity of $Y_1$, we end up with the existence of $Du$ in $L^{\infty}(0,T;L^{2}(D_0,\R^{nN}))$ with probability one.

We now proceed to the inductive step. Assume for a given $j \in J$ that
(i)$_\ell$--(iv)$_\ell$ are valid on open sets
$D_{\ell} \subset D_{\ell-1}$ with random variables $Y_{q_\ell} \colon [0,T] \times \Omega \to (0,1]$ of the required form for all $\ell \in \{0,\ldots,j-1\}$. Then,
keeping in mind (iii)$_{j-1}$ and the definition of the number $q^*$, we note that the assumptions of Lemma
\ref{diff-quotients-iteration} are satisfied (for $p,D'$ replaced by
$q_{j-1},D_{j-1}$), and we hence deduce (with the admissible choice $q =
q_{j}$) the estimate
\begin{multline*}
\quad \sup_{\kabs{h} < d_j} \E \Big[ \Big( \sup_{t \in (0,T)} 
\knorm{Y_{q_j}^{q_j} \, W_{q_j}(\dq u)}^2_{L^2(D_j)} 
	+ \int_0^T \knorm{Y_{q_j}^{q_j} \, DW_{q_j}(\dq u)}_{L^2(D_j)}^2 \dt \Big)^{\frac{1}{q_j}} \Big] \\
	\leq  c \, \Big( \knorm{W_{q_{j-1}}(D_k u(x,0))}^{2}_{L^{2}(D)} + 1 + 
	\E \big[ \knorm{f_H(s) )}_{L^a(D_T)}^a \, \big]
	\Big)^{\frac{1}{q_j}}  =: C_j \quad
\end{multline*}
for every $k \in \{1,\ldots,n\}$, a domain $D_j \subset D_{j-1}$ satisfying $d_j
:= \dist (D_j , \partial D_{j-1}) > 0$ and a random variable $Y_{q_j}$ defined via $G_{q_j}$ given in Lemma~\ref{diff-quotients-iteration} and satisfying in particular $P( \inf_{t \in [0,T]} Y_{q_j} > 0)  =1$. This shows (i)$_j$, and (ii)$_j$ in turn
is an immediate consequence after the application of the Sobolev embedding as
above. Moreover, the statements (iii)$_j$ and (iv)$_j$ again follow from
(ii)$_j$ and (i)$_j$, respectively, after the application of Theorem
\ref{thm-derivatives} with the choices $p=q=2q_j \frac{n+2}{n}$ and $p=2q_j$, $q
= \infty$, respectively. This finishes the proof of the induction.

As an immediate consequence of the induction, we can now conclude the desired higher integrability result to a great power, via the following observation. 
Via (iv) we find in the limit 
\begin{equation*}
Du \in L^{\infty}(0,T; L^{2q^*}(D_c,\R^{nN}))
\end{equation*}
with probability one, and by definition the exponent $2q^*$ is greater than the space dimension $n$. Hence, assumption~\eqref{int_assumption_Hoelder} of Proposition~\ref{prop_Hoelder} is guaranteed. For its application we still need to check the integrability condition on $a(x,s), b(x,s)$ given by
\begin{equation*}
 a(x,s) := \diverg A(x,s,u,Du) \qquad \text{and} \qquad b(x,s) := H(x,s,Du) \,.  
\end{equation*}
Since $A(x,t,u,z)$ is differentiable in $x$, $u$, and $z$ with bounds~\eqref{GV-A}, we obtain $a \in L^2(D_c \times (0,T),\R^N)$ with probability one as a direct consequence of $Du \in V^2(D_c \times (0,T),\R^{nN})$ and $f \in L^a(D_T) \subset L^4(D_T)$. Furthermore, the growth of $H$ according to~\eqref{GV-H} with $f_H \in L^a(D_T \times \Omega)$ implies $b \in L^{2+\epsilon}(0,T;L^2(D_c,\R^{n'N})$ with probability one. Thus, Proposition~\ref{prop_Hoelder} yields the asserted H\"older continuity of $u$ with probability one and finishes the proof of the theorem. 
\end{proof}
\end{theorem}

\section{Regularity of the average due to noise}
\label{sec_noise-regularization}

It has been recently proved that a Stratonovich bilinear multiplicative noise
may have a regularizing effect on certain classes of PDEs, see~\cite{FLALNM}
for a review, based on a number of works including~\cite{FLAGUBPRI10,FLAGUBPRI10a,ATTFLA11}. In most cases, \textit{uniqueness} by
noise is the topic of these works. The problem of the interaction between noise and
\textit{singularities} is more difficult and less explored. But two examples
are known: 
\begin{enumerate}
 \item[(i)] for linear transport equations of the form
\[
du = (b(x,t) \cdot Du)  \dt +\sigma Du\circ dB_t
\]
with $b\in C(0,T;C_{b}^{\alpha}(\mathbb{R}^{n},\mathbb{R}^{n}))$, where 
regular initial condition may develop discontinuities in finite time in the case $\sigma=0$ (think of the simple example in dimension $n=1$ given by 
$b(x)  = - {\rm sign}(x)  \sqrt{\kabs{x}}$), it is known that $C^{1}$-smoothness 
is preserved for $\sigma \neq 0$, see~\cite{FLAGUBPRI10,FLAGUBPRI11}, 
where similar results have been also proved for linear continuity equations;
 \item[(ii)] for the point vortex motion associated to the 2D Euler equations, it has
been proved that coalescence of vortices cannot happen when a suitable
Stratonovich bilinear multiplicative noise is added to the equations, see
\cite{FLAGUBPRI10a}. 
\end{enumerate}
One should also notice that other singularities, like those
arising in the inviscid Burgers equation, do not disappear under noise, see
\cite{FLALNM}, so each equation requires its own understanding and
investigation. Moreover, no general method exists to investigate these kind
of properties.

Our aim here is to give a simple partial result in this direction (namely the effect of
noise on singularities) for \textit{linear systems}. We consider the linear
stochastic system with Stratonovich bilinear multiplicative noise of the form
\begin{equation}
du  =\operatorname{div} \big( A(x,t)  Du \big) \dt + \sigma Du \circ dB_t  \,,
\qquad
u|_{t=0}=u_{0} \label{linear SPDE}
\end{equation}
with bounded measurable coefficient matrix $A$, where $B_t$ is
a Brownian motion in $\mathbb{R}^{n}$, defined on a filtered probability space
$\left(  \Omega,F_{t},P\right)  $. The space variable $x$ varies in a possibly
unbounded regular open domain $D\subset\mathbb{R}^{n}$. On $A$ we assume that
there exist $\lambda_{0},\lambda_{1}>0$ such that
\begin{equation}
\lambda_{0} \kabs{\xi}^{2} \leq \sp{A(x,t) \, \xi}{\xi} \qquad 
\text{and} \qquad \kabs{A(x,t) \, \xi} \leq \lambda_{1} \kabs{\xi}
\label{linear assumption}
\end{equation}
for all $\xi\in\mathbb{R}^{nN}$, a.\,e. $(x,t)  \in D \times [0,T]$. Actually, this is analogous to assumption~\eqref{GV-A}$_2$ (then $\nu$ corresponds to the ratio  $\frac{\lambda_0}{\lambda_1}$), rewritten for vector fields which are linear in the gradient variable. We further note that for now we do not assume any regularity with respect to $x$, but at the same time we do not allow any dependency on $\Omega$. Let us clarify the vector notation used in the stochastic part: $\sigma Du(x,t)  \circ dB_t  $ is
a vector with $N$ components, and
\[
(  \sigma Du(x,t)  \circ dB_t )
^{\alpha}=\sigma\sum_{i=1}^{n}D_i u^{\alpha}(x,t)
\circ dB^{i}_t \, .
\]

\begin{remark}
Let us recall that Stratonovich noise is the natural one for modelling: the so
called \textit{Wong-Zakai principle}, proved for several classes of SPDEs (see
for instance the appendix of~\cite{FLAGUBPRI10} for the linear transport
equation),
states that solutions $u_{n}\left(  x,t\right)  $ of deterministic equations
with smooth random coefficients $B_{n}(t)  $ of the form%
\[
\frac{\partial u_{n}}{\partial t}=\operatorname{div}
\big( A(x,t)  Du_{n}\big)  + \sigma Du_{n} \frac{dB_{n}(t)  }{dt}\, , 
\qquad u_n|_{t=0}=u_{0} 
\]
converge (in proper topologies and under proper assumptions on $B_{n}$, the
details depend on the problem and result) to solutions $u$ of the previous
SPDE with Stratonovich noise (not It\^{o} noise). We have stated the principle
for our system of parabolic equations just for sake of definiteness, but in
fact it has not been proved before in this generality. We do not want to give
a proof here, which would require a considerable work. We only quote this fact
by analogy with other equations, as a general motivation for the choice of
Stratonovich noise.
\end{remark}

Let us give the definition of weak solution to equation~\eqref{linear SPDE},
similarly to~\cite{FLAGUBPRI10}. To understand one of the requirements (the fact
that $s\mapsto\int_{D}u(x,s)  D\varphi(x)  \dx$ must have a modification which is a
continuous adapted semi-martingale), we recall a few facts about Stratonovich
stochastic integrals, taken for instance from~\cite{KUNITA84}. If $B_t$ is a $(\Omega,F_{t},P)$-Brownian motion in $\mathbb{R}^{n}$
and $X\left(  t\right)  $ is a continuous $F_{t}$-adapted semi-martingale, the
following uniform-in-time limit exists in probability%
\[
\int_{0}^{t}X(s)  \circ dB_s  =\lim_{n\rightarrow
\infty}\sum_{t_{i}\in\pi_{n},t_{i}\leq t}\frac{X(  t_{i+1}\wedge
t)  +X(t_{i})}{2} (  B_{t_{i+1}\wedge t}
-B_{t_{i}} )
\]
and is called Stratonovich integral of $X$ with respect to $B$. Here $\pi_{n}$
is a sequence of finite partitions of $[0,T]$ with size
$\kabs{\pi_{n}} \rightarrow0$ and elements $0=t_{0}<t_{1}<...$.
Under the same assumptions it is defined the joint quadratic variation between
$X$ and $B$:
\[
[X,B]_{t} = \lim_{n\rightarrow\infty} \sum_{t_{i}\in\pi_{n},t_{i}\leq t}(X(t_{i+1}\wedge t)  - X(t_{i})) \, (  B_{t_{i+1}\wedge t}  - B_{t_{i}}) \,,
\]
and they are related to the It\^{o} integral
\[
\int_{0}^{t} X(s)  \, dB_s  = \lim_{n\rightarrow\infty
}\sum_{t_{i}\in\pi_{n},t_{i}\leq t} X(t_{i}) \, (B_{t_{i+1}\wedge t} - B_{t_{i}})
\]
(which is defined under more general assumptions on $X$) by the formula
\[
\int_{0}^{t}X(s)  \circ dB_s  =\int_{0}^{t} X(s) \, dB_s  +\frac{1}{2} \, [X,B]_{t} \,.
\]

\begin{definition}
\label{definition solution stoch}
If $u_{0}\in L_{loc}^{2}\left(
D,\mathbb{R}^{N}\right)  $, we say that a random field $u(x,t)$
is a weak solution of equation~\eqref{linear SPDE} if:
\begin{enumerate}
\item[(i)] with probability one, we have $u \in V^2(B_T, \R^N)$
for all bounded open sets $B\subset D$, where $B_T = B \times (0,T)$,
\item[(ii)] for all $\varphi\in C_{0}^{\infty}(D,\mathbb{R}^{N})$, the
$\mathbb{R}^{n}$-valued process $s\mapsto\int_{D}u(x,s)
D\varphi(x) \dx$ has a modification which is a continuous adapted
semi-martingale, and for all $t \in [0,T]$, we have $P$-a.\,s.
\begin{multline*}
\int_{D} u(x,t) \, \varphi(x)  \dx + \int_{0}^{t} \int
_{D} A(x,s)  Du(x,s)  D\varphi(x)  \dx \ds 
+ \sigma \int_{0}^{t} \Big(  \int_{D}u(x,s)  D\varphi(x) \dx \Big) 
\circ dB_s \\
=\int_{D}u_{0}(x) \varphi(x)  \dx \,.
\end{multline*}
\end{enumerate}
\end{definition}

A posteriori, from the equation itself, it follows that for all $\varphi\in
C_{0}^{\infty}\left(  D,\mathbb{R}^{N}\right)  $ the real-valued process
$s \mapsto\int_{D}u(x,s)  \varphi(x) \dx$ has a continuous modification. We shall always use it. Notice further that we give the following meaning to the vector notation above:
\[
\int_{0}^{t} \Big(  \int_{D}u(x,s) \, D\varphi(x)
\dx \Big)  \circ dB_s  =\sum_{i=1}^{n}\sum_{\alpha=1}^{N}
\int_{0}^{t} \Big(  \int_{D}u^{\alpha}(x,s) \, D_i \varphi^{\alpha}(x) \dx \Big)  \circ dB^{i}_s \, .
\]

\begin{proposition}
\label{proposition Stratonovich Ito}
A weak solution in the previous
Stratonovich sense satisfies the It\^{o} equation
\begin{multline*}
\quad  \int_{D}u(x,t) \, \varphi(x)  \dx+\int_{0}^{t}
\int_{D}A(x,s)  Du(x,s)  D\varphi(x)
\dx \ds+\sigma\int_{0}^{t} \Big(  \int_{D}u(x,s)  D\varphi(
x)  \dx \Big) \, dB_s  \\
 =\int_{D}u_{0}(x)  \varphi(x)  \dx+\frac
{\sigma^{2}}{2}\int_{0}^{t}\int_{D}u(x,s) \Delta\varphi(x)  \dx \ds \quad
\end{multline*}
for all $\varphi\in C_{0}^{\infty}\left(  D,\mathbb{R}^{N}\right)  $. The
converse is also true. With a language similar to that of 
Definition~\ref{definition solution stoch}, we could say that $u$ is a weak solution of
the It\^{o} equation
\begin{equation}
du  =\operatorname{div} \Big( \Big( A(x,t)
+\frac{\sigma^{2}}{2} \Big)  Du \Big)  \dt+\sigma Du \, dB_t  \,,
\qquad u|_{t=0}=u_{0} \, .\label{Ito SPDE}
\end{equation}
\end{proposition}

\begin{proof}
From the facts recalled above about Stratonovich integrals we have
\begin{align*}
\int_{0}^{t} \Big(  \int_{D}u^{\alpha}(x,s) \, D_{i}
\varphi^{\alpha}(x)  \dx \Big)  \circ dB^{i}_s
&  =\int_{0}^{t} \Big(  \int_{D}u^{\alpha}(x,s) \, D_{i}
\varphi^{\alpha}(x)  \dx \Big) \, dB^{i}_s  \\
& \quad {} +\frac{1}{2} \Big[  \int_{D}u^{\alpha}(x,\cdot)
\, D_{i} \varphi^{\alpha}(x)  \dx,B^{i} \Big]_{t}.
\end{align*}
Hence, we get
\begin{multline*}
\quad \int_{D}u(x,t) \, \varphi(x)  \dx+\int_{0}^{t}
\int_{D}A(x,s) \, Du(x,s) \, D\varphi(x)
\dx \ds + \sigma \int_{0}^{t} \Big(  \int_{D} u(x,s) \, D\varphi(
x)  \dx \Big) \, dB_s  \\
  = \int_{D} u_{0}(x) \, \varphi(x)  \dx-\frac{\sigma
}{2}\sum_{i=1}^{n}\sum_{\alpha=1}^{N} \Big[  \int_{D}u^{\alpha}(
\cdot) \, D_i \varphi^{\alpha}(x)
 \dx,B^{i} \Big]_{t} \,. \quad
\end{multline*}
By the equation in Definition~\ref{definition solution stoch} we also have
\begin{multline*}
\quad \int_{D}u(x,t) \, D_{i}\varphi(x)
\dx+\int_{0}^{t}\int_{D} A(x,s) \, Du(x,s)
\, D D_{i}\varphi(x) \dx \ds \\
 =\int_{D}u_{0}(x) \, D_{i}\varphi(x)
\dx- \sigma \int_{0}^{t} \Big(  \int_{D}u(x,s) \, D D_{i}
\varphi(x)  \dx \Big)  \circ dB_s  \,. \quad
\end{multline*}
Moreover, recall that
\begin{equation*}
\int_{0}^{t} \Big(  \int_{D}u(x,s) \, D D_{i}
\varphi(x)  \dx \Big)  \circ dB_s  =\sum_{j=1}
^{n}\int_{0}^{t} \Big(  \int_{D} u(x,s) \, D_{j}
D_{i}\varphi(x) \dx \Big) \circ dB^{j}_s \,.
\end{equation*}
Thus, by the classical rules about quadratic variation, see~\cite{KUNITA84},
we have
\begin{align*}
\sum_{\alpha=1}^{N} \Big[  \int_{D}u^{\alpha}(x,\cdot) \,
D_{i}\varphi^{\alpha}(x)  \dx,B^{i} \Big]_{t}
  & = \Big[  \int_{D}u(x,\cdot) \, D_{i}\varphi(
x)  \dx,B^{i} \Big]_{t} \\
  & =-\sigma\int_{0}^{t} \Big(  \int_{D} u(
x,s) \, D_{i} D_{i}\varphi(x) \dx \Big)  \ds.
\end{align*}
The proof that the Stratonovich equation yields the It\^{o} one is complete,
and the proof of the converse statement is the same (recall that the existence of the continuous modification in (ii) of Definition~\ref{definition solution stoch} follows immediately from the equation in Definition~\ref{def_weak_solution_ito}).
\end{proof}

The degree of parabolicity of the It\^{o} SPDE~\eqref{Ito SPDE} is the same as
the one of~\eqref{linear SPDE}, it is given just by the properties of
$A(x,t)$. The term $\frac{\sigma^{2}}{2}\Delta u(x,t) dt$ is fully compensated by the It\^{o} term $\sigma Du(x,t)  dB_t $ and does not contribute to any additional
parabolicity. This is a well recognized phenomenon in the theory of SPDEs, see
for instance~\cite{KRYROZ79}. A simple way to see this fact is to
consider the case $A \equiv 0$.

\begin{proposition}
Consider the equation
\begin{equation}
\label{example_A_0_strat}
du=Du\circ dB_{t},\qquad u|_{t=0}=u_{0}
\end{equation}
in the full space $D=\mathbb{R}^{n}$, where $B$ is an $n$-dimensional Brownian
motion and $u:D \times [0,T] \times \Omega \rightarrow \mathbb{R}^{N}$. 
This is equivalent (when formulated in a weak sense) to the equation
\begin{equation}
du=\frac{1}{2}\Delta u \dt+Du \, dB_{t},\qquad u|_{t=0}=u_{0}.\label{example Ito}%
\end{equation}
Assume $u_{0}\in L^{2}(D,\mathbb{R}^{N})$. Then
\[
u(x,t)  =u_{0}(x+B_{t})
\]
is a weak solution, in the sense that 
\begin{enumerate}
\item[(i')] with probability one, we have $\int_{0}^{T}\int_{\mathbb{R}^{n}}\left\vert
u(x,t)  \right\vert ^{2}dx \dt<\infty$,
\item[(ii')] condition (ii) of Definition~\ref{definition solution stoch} hold true.  
\end{enumerate}

\begin{proof}
Condition (i') comes from
\[
\int_{0}^{T}\int_{\mathbb{R}^{n}} \kabs{u(x,t)}^{2} \dx \dt =\int_{0}^{T}\int_{\mathbb{R}^{n}} \kabs{ u_{0}(x+B_{t})}^{2} \dx \dt =\int_{0}^{T}\int_{\mathbb{R}^{n}} \kabs{ u_{0}(x)}^{2} \dx \dt < \infty \, .
\]
Condition (ii) of Definition~\ref{definition solution stoch} is due to the following argument. For every $\psi\in
C_{0}^{\infty}(\mathbb{R}^{n},\mathbb{R}^{N})$ we have
\[
\int_{\mathbb{R}^{n}}u(x,t)  \psi(x)  \dx = \int
_{\mathbb{R}^{n}}u_{0}(x+B_{t})  \psi(x)
\dx =\int_{\mathbb{R}^{n}}u_{0}(x) \psi(x-B_{t}) \dx
\]
and $\psi(x-B_{t})$ is the semi-martingale
\[
\psi (x-B_{t})  = \psi(x)  - \int_{0}^{t}D\psi(x-B_{s}) \, dB_{s} +\frac{1}{2}\int_{0}^{t}\Delta\psi(x-B_{s})  \ds \,.
\]
Consequently, we obtain
\begin{align*}
\int_{\mathbb{R}^{n}}u (x,t)  \psi(x) \dx  &
=\int_{\mathbb{R}^{n}}u_{0}(x)  \psi(x)  \dx - \int
_{0}^{t} \Big(  \int_{\mathbb{R}^{n}}u_{0}(x)  D\psi(x-B_{s})  \dx \Big) \,  dB_{s}\\
& +\frac{1}{2}\int_{0}^{t} \Big(  \int_{\mathbb{R}^{n}}u_{0}(x)
\Delta\psi(x-B_{s})  \dx \Big) \ds \,,
\end{align*}
which shows that the stochastic process $s\mapsto\int_{\mathbb{R}^{n}}u(x,s)
\psi(x) \dx$ has a modification which is a continuous adapted semi-martingale. In addition, this computation may also be used to prove the equivalence with the It\^{o} formulation~\eqref{example Ito}. Finally,
\begin{multline*}
\quad \int_{\mathbb{R}^{n}}u(x,t)  \varphi(x)
dx-\int_{\mathbb{R}^{n}}u_{0}(x)  \varphi(x)
dx+\int_{0}^{t}\Big(  \int_{\mathbb{R}^{n}}u(x,s) D\varphi(x) \dx \Big) \circ dB_{s}\\
 =\int_{\mathbb{R}^{n}}u_{0}(x)  \varphi(x-B_{t}) \dx- 
\int_{\mathbb{R}^{n}}u_{0}(x)  \varphi(x) \dx
+\int_{0}^{t} \Big(  \int_{\mathbb{R}^{n}} u_{0}(x) D\varphi(x-B_{s})  \dx \Big) \circ dB_{s} \,, \quad
\end{multline*}
and this is equal to zero because
\[
\varphi(x-B_{t})  = \varphi(x)  -\int_{0}^{t} D\varphi(x-B_{s}) \circ dB_{s} \, .
\]
Thus, also condition (iii) is satisfied, and the proof is complete.
\end{proof}
\end{proposition}

\begin{remark}
\label{rem_example_irregular}
The previous proposition shows that the It\^{o} equation~\eqref{example Ito}
has no regularizing properties, in spite of the presence of the term $\frac
{1}{2}\Delta u$ (it is fully compensated by the It\^{o} term). In particular,
if $u_{0}=1_{x_{1}>0}$, the solution $u(x,t)  =1_{B_{t}^{1} < x_{1}}$ 
is discontinuous in $x$ for every given $(t,\omega)$. At the same time
\[
E[u(x,t)]  =E[u_{0}(x+B_{t})]  =P(B_{t}^{1}<x_{1})
\]
is smooth. This means, there are easy examples of a weak solution which have a smooth average, but which are irregular with probability one. Thus, smoothness of $E[u(x,t)]$ does not imply smoothness of $u(x,t)$, and so in general the regularity of $E[u(x,t)]$ is not enough to hope for regularity of $u(x,t)$ itself. However, it is important to observe that this example started from an \emph{irregular} initial data, and that this singularity was preserved in time. Obviously, the same reasoning applies to see that in fact every solution to~\eqref{example_A_0_strat} is H\"older continuous in $D_T$ if $u_0$ is additionally assumed to be H\"older continuous, so in particular in the case $u_0 \in W^{1,q}(D,\R^N)$ for some $q > n$ (which was always required for the regularity statements before).
\end{remark}

Now let us come back to weak solutions to the general linear system~\eqref{linear SPDE} with Stratonovich noise. The crucial observation is that the average of $u$ solves an equation with improved parabolicity. Let us first recall the classical definition used also
before in this paper. If $v_{0}\in L_{loc}^{2}(D,\mathbb{R}^{N})
$, we say that a (deterministic) function $v(x,t)$ is a weak
solution of the parabolic equation
\begin{equation}
\frac{\partial v}{\partial t} = \operatorname{div} \Big(
\Big( A(x,t)  + \frac{\sigma^{2}}{2} \Big)  Dv \Big) \,,
\qquad v|_{t=0}=v_{0}  \label{parabolic determ}%
\end{equation}
if $v \in V^2_{\rm{loc}}(D_T,\R^N)$
(in the sense of (i) of Definition~\ref{definition solution stoch} above)
and
\begin{multline*}
\quad  \int_{D}v(x,t)  \varphi(x)  \dx+\int_{0}^{t}
\int_{D}A(x,s)  Dv(x,s)  D\varphi(x)
\dx \ds \\
  =\int_{D}v_{0}(x)  \varphi(x)  \dx - \frac
{\sigma^{2}}{2}\int_{0}^{t}\int_{D}Dv(x,s)  D\varphi(x) \dx \ds \quad
\end{multline*}
for all $\varphi\in C_{0}^{\infty}(D,\mathbb{R}^{N})$.

\begin{proposition}
\label{stoch to det}\bigskip If $u$ is a weak solution of equation
\eqref{linear SPDE}, then
\[
v(x,t)  :=E[u(x,t)]
\]
is a weak solution of the parabolic equation~\eqref{parabolic determ}.
\end{proposition}

\begin{proof}
\textbf{Step~1}. We first observe that for these linear systems, we have the a~priori boundedness of the solution $u$ in the sense that 
\begin{equation}
\sup_{t\in [0,T]  } E \Big[  \int_{B} \kabs{ u(x,t)}^{2} \dx \Big]  
+E \Big[  \int_{0}^{T} \int_{B} \kabs{ Du (x,t) }^{2} \dx \dt \Big] < \infty
\label{bounds on u}
\end{equation}
for all bounded sets $B \subset D$. The proof of this property follows the line of arguments of the proof of Lemma~\ref{lemma_apriori} (but is in fact much easier). We do not want to go into details, but only mention the peculiarities. First, by the linear structure in~\eqref{linear SPDE} the function $G_0$ appearing in Lemma~\ref{lemma_apriori} can be chosen constant. This explains, why the estimate~\eqref{bounds on u} doesn't involve weights as before. Furthermore, in Lemma~\ref{lemma_apriori} we were content with a bound for the expected value of the spatial derivatives of $u$ only. However, adjusting the arguments from Step 3b in the proof of Lemma~\ref{Diff-quotients}, we obtain a bound for the average for the full $V^2$-norm for every bounded set $B$ compactly supported in $D$. This immediately gives~\eqref{bounds on u}.

\textbf{Step~2}. The regularity property $v\in L^{\infty} (
0,T;L_{loc}^{2} (  D,\mathbb{R}^{N} ) )$ is a direct
consequence of the first condition in~\eqref{bounds on u} from Step~1. In order to prove that also $v \in L^{2} (0,T;W^{1,2}_{loc} (D,\mathbb{R}^{N} ) )$ holds true, we first 
observe that we have (a.\,s. in $t$)
\[
E \Big[  \int_{D} D_{i}u^{\alpha} (x,s) \, \psi(
x)  \dx \Big]  =-E \Big[  \int_{D}u^{\alpha}(x,s)
\, D_{i}\psi(x)  \dx \Big]
\]
for all $\psi\in C_{0}^{\infty}(D,\mathbb{R})$. This implies
(by the integrability derived in~\eqref{bounds on u})
\[
\int_{D}E \big[ D_{i}u^{\alpha}(x,s) \big]
\, \psi (x)  \dx=-\int_{D}E \big[ u^{\alpha}(x,s)
\big] \, D_{i} \psi(x)  \dx \,,
\]
which in turn gives us that $E[  u^{\alpha} (x,s)]$ is
weakly differentiable in $x$ with partial derivative equal to $E [D_{i} u^{\alpha
}( x,s )]$. Thus, $v$ is weakly differentiable in $x$ and
\begin{align*}
\int_{0}^{T}\int_{B} \kabs{ Dv (x,t) }^{2} \dx \dt  &
=\int_{0}^{T}\int_{B} \kabs{ DE [  u (x,t)] }^{2} \dx \dt
=\int_{0}^{T}\int_{B} \kabs{ E [  Du(x,t) ] }^{2} \dx \dt\\
& \leq \int_{0}^{T} \int_{B} E\big[  \kabs{ Du (  x,t ) }^{2} \big]  \dx \dt
=E \Big[  \int_{0}^{T}\int_{B} \kabs{ Du (x,t) }^{2} \dx \dt \Big]  < \infty
\end{align*}
for all bounded $B\subset D$. The regularity properties of $v$ have been
checked. 

\textbf{Step~3}. The property $E[\int_{0}^{T}\int_{B} \kabs{ u(
x,t)}^{2} \dx \dt]<\infty$ implies that the It\^{o} integral in
the equation of Proposition~\ref{proposition Stratonovich Ito} is a
martingale, hence it has zero expected value. By the same assumption, we can
interchange expectation and integrals, and we get
\begin{multline*}
\quad  \int_{D}v(x,t) \, \varphi(x) \dx + E \Big[  \int
_{0}^{t}\int_{D}A(x,s) \,  Du(x,s) \, D\varphi(x) \dx \ds \Big]  \\
 =\int_{D}u_{0}(x) \, \varphi(x) \dx+\frac{\sigma^{2}}{2} 
\int_{0}^{t}\int_{D}v(x,s) \, \Delta\varphi(x) \dx \ds \,. \quad
\end{multline*}
From the property $E[\int_{0}^{T}\int_{B} \kabs{ Du(x,t)} \dx \dt ]<\infty$ and the boundedness of $A$ it follows that
\[
E \Big[  \int_{0}^{t}\int_{D} A(x,s) \, Du(x,s)
\, D\varphi(x) \dx \ds \Big]  =\int_{0}^{t}\int_{D}A(
x,s) \,  E[ Du(x,s)] \,  D\varphi(x)
\dx  \ds.
\]
Since we know from Step~2 that $E[Du(x,s)]=Dv(x,s)$, the proof is complete. 
\end{proof}

Let us now explain the possibly regularizing effect of noise.
Assume $\sigma=0$. Then, as already explained in the introduction, weak solutions 
may miss full regularity. One can find in~\cite{STAJOH95} an example of matrix $A$ satisfying assumption~\eqref{linear assumption} and an example of a weak solution to the associated parabolic system~\eqref{system-intro-linear}
such that it is H\"{o}lder continuous on a local time interval and then its 
$L^{\infty}$ norm blows-up. More precisely, this matrix turns out to have an 
ellipticity ratio $\frac{\lambda_{1}}{\lambda_{0}}$ which is smaller than the critical 
one employed before, which was an essential ingredient in order to obtain globally H\"older continuous weak solutions (see~\cite{KOSHELEV93,KALITA94,KOSHELEV95}). However, the matrix constructed by Star\'a and John~\cite{STAJOH95} also fails to satisfy the regularity with respect to $x$, i.\,e. the matrix $A$ is not differentiable in $x$. 
For this reason it is not clear whether the counterexample could by constructed  
due to the small ellipticity ratio or the low regularity in $x$ or a combination of both. As far as we know there is no counterexample available in the literature which answers this question, and so even in the deterministic setting this irregularity phenomenon for weak solutions of parabolic systems is not understood completely. Instead, for the elliptic (stationary) case Koshelev was able to give a sharp result, namely that (in the linear case considered in this section) full H\"older continuity of the weak solution to $\diverg (A(x) Du) = 0$ holds provided that the matrix $A$ is symmetric (for simplicity), measurable, bounded, and  satisfies~\eqref{linear assumption} with
\begin{equation*}
\frac{\lambda_1 - \lambda_0}{\lambda_1 + \lambda_0} \sqrt{1 +
\frac{(n-2)^2}{n-1}} < 1 \,.
\end{equation*}
The sharpness of this condition follows by a modification of De Giorgi's famous counterexample~\cite{DEGIORGI68}, see~\cite[Section~2.5]{KOSHELEV95}. Returning to the parabolic setting we now state a consequence from the previous Proposition~\ref{stoch to det}, which for randomly perturbed systems~\eqref{linear SPDE} gives a regularity result for the average $E[u(x,t)]$ if the matrix is assumed to be regular with respect to $x$. Since the existence of a deterministic counterexample is not clear, the Stratonovich multiplicative noise is only possibly regularizing, but in any case it might be of 
its own interest since the noise improves the parabolicity of the equation solved 
by the average. 

\begin{proposition}
Assume $q>n$, $A$ with property~\eqref{linear assumption} such that $\kabs{D_x A}$ is bounded uniformly by some constant $L>0$, and let $D \subset \R^n$ be a bounded, regular domain. Then there exists $\sigma_{0}\geq0$ such that for all $\sigma>\sigma_{0}$, all initial conditions $u_0 \in W^{1,q}(D,\R^N)$, and all weak solutions $u$ of equation ~\eqref{linear SPDE} satisfying~\eqref{bounds on u}, we have that $(x,t) \longmapsto E[u(x,t)]$ is locally H\"{o}lder continuous on $D \times [0,T]$. One can take $\sigma_{0}$ depending only on $\frac{\lambda_{1}}{\lambda_{0}}$.
\end{proposition}

\begin{proof}
The matrix $A(x,t) +\frac{\sigma^{2}}{2}I$ satisfies the
assumptions needed for the deterministic regularity results in~\cite{KOSHELEV93,KALITA94,KOSHELEV95}, see also Theorem~\ref{thm_Kalita}, for all 
$\sigma$ greater than some $\sigma_{0}$ which can be defined in terms of $\frac{\lambda_{1}}{\lambda_{0}}$. This implies that any weak solution $v$ of equation
\eqref{parabolic determ} is locally H\"{o}lder continuous on $D \times [
0,T]$. It is sufficient to apply this result to $v(x,t)=E[u(x,t)]$. 
\end{proof}

\begin{remark}
Given the ratio $\frac{\lambda_{1}}{\lambda_{0}}$, the result is true for all
matrices $A$ with that ratio and all (regular) initial conditions. Intuitively speaking
it looks impossible that regularization comes from the operation of
mathematical expectation: it could regularize problems with special
symmetries, such that singularities for different $\omega$'s average out (compare Remark~\ref{rem_example_irregular} for this phenomenon under an irregular initial condition). But here $A$ and $u_{0}$ are quite generic (though regular). 
Thus we believe that H\"{o}lder regularization takes place at the level of $u$ itself. However, this problem is open.
\end{remark}

\end{document}